\documentclass[a4paper, 11pt]{article}

\usepackage[T1]{fontenc}
\usepackage[utf8]{inputenc}
\usepackage{url}

\usepackage[backref=page,colorlinks,bookmarksopen=true]{hyperref}
\usepackage{amsfonts,amsmath,amssymb}
\usepackage{amsthm}
\usepackage{enumerate}
\usepackage{amsrefs}
\usepackage{comment}

\usepackage{array}

\usepackage{imakeidx} % création d'index
\makeindex[columns=1,title=List of notations, intoc]

\usepackage{mathrsfs, xspace}
\usepackage{stmaryrd}
\usepackage{a4wide}
\usepackage[all]{xy}

\usepackage{pgf,tikz}
\usetikzlibrary{arrows}
\usepackage{bbold}

\theoremstyle{plain}
\newtheorem{proposition}{Proposition}[section]
\newtheorem{theorem}[proposition]{Theorem}

\newtheorem{lemma}[proposition]{Lemma}
\newtheorem{corollary}[proposition]{Corollary}

\newtheorem*{claim*}{Claim}

\theoremstyle{definition}
\newtheorem{definition}[proposition]{Definition}

\newtheorem*{convention}{Convention}

\theoremstyle{remark}
\newtheorem{remark}[proposition]{Remark}
\newtheorem{example}[proposition]{Example}

\DeclareMathOperator{\Conv}{Conv}
\DeclareMathOperator{\Aut}{Aut}

\DeclareMathOperator{\Ker}{Ker}
\DeclareMathOperator{\proj}{proj}

\newcommand{\tld}{\widetilde}
\DeclareMathOperator{\Ch}{Ch}

\DeclareMathOperator{\Res}{Res}

\DeclareMathOperator{\GL}{GL}
\DeclareMathOperator{\SL}{SL}

\DeclareMathOperator{\PGL}{PGL}

\DeclareMathOperator{\Isom}{Isom}

\DeclareMathOperator{\pr}{pr}

\newcommand{\RR}{\mathbf{R}}                          % |R
\newcommand{\CC}{\mathbf{C}}                          % |C
\newcommand{\NN}{\mathbf{N}}                          % |N
\newcommand{\ZZ}{\mathbf{Z}}                          % /Z
     
\DeclareMathOperator{\Prob}{Prob}
\renewcommand{\Im}{\mathrm{Im}}
\DeclareMathOperator{\lk}{lk}

\newcommand{\calB}{\mathcal{B}}

\newcommand{\scrF}{\mathscr F}
\newcommand{\scrW}{\mathscr W}
\newcommand{\op}{\mathrm{op}}

\newcommand{\pmop}{ {\pm,\mathrm{op}}}

\newcommand{\EE}{\mathbf E}
\newcommand{\IR}{\mathcal I}
\newcommand{\JR}{\mathcal J}

\newcommand{\newcomment}[4]{%
\newcounter{#2counter}
\expandafter\newcommand\csname #1\endcsname[1]{%
\refstepcounter{#2counter}%
{\color{#4}(#3\arabic{#2counter})}\marginpar{\scriptsize\raggedright\textbf{\color{#4}(#2 \arabic{#2counter}):} ##1}%
}}

\begin{document}

\title{Normal Subgroup Theorem for groups acting on $\tilde A_2$-buildings}%\\

\author{Uri Bader, Alex Furman, Jean L\'ecureux}
\maketitle

\begin{abstract}
    Let $\Gamma$ be a group acting with finite stabilizers and finite fundamental domain on a building of type $\tilde A_2$. We prove that any non-trivial normal subgroup of $\Gamma$ is of finite index in $\Gamma$.
\end{abstract}

\tableofcontents

\section*{Introduction}

The main examples of affine buildings come from Bruhat-Tits theory: if $k$ is a local field, and $\mathbf G$ a $k$-simple algebraic group, then Bruhat and Tits associated to $\mathbf G$ an affine building on which $\mathbf G(k)$ acts. In particular, cocompact lattices in $\mathbf G(k)$ also act on $X$, properly and cocompactly. In higher rank, such lattices are all arithmetic, by  the celebrated work of Margulis and Venkataramana \cites{Margulis,Venka}. In particular, their algebraic structure is fairly well understood. For example, one can produce many finite index normal subgroups by taking congruence kernels. To clarify the terminology, in the absence of an ambient algebraic group, we say that a group $\Gamma$ acting on a simplicial complex $X$ is a \emph{cocompact} (or uniform) \emph{lattice} in $X$  if it acts with finite stabilizer and a compact fundamental domain.

A fundamental result of Tits (see the last corollary of \cite{Tits_cours84} as well as the book \cite{WeissBook}) is that all irreducible affine buildings of dimension at least 3 are Bruhat-Tits buildings. On the other hand, buildings of dimension 1 are just trees, and their cocompact lattices are virtually free.

The remaining case of buildings of dimension 2 turns out to be quite interesting. First, there are reducible buildings of dimension 2, namely, products of trees. In that case, the celebrated work of Burger and Mozes \cite{BurgerMozes_lattices} produces examples of lattices in products of two trees which are simple (and in particular, non-linear). Then, there are irreducible buildings of dimension 2, of three possible types ($\tilde A_2$, $\tilde C_2$ or $\tilde G_2$), buildings of type $\tilde A_2$ being the simplest and most studied examples. 
In this setting, it is already nontrivial to produce examples with cocompact lattices which are not Bruhat-Tits buildings. In the sequel, we will call these buildings \emph{exotic}.
The first construction appeared in \cite{Ronan_triangle} and \cite{Tits_Andrews}, where Ronan and Tits managed to construct the first examples of such buildings. There have been many other constructions since then. We refer to \cite[\S1.3]{BCL} for a more thorough treatment of the literature. In particular, it has been proved that there exist exotic lattices in 2-dimensional buildings of arbitrarily large thickness; in some cases (panel-regular lattices) it has even been proven that most buildings are not Bruhat-Tits (this was conjectured in \cite{Witzel} and proved in \cite{Radu}). 

Another potential source of examples are \emph{Galois lattices}: starting from a Bruhat-Tits building $X$, associated to an algebraic group $\mathbf G(k)$, the full automorphism group of $X$ is not $\mathbf G(k)$, but an extension of $\Aut(\mathbf G)(k)$ by a subgroup of $\Aut(k)$, which is infinite in positive characteristic (see \cite{BorelTits} and \cite[Proposition 9.1]{BCL}). Therefore there may exist a lattice $\Gamma$ in $\Aut(X)$ which has infinite image in $\Aut(k)$: this is what we call a Galois lattice, and it may or may not be cocompact. So far it is not known whether Galois lattices exist or not in higher rank.

These "exotic" buildings of type $\tilde A_2$ and their lattices remain largely mysterious, although they have been studied by many authors. For example, there has been a large amount of work devoted to proving property (T) and its variations, starting with the paper of Cartwright-M\l otkowski-Steger \cite{CartMlotkT}, then followed by Pansu \cite{Pansu} and \.{Z}uk's criterion \cite{Zuk} in increasing generality and in particular for any cocompact lattice, with no thickness assumption. Using a different approach it was generalized to arbitrary affine buildings by I. Oppenheim \cite{oppenheim}. Recently it was also proved that these lattices also have Lafforgue's strong property (T) \cite{StrongT}. 

In \cite{BCL} the authors started to study these lattices with a different point of view, namely, by adapting and applying Margulis' methods about lattices in higher rank simple Lie groups. It was proven there that cocompact lattices in exotic buildings, as well as Galois lattices, are fairly different from their arithmetic counterparts: they are always non-linear. This paper can be seen as a continuation of this theme. Indeed, we are able to prove the Normal Subgroup Theorem:

\begin{theorem}\label{mainthm}
Let $\Gamma$ be a cocompact lattice of a building of type $\tilde A_2$. Then any non-trivial normal subgroup of $\Gamma$ is of finite index.
\end{theorem}

In the case when the building is Bruhat-Tits, and $\Gamma<\mathbf G(k)$, this theorem is due to Margulis \cite{Margulis}. Therefore our main contribution is about exotic buildings, and possibly also to Galois lattices.  
This theorem was announced by Shalom and Steger in 2008, but the proofs have not appeared in print, and we do not know how large is the overlap between our approach and theirs.
 Our motivation for writing this proof is two-fold. First, it is a first step towards proving stronger results: the Stuck-Zimmer theorem \cite{StuckZimmer} or character rigidity, in the spirit of \cite{Peterson}.

Another generalization of Margulis' Normal subgroup Theorem in a more geometric setting (for lattices in products) has been obtained by Bader and Shalom \cite{BaderShalom}, greatly generalizing previous work of Burger and Mozes \cite{BurgerMozes_lattices}. Their proof techniques are quite different, but in both cases they led to the proof of simplicity of some groups (either lattices in product of trees or Kac-Moody groups \cite{CapraceRemy}). 
Extending the parallelism with these situations, our second motivation stems from the following corollary.

\begin{corollary}
    Let $\Gamma$ be a cocompact lattice of a building of type $\tilde A_2$. Then either $\Gamma$ is virtually simple, or it is residually finite.
\end{corollary}

\begin{proof}
    If $\Gamma$ is not residually finite, then the intersection $N$ of all finite-index subgroups of $\Gamma$ is not trivial. But $N$ is a normal subgroup of $\Gamma$, so it is of finite index, and therefore is also a cocompact lattice of the same building. Applying Theorem \ref{mainthm} to $N$, we get that every non-trivial normal subgroup of $N$ is of finite index in $N$, and therefore also in $\Gamma$, so that it contains $N$. Hence $N$ is a finite index simple subgroup of $\Gamma$.
\end{proof}

If $\Gamma$ is a lattice in an algebraic group over a local field, then (as mentioned above) it is clear that $\Gamma$ is residually finite. On the other hand, if $\Gamma$ is a cocompact lattice in an exotic building, then the situation is very much unclear. It was conjectured in \cite[Conjecture 1.5]{BCL} that the first conclusion always holds ($\Gamma$ is virtually simple), but so far there has been very little progress in that direction. However, recently, Thomas Titz Mite and Stefan Witzel have been able to produce cocompact lattices in buildings of type $\tilde C_2$ which are not residually finite \cite{MiteWitzel}. Although the proof needs to be adapted in non-trivial ways, it is very likely that the analogue of Theorem \ref{mainthm} also holds for $\tilde C_2$ buildings, thus providing new  examples of infinite simple groups. We note that these building lattices are never quasi-isometric if they act on different buildings (see for example \cite{KramerWeiss}), so that it would provide infinitely many different QI classes of finitely presented groups. On the other hand, rigidity in Measure Equivalence for these groups is not known at the moment, although it probably should be expected.

\medskip

The main steps of the proof of Theorem \ref{mainthm} follow the beautiful strategy of Margulis. If $N$ is a non-trivial normal subgroup of $\Gamma$, then we first prove that $\Gamma/N$ has property (T) (we get this for free as $\Gamma$ itself has property (T)), and then that $\Gamma/N$ is amenable, which is the hard part. To do that, we introduce the boundary $\Delta$ of the building, which is an amenable $\Gamma$-space. The space $\Delta$ is the set of chambers of a spherical building (of rank 2, hence of dimension 1, \textit{i.e} a graph), and is equipped with two $\Gamma$-equivariant maps $\Delta\to \Delta_+$ and $\Delta\to \Delta_-$, where $\Delta_+$ and $\Delta_-$ are the two sets of vertices of different types. The space $\Delta$ is endowed with a natural measure class, which is $\Aut(X)$-invariant. As in Margulis' proof, the amenability of $\Gamma/N$ will follow from a Factor Theorem. In our case this Factor Theorem translates as follows.

\begin{theorem}\label{thm:factor}
    Any $\Gamma$-equivariant measurable factor $\Delta\to Z$ is (measurably) isomorphic to either $\Delta\to \Delta_+$, $\Delta\to \Delta_-$, the identity map or a point.
\end{theorem}

The proof of this Theorem is our main contribution. The general strategy to prove this theorem is to find enough elements in $\Gamma$ with some kind of contracting properties: namely, contracting almost all of $\Delta$ to the set of chambers adjacent to a given vertex $u\in\Delta_-$ (the \emph{residue} $\Res(u)$ of $u$). In fact, what we prove is that one can approximate a special subgroup of permutations $\Res(u)$ called the \emph{projectivity group}. This group is a classical object of study in incidence geometries and possesses a very rich structure, even when the building has few automorphisms. In particular, it is always 3-transitive on $\Res(u)$. 

In the classical case when $\Delta$ is the spherical building associated to $G=\SL_3(K)$ (for a field $K$), then the projectivity group is simply $\PGL_2(K)$, and its action on $\Res(u)$ is conjugated to the natural action of $\PGL_2(K)$ on the projective line. This fact, while not completely obvious, can be checked with relatively simple calculations (see \cite{Knarr} for corresponding statements in Moufang 2-dimensional buildings). In general for buildings associated to algebraic groups the projectivity group is closely related to the Levi subgroup of the parabolic group associated to the facet: see for example the recent preprint \cite{Busch}. However, we would like to emphasize that for non-Moufang buildings the projectivity group is typically very large, and does not act in any way on the building itself. This surprising fact is the heart of the proof of the main theorem of \cite{BCL} (see \cite[Theorem 4.7]{BCL}).

The statement of Theorem \ref{thm:factor} refers to a measure class on $\Delta$, which is naturally defined. Note also that the strategy sketched in the previous paragraph naturally gives  some topological convergence, and therefore it requires some more work to get a statement about the measure class. In the classical case, this is obtained through a non-trivial version of the Lebesgue Density Theorem \cite[IV.1]{Margulis}. Furthermore, also in the classical case, the proof uses various actions on $G$-homogeneous spaces, with their measured structure.

In our setting, these homogeneous spaces will be replaced by some natural sets of embeddings: for example, the set of embeddings of an apartment into the building. It is not obvious in general how to construct natural measures with good properties on such spaces. In order to do so, we were led to introduce a fairly general formalism which explains how and when one can construct \emph{prouniform} measures on the set of embeddings of a simplicial complex into another, under some condition that we call \emph{P-symmetricity}. We then check that this condition holds for all the spaces that we need. Finally, we found that it was more suited to our setting to use the Martingale Convergence Theorem rather in place of the Lebesgue Density Theorem. Furthermore, in order to attach the kind of convergence we obtain to the group action, we were led to introduce yet another space, that we call the \emph{detecting flow}, which can be thought of as a kind of geodesic flow (or Cartan flow in our case) with additionnal "antennas".

%\textbf{About other types}

The present paper is written only for buildings of type $\tilde A_2$. We have no doubt that the results hold for more general irreducible buildings, and in fact this is currently being written by the third author with S.~Witzel. The core of the argument is in general the same, however, there are a few places where we do use the assumption that we have $\tilde A_2$ buildings. The first one is the construction of the singular flow, that we take directly from \cite{BCL}, which was written only in $\tilde A_2$ buildings. This is only a small issue, as the construction applies verbatim in the general case. We also use at a couple places (notably in Theorem \ref{thm:2minimal}) that in a spherical building of type $A_2$, two distinct vertices of the same type are opposite.  This provides mostly a shortcut and is not essential to the arguments. The most important reason why we were not able to treat the general case in one go is the more delicate construction of the detecting flow, which we introduce in Section \ref{sec:detecting}. This construction relies on some elementary constructions in convex geometry described in \S \ref{sec:crazydiamonds}, that do not generalize easily to other types.

\medskip
\textbf{Structure of the paper.} The first section (after this introduction) explains the construction of prouniform measures on sets of embeddings between simplicial complexes, and can be read completely independently of the paper. The second section describes quickly the background on buildings which is necessary to us, and then in \S 3 we check that the assumptions needed in \S1 are indeed true in the setting we want. It is also the occasion to define the various spaces and flows that we will use in the course of the proof. The proof itself starts in Section 4, where we explain how to produce the elements of $\Gamma$ with the required dynamical properties. Finally, in \S5, we conclude the proof of Theorem \ref{thm:factor} and \ref{mainthm}.

In this work we were led to introduce many different spaces. To help the reader keep track of everything, we also included a list of notations at the very end of the paper.

\medskip
\textbf{Acknowledgment:} We are thankful to Pierre-Emmanuel Caprace and Stefan Witzel for various discussions and encouragements, and to Stefan Witzel and Antoine Derimay for spotting some mistakes in an earlier version of this text. We are grateful to the referee for his or her comments, which led to many improvements of this text.

\section{Spaces of simplicial maps as measured spaces}\label{sec:measure}

In this section, our goal is to define "uniform" measures on spaces of embeddings of a simplicial complex into another. This will require some symmetricity conditions that we define below. To define these measures, we first treat the case when the simplicial complex at the source is finite, and then use some projective limit to treat the general case.

\subsection{Prouniform measures}

We consider here connected simplicial complexes.
In this section, for a connected simplicial complex $Y$ we denote by $Y^k$ the set of its $k$-dimensional simplices.
In particular, $Y^0$ is its set of vertices.
$Y^0$ is endowed with the combinatorial metric given by the 1-skeleton, that is the graph $(Y^0,Y^1)$.
For two connected simplical complexes $Y$ and $Y'$,
by isometry from $Y'$ to $Y$ we mean a simplicial map $i$ which induces an isometry $Y'^0\to Y^0$
and such that every simplex in $Y$ whose boundary is in the image of $Y'$ is itself in the image of $Y'$, and also such that every geodesic segment between vertices in $i(Y'^0)$ is entirely contained in $i(Y'^0)$.
We denote by $\text{Isom}(Y',Y)$ the set of all such isometries.
If $Y'$ is a subsimplicial complex of $Y$ and the inclusion is an isometry we say that $Y'$ is convex in $Y$.

Throughout this section we fix a locally finite connected simplical complex $X$.
Given any finite simplicial complex $Y$, we endow the countable set $(X^0)^{Y^0}$ with the counting measure and let $m_Y$ be its restriction to $\text{Isom}(Y,X)$.
Typically $(\text{Isom}(Y,X),m_Y)$ is an infinite measure space, but if $Y$ is connected, $Y'$ is nonempty and $i:Y'\to Y$ is an isometry,
then the map
$i^*:(\text{Isom}(Y,X),m_Y)\to (\text{Isom}(Y',X),m_{Y'})$ given by $\beta\mapsto \beta\circ i$ is a finite measure extension, that is
there exists a disintegration of $m_Y$ with respect to $m_{Y'}$ with finite measured fibers.
Indeed, given an isometry $\alpha:Y'\to X$, 
the set
\[ \{\beta:Y\to X\mid~\beta \mbox{ is an isometry and } \alpha=\beta\circ i\}, \]
is finite, as $Y$ is connected and $X$ is locally finite,
and for every $U\subset \text{Isom}(Y,X)$,
\[ m_Y(U)=\int_{\text{Isom}(Y',X)} 
|\{\beta\in U\mid~\beta \mbox{ is an isometry and } \alpha=\beta\circ i\}|
dm_{Y'}(\alpha). \]

\begin{definition}
Given a property of connected finite simplicial complexes $P$,
we say that $X$ is $P$-symmetric if it satisfies the following property:
\begin{itemize}
\item
For every pair of nonempty finite connected $P$ simplicial complexes $Y',Y$ and an isometry $i:Y' \to Y$, given an isometry $\alpha:Y'\to X$ the size of the set
\[ \{\beta:Y\to X\mid~\beta \mbox{ is an isometry and } \alpha=\beta\circ i\} \]
is positive and independent of $\alpha$.
\end{itemize}
In case $X$ is $P$-symmetric and $Y,Y'$ and $i$ are as above,
we denote the size of the above set by $[Y:Y']_i$ or by $[Y:Y']$, in case $Y'\subset Y$ is a convex subset and $i$ is the inclusion.
\end{definition}

\begin{example}
Let $P$ be the property {\em tree of degrees bounded by $d$}.
Then a $d$-regular tree is $P$-symmetric.
\end{example}

We fix for the rest of the section a property of finite connected simplicial complexes $P$.
We assume $P$ is such that finite segments (isomorphic to a convex subset of $\ZZ$), and in particular points, satisfy $P$. We also assume that $P$ is stable under taking convex subcomplexes. 

We also assume that our fixed simplical complex $X$ is $P$-symmetric.

Note that for a pair of finite connected simplicial complexes $Y',Y$ both having $P$ and an isometry $i:Y' \to Y$,
the push forward of $m_Y$ under the restriction map $i^*:\text{Isom}(Y,X)\to \text{Isom}(Y',X)$ satisfies 
\begin{equation}\label{eq:i*}
(i^*)_*(m_Y)=[Y:Y']_i \cdot m_{Y'}.
\end{equation}

\begin{lemma}\label{lem:[XX']}
Let $Y$ be a finite connected simplicial complex satisfying $P$.
\begin{itemize}
\item
For any isometry $i:Y\to Y$, $[Y,Y]_i=1$.
\item 
For finite connected $P$ simplicial complexes $Y'$ and $Y''$ and isometries $i:Y''\to Y'$, $j:Y'\to Y$ one has
\[ [Y,Y'']_{j\circ i}=[Y,Y']_j\cdot [Y',Y'']_i. \]
\item
For every pair of vertices $y,y'\in Y^0$, $[Y,\{y\}]=[Y,\{y'\}]$.
\end{itemize}
\end{lemma}

\begin{proof}
The proof of the first two claims is straight forward and left to the reader.
In order to see the third one, we first assume that there exists $I\subset Y$  a geodesic segment with end points $y,y'$
and let $\tau:I\to I$ be the unique non-trivial isometry.
Then, using the first two claims, we have $[I,\{y\}]=1\cdot [I,\{y\}]=[I,I]_\tau \cdot [I,\{y\}]=[I,\{y'\}]$, thus
\[ [Y,\{y\}]=[Y,I] \cdot [I,\{y\}]= [Y,I] \cdot [I,\{y'\}]=[Y,\{y'\}]. \]
The result follows by convexity of $Y$.
\end{proof}

In view of the above lemma, we use the notation $[Y]$ to denote the common value $[Y,\{y\}]$ for $y\in Y_0$.
We observe that for a pair of finite connected simplicial complexes $Y',Y$ both having $P$ and an isometry $i:Y' \to Y$
we have 
\begin{equation} \label{eq:X}
[Y]=[Y,Y']_i\cdot [Y'].
\end{equation}

Fix an origin $y_0\in Y^0$ and $o\in X^0$. We define a probability measure $\mu_Y^o$ on $\Isom(Y,X)$ by declaring it to be the uniform probability  on the (finite) set $\Isom(Y,X)^o:=\{\beta\in \Isom(Y,X)\mid \beta(y_0)=o\}$.

More generally, if $i:Y'\to Y$ is an isometry, and $\alpha:Y'\to X$ is a given isometry, we write $\mu_{Y,Y'}^\alpha$ to be the uniform probability on  $\Isom(Y,X)^\alpha:=\{\beta\in \Isom(Y,X)\mid \beta\circ i = \alpha \}$. If $Y'=\{y_0\}$ of course we recover the same construction.

For every finite connected $P$ simplicial complex $Y$ we define the measure $\mu_Y$ on $\text{Isom}(Y,X)$ by  $\mu_Y=\sum\limits_{o\in X_0} \mu_Y^o$.

%For every finite connected $P$ simplicial complexes $X$ we define the measure $\mu_X$ on $\text{Isom}(X,Y)$ by 

\begin{lemma}
For every finite connected $P$ simplicial complex $Y$ we have
\begin{equation} \label{eq:mu}
[Y]\cdot \mu_Y=m_Y.
\end{equation}
and  for every pair of finite connected 
simplicial complexes $Y',Y$ both having $P$ and an isometry $i:Y' \to Y$
we have 
\begin{equation} \label{eq:i*mu}
(i^*)_*\mu_Y=\mu_{Y'}
\end{equation}

\end{lemma}

\begin{proof}
Since $\Isom(Y,X)$ is countable, it suffices to prove Equation \eqref{eq:mu} on singletons. Let $f\in \Isom(Y,X)$, by construction we have $\mu_Y(f)= \mu_Y^o (f)$ where $o=f(y_0)$, and this number is by definition $\frac{1}{[Y]}$. 

Equation \eqref{eq:i*mu} follows 
by combining Equations~(\ref{eq:i*}), (\ref{eq:X}) and (\ref{eq:mu}).
\end{proof}

The following lemma is a simple double counting argument :

\begin{lemma}\label{lem:finiterestricted}

Let $Y_0,Y_1,Y_2$ be finite connected simplicial complexes having $P$, with isometries $i_k:Y_k\to Y_{k+1}$, and $\alpha_0\in\Isom(Y_0,X)$.

Then

$$\mu_{Y_2,Y_0}^{\alpha_0} = \sum_{\alpha_1\in \Isom(Y_1,X)} \mu_{Y_1,Y_0}^{\alpha_0} (\alpha_1) \cdot \mu_{Y_2,Y_1}^{\alpha_1} $$

In particular, taking $Y_0$ to be a point we get 
$$\mu_{Y_2}^{o} = \sum_{\alpha_1\in \Isom(Y_1,X)} \mu_{Y_1}^{o} (\alpha_1)\cdot \mu_{Y_2,Y_1}^{\alpha_1} $$

\end{lemma}

\begin{proof}
Clearly it suffices to prove the first equation.
Ongce again it suffices to prove the Lemma on singletons. So let $f\in \Isom(Y_2,X)^\alpha_0$, and let $\alpha'_1 = f\circ i_1$. On the left-hand side, $\mu_{Y_2,Y_0}^{\alpha_0} (f) = \frac{1}{[Y_2:Y_0]}$. On the right-hand side, each term of the sum is $0$ except when $\alpha_1=\alpha'_1$, so that we get $\mu_{Y_2,Y_1}^{\alpha'_1}(f) \mu_{Y_1,Y_0}^{\alpha_0}(\alpha'_1)$. 

But $\mu_{Y_1,Y_0}^{\alpha_0}(\alpha'_1) = \frac{1}{[Y_1:Y_0]}$ and $\mu_{Y_2,Y_1}^{\alpha'_1}(f) = \frac{1}{[Y_2:Y_1]}$ by definition. The result follows from Lemma \ref{lem:[XX']}.
\end{proof}

\begin{definition}
Given a property of connected finite simplicial complexes $P$,
a connected simplicial complex is said to be ind-$P$ if it could be presented as an ascending union of finite convex subcomplexes having property $P$.
\end{definition}

Recall that we have fixed a simplicial complex $X$ which is $P$-symmetric. Now we let $Y$ be a connected ind-$P$ simplicial complex, and note that in the category of sets
\[ \text{Isom}(Y,X) = \varprojlim \text{Isom}(Y',X), \]
where $Y'$ runs over the directed set of finite, convex subcomplex of $Y$ having property $P$ (with the inclusion), and the map associated to an inclusion $i_{Y',Y''}:Y'\to Y''$ are the restrictions $i_{Y',Y''}^*:\Isom(Y'',X)\to \Isom(Y',X)$. 

By Equation \eqref{eq:i*mu}, the map $i_{Y',Y''}^*$ is measure-preserving. Therefore, using for example Caratheodory's Extension Theorem, there exists a unique measure $\mu_Y$ on $\Isom(Y,X)$ such that, for every finite convex subcomplex $Y'\subset Y$ with property $P$, the restriction map $i_{Y',Y}^*:\Isom(Y,X)\to \Isom(Y',X)$ satisfies $(i_{Y',Y}^*)_*\mu_Y=\mu_{Y'}$.

In other words, the measure $\mu_Y$ on $\text{Isom}(Y,X)$ is defined by setting
\[ (\text{Isom}(Y,X),\mu_Y) = \varprojlim (\text{Isom}(Y',X),\mu_{Y'}), \]
where $Y'$ runs over all finite, convex subcomplexes of $Y$ having property $P$ and the limit is taken in the category of measured spaces.

\begin{definition}
The measure $\mu_Y$ is called the \emph{prouniform measure} on $\Isom(Y,X)$.
\end{definition}

Note that $\mu_Y$ depends on both $X$ and $Y$ ; since $X$ has been fixed once and for all and in order to lighten the notations we prefer not to make it appear in $\mu_Y$.

More generally, for a pair of connected ind-$P$ simplicial complexes $W,Y$ and an isometry $i:W \to Y$, and $\alpha\in \Isom(W,X)$, we can construct a relative version $\mu_{Y,W}^\alpha$ as follows.
Let $W'$ be a connected  ind-$P$ simplicial complex, with isometries $i':W\to W'$, $j:W'\to Y$ such that $i=j\circ i'$. Assume that $i'(W)$ contains all but finitely many simplices of $W'$.

\begin{lemma}
The set $\Isom(W',X)^\alpha=\{\beta\in \Isom(W',X)\mid \alpha=\beta\circ i'\}$ is finite, and its size is independent of $\alpha$.
\end{lemma}

\begin{proof}
We identify below $W$ with the convex subset $i'(W)$ of $W'$.
Let $W_1$ be a finite convex subcomplex  of $W'$ containing all simplices of $W'\setminus W$ and satisfying $P$. Since $W'$ is ind-$P$, this always exists. Let $\alpha_1$ be the restriction of $\alpha$ to $W \cap W_1$ (which also satisfies $P$ by assumption). Then the number of extensions of $\alpha_1$ to an isometry $\beta_1$ of $W_1$ is positive, finite and independent of $\alpha$. Using $P$-symmetry again, one can extend $\beta_1$ in a unique way to an isometry $\beta'_1\in \Isom(W'_1,X)^\alpha$, for any $W'_1\subset W'$ finite, having P and containing $W_1$, and therefore to a unique isometry $\beta\in \Isom(W',X)^\alpha$.
\end{proof}

For $W,W'$ as above, 
we equip $\Isom(W',X)^\alpha$
 with the uniform measure, which we denote $\mu_{W',X}^\alpha$.

More generally, define $\Isom(Y,X)^\alpha = \{\beta\in \Isom(Y,X)\mid \alpha=\beta\circ i\}$. Observe that again, in the category of sets, we have
\[ \text{Isom}(Y,X)^\alpha= \varprojlim \text{Isom}(W',X)^\alpha, \]
where $W'$ runs over all cofinite complexes containing $W$ as a subcomplex and being ind-P. 
We use this observation in order to introduce a measure, $\mu_{Y,W}^\alpha $, on $\text{Isom}(Y,X)^\alpha$ by setting
\[ (\text{Isom}(Y,X)^\alpha,\mu_{Y,W}^\alpha) = \varprojlim (\text{Isom}(W',X),\mu_{Y,W'}^\alpha), \]

In particular, when $W$ is reduced to a point, we define this way the "restricted" measure $\mu_Y^o$ for every $o\in X_0$. Note that again we get
$$\mu_Y=\sum_{o\in X_0} \mu_Y^o$$

\begin{definition} 
The measures $\mu_Y^\alpha$ are called \emph{restricted prouniform measures}.
\end{definition}

\begin{lemma}\label{lem:desintegration_prouniform}
Let $Y_0,Y_1,Y_2$ be connected ind-$P$ simplicial complexes, with isometries $i_k:Y_k\to Y_{k+1}$, and $\alpha_0:Y_0\to X$ be an isometry. 

Then

$$\mu_{Y_2,Y_0}^{\alpha_0} = \int_{\Isom(Y_1,X)} \mu_{Y_2,Y_1}^{\alpha_1} d\mu_{Y_1,Y_0}^{\alpha_0} (\alpha_1)$$
\end{lemma}

\begin{proof}
It suffices to prove the equality against every locally constant function, and therefore against the characteristic function of a set of isometries determined by a finite convex subcomplex of $Y$. In that case this follows from Lemma \ref{lem:finiterestricted}.
\end{proof}

In particular if $Y_0$ is a point then 

$$\mu_{Y_2}^{o} = \int_{\Isom(Y_1,X)} \mu_{Y_2,Y_1}^{\alpha_1} d\mu_{Y_1}^{o} (\alpha_1)$$

Using the lemma we directly get the following.

\begin{corollary}\label{cor:measurepreserving}
Let $Y_0,Y_1,Y_2$ be connected ind-$P$ simplicial complexes, with isometries $i_k:Y_k\to Y_{k+1}$, and $\alpha_0:Y_0\to X$ be an isometry. 

Then $(i_1^*)_*\mu_{Y_2,Y_0}^{\alpha_0} = \mu_{Y_1,Y_0}^{\alpha_0}$
\end{corollary}

Again in the special case when $Y_0$ is a point we have that $(i_1^*)_*\mu_{Y_2}^o = \mu_{Y_1}^o$.

\begin{example}
Let $X$ be a regular tree, and $Y$ be a half-line $[0,+\infty)$. Then the set of embeddings of $Y$ into $X$ which send $0$ to $o$ can be identified to the boundary $\partial_\infty T$ ; the measure $\mu_Y^o$ is then the usual harmonic measure on the tree.
\end{example}

\begin{example}
Let $X$ be a regular tree, $Y$ be a line and $Y'\subset Y$ be a half-line. Then $\Isom(Y,X)$ is the "geodesic flow" on the tree, and $\mu_Y$ is the well-known Bowen-Margulis measure on this space. And as we saw before, $\Isom(Y',X)$ can be identified to the boundary of the tree. With this identification in mind, the restriction map is just the map which associates to a geodesic its endpoint. 
\end{example}
%
%%
%\subsection{Extensions}
%
%We keep the same setting as above, and we will consider a slight generalization of the same construction. Fix a  pair of connected ind-$P$ simplicial complexes $W,X$ and an isometry $i:W \to X$. For $g\in \Isom(W,Y)$ let $\Isom(X,Y)^\alpha =\{\beta\in \Isom(X,Y)\mid \beta \circ i = \alpha\}$.
%
%As above, we endow $\Isom(W,Y)$ (resp. $\Isom(X,Y)$) with the measure $\mu_W$ (resp. $\mu_X$). If $W'$ is a connected  ind-$P$ simplicial complex with isometries $i_1:W\to W'$ and $i_2:W'\to X$ with $i=i_2\circ i_1$, by assumption, the number of $\beta \in \Isom(W',Y)^\alpha$ is independent of $\alpha$ and equal to $[W':W]$.

\subsection{Pushforwards and martingales}

Finally, let us end this section by some statements about spaces of measures on prouniform spaces. Let $Y,Y'$ be an ind-$P$ simplicial complexes, and $i:Y'\to Y$ be an inclusion. Then the restriction $i_*$, denoted $r$ to avoid to write down too many stars, is a map from $\Isom(Y,X)\to \Isom(Y',X)$. Again, to avoid cumbersome notations, write $I_Y=\Isom(Y,X)$. Then we get a map $r_*:L^1(I_Y,\mu_Y)\to L^1(I_{Y'},\mu_{Y'})$, which is just defined as the pushforward of a finite measure. More concretely, the map $r_*$ is an average of $f$ along the fibers of $r$.

On the other hand, we also get a map $r^*:L^1(I_{Y'})\to L^1(I_Y)$ (or more generally for any $L^p$ space, for $1\leq p \leq +\infty$), defined by $r^* f(z) = f(r(z))$.

The dual of $L^1(I_Y,\mu_Y)$ is of course $L^\infty(I_Y,\mu_Y)$, with duality given by integration along $\mu_Y$, which we denote by $\langle\cdot, \cdot\rangle$. The restriction $r:I_Y\to I_{Y'}$ defines again a pullback $r^*:L^\infty(I_{Y'})\to L^\infty(I_Y)$. 

We first record the following  observation.

\begin{lemma}\label{lem:duality}
For every $\varphi\in L^1(Y,\mu_Y)$ and $f\in L^\infty(Y',\mu_{Y'})$ we have
$$\langle \varphi,r^* f \rangle = \langle r_*\varphi,f \rangle$$
\end{lemma}

\begin{proof}
Write the disintegration $d\mu_Y(z)=\mu_Y^{r(z)}d\mu_{Y'}(r(z))$, where for every $\alpha\in Y'$ the measure $\mu_Y^\alpha$ is supported on $r^{-1}(\alpha)$.

Then for every $\alpha\in Y'$ we get $r_*\varphi(\alpha)=\int \varphi(x) d\mu_{Y}^\alpha(x)$ and 
\begin{align*}
\langle r_*\varphi,f \rangle &= \int_{I_{Y'}}r_*\varphi(\alpha) f(\alpha) d\mu_{Y'}(\alpha)\\
&= \int_{I_{Y'}} f(\alpha)\int_{I_Y} \varphi(z) d\mu_Y^\alpha(z) \\
&=\int_{I_Y} f(r(z)) \varphi(z) d\mu_Y(z)\\
&= \langle \varphi,r^*f \rangle
\end{align*}

\end{proof}

Our final remark for this section is that our setup is well-suited for applications of the Martingale Convergence Theorem. 
Let $Y'$ be a convex ind-$P$ subset of $Y$ and let $(Y_k)_{k\in\NN}$ be an increasing sequence of finite convex subsets of $Y$ with property $P$, such that $Y=\bigcup_{k\in \NN} Y_k$.  
Let $r_k:I_Y\to I_{Y_k}$ be the restriction maps and denote by $\calB_k$ and $\calB_\infty$ the $\sigma$-algebras on $I_Y$ pulled back from the given $\sigma$-algebras on $I_{Y_k}$ and $I_{Y'}$ by the maps $r_k$ and $r$ correspondingly.
Then $\calB_\infty=\cup_k \calB_k$ and the Martingale Convergence Theorem reads as follows.

\begin{theorem}\label{thm:martingales}
 For every $f\in L^1(\Isom(Y,X))$, for almost every $i\in\Isom(Y,X)$ we have
 $$\lim_k \EE (f\mid \calB_k)(i) = \EE(f\mid \calB_\infty)(i)$$
    The convergence also holds in $L^1$.
\end{theorem}

Let us emphasize the following consequence.

\begin{lemma}\label{lem:densityL1_gen}
%Let $(Y_k)_{k\in\NN}$ be an increasing sequence of finite convex subsets of $Y$ with property $P$, such that $Y=\bigcup_{k\in \NN} Y_k$. Let $r_k:I_Y\to I_{Y_k}$ be the restriction map.
%Then the space $\bigcup_{k\in \NN} r_k^*L^1(I_{Y_k},\mu_{Y_k})$ is dense in $L^1(I_Y,\mu_Y)$ (in the norm topology).
%
%Similarly, if $Y'$ is a convex ind-$P$ subset of $Y$, and $\alpha\in \Isom(Y',X)$ is fixed, then  
For $\mu_{Y'}$-a.e $\alpha\in I_{Y'}$, the space $\bigcup_{k\in \NN} r_k^*L^1(I_{Y_k},\mu_{Y_k}^{r_k(\alpha)})$ is norm dense in $L^1(I_Y,\mu_Y^\alpha)$.
\end{lemma}

\begin{proof}
By the previous discussion, we get that for $\phi\in L^1(I_Y)$ the sequence $(r_k^*\phi)_{k\in\NN}$ is a martingale, bounded by $\Vert \phi\Vert$. Hence it converges (both a.s. and in $L^1$) to $\phi$.

\end{proof}

\section{Buildings of type $\tilde A_2$}\label{sec:buildings}

In this section we will explain the facts that we need about $\tilde A_2$ buildings.
 We assume some familiarity with the notion, and refer the reader to the reference book \cite{AbramenkoBrown} for a more thorough treatment. 
 Other references include \cite{WeissBook} or the paper \cite{BCL} which adopted a similar point of view and notations.
 Nevertheless, we will start by recalling a few basic facts and definitions, also in order to fix the vocabulary and notations. 
 Since we will introduce a large variety of notation, we made a list gathering them all by the end of the article. 

\subsection{Generalities}

Let $\Sigma$ \index{$\Sigma$ \hfill model apartment|bb}be the tessellation of the Euclidean plane by equilateral triangles, created by three families of parallel lines. 
This is a \emph{model apartment}. Lines of $\Sigma$ are called \emph{walls}, each wall divides $\Sigma$ into two connected components, which are called \emph{half-apartments}.
Maximal simplices (\textit{i.e} triangles) are called \emph{alcoves}. 

\begin{definition}
 A \emph{building of type $\tilde A_2$} is a simplicial complex $X$, of dimension 2, covered by copies of $\Sigma$ (called \emph{apartments}), such that
 \begin{itemize}
     \item Every two simplices of $X$ are contained in some apartment
     \item If $A,A'$ are two apartments then there exists an isomorphism $A\to A'$ fixing $A\cap A'$ pointwise
     \item Every edge is contained in at least 3 triangles.
 \end{itemize}
\end{definition}

From now on we fix an $\tilde A_2$ building $X$\index{$X$ \hfill $q$-regular building of type $\tilde A_2$|bb}. We equip it with the maximal system of apartments: every embedded copy of $\Sigma$ in $X$ is an apartment. The number of triangles containing a given edge can be proven not to depend on the edge, and we will assume it is finite. This number will be written $q+1$, where $q>1$ is the \emph{thickness} of the building.

We define an alcove (resp. a wall, an half-apartment) in the building as the image of the image of an alcove (resp. a wall, an half-apartment) by an embedding of $\Sigma$ in $X$.

The following is \cite[Theorem 11.52]{AbramenkoBrown}.

\begin{theorem}\label{thm:isomappart}
Let $X$ be an $\tilde A_2$ building. Let $C\subset X$ be either convex or of nonempty interior. Assume that $C$ is isometric to a subset of an apartment. Then there exists an apartment $A\subset X$ containing $C$.
\end{theorem}

From this we can deduce the following well-known corollaries (see eg \cite[Exercice 5.83a)]{AbramenkoBrown}.)

\begin{corollary}\label{cor:sequenceappt}
Let $(C_n)_{n\in \NN}$ be a sequence of convex subsets of $X$, with $C_i\subset C_{i+1}$ for all $i\in \NN$. Assume that each $C_i$ is contained in an apartment. Then there exists an apartment containing $\bigcup_{n\in\NN} C_n$.
\end{corollary}

\begin{corollary}\label{cor:root+alcove}
Let $\alpha$ be an half-apartment contained in $X$, and $c$ be an alcove with on edge on the wall $\partial\alpha$. There exists an apartment containing $\alpha\cup c$.
\end{corollary}

The vertices of $\Sigma$ can be colored uniquely (up to renumerotation) by three colors in such a way that any two adjacent vertices are of a different colors. These colors are also called \emph{types}. Using the second axiom one can check that this coloration extends to the whole building, so that each vertex of $X$ has a given type. An automorphism of $\Sigma$, or of $X$, is called \emph{type-preserving} if it preserves the type. If it does not, then it induces a permutation of the set of types, and it is called \emph{type-rotating} if this permutation is cyclic.

Fix an origin vertex $0\in \Sigma$. Then the walls passing through $0$ divide $\Sigma$ into 6 connected components, which are called \emph{sectors} or \emph{Weyl chambers}. We choose one of them and call it the \emph{positive Weyl chamber} of $\Sigma$, and denote it by $\Lambda$\index{$\Lambda$ \hfill a fixed sector in $\Sigma$|bb}. For every pairs of vertices $x,y\in \Sigma$, there is a unique type-rotating permutation which sends $x$ to $0$ and $y$ to an element of $\Lambda$. This element is called the \emph{combinatorial distance} and denoted $\sigma(x,y)$\index{$\sigma:X\times X\to \Lambda$ \hfill the combinatorial distance|bb}. More generally, if $x,y\in X$ then there is an apartment containing $x$ and $y$, and we can define $\sigma(x,y)$ as the combinatorial distance between $x$ and $y$ in this apartment. This does not depend on the choice of the apartment.
For $\lambda\in \Lambda$, and a vertex $x\in X$, the set 
$$V_\lambda(x)= \{y\in X \mid \sigma(x,y) = \lambda\}$$
is the \emph{combinatorial sphere} of radius $\lambda$ centered at $x$.

The \emph{convex hull} of a set $F\subset X$ of vertices of $X$ is the convex hull in the 1-skeleton of $X$. It is denoted $\Conv(F)$\index{$\Conv$ \hfill convex hull in the 1-skeleton of $X$|bb}. If $F$ is contained in an apartment, then $\Conv(F)$ is the intersection of all apartments containing $F$. Note that if $x,y$ are two vertices of $X$ then the isometry class of $\Conv(x,y)$ only depends on $\sigma(x,y)$, by convexity of sectors.

We define the \emph{length} of an element $\lambda\in\Lambda$ as the distance between $o$ and $\lambda$ in the 1-skeleton of $\Sigma$. We denote this length $\ell(\lambda)$\index{$\ell$ \hfill length function on $\Lambda$|bb}.

\subsection{Boundaries}

Let us now define the boundary of $X$ and its structure as a spherical building. We have defined a sector in $\Sigma$, and as usual we define a sector in $X$ as the image of a sector in $\Sigma$ by any embedding in $\Isom(\Sigma,X)$. A sector is \emph{based} at some vertex of $X$, which is the image of $0$ by this embedding. Two sectors are \emph{equivalent} if their intersection contains another sector. Equivalence classes of sectors are called \emph{chambers at infinity}, or sometimes just chambers. The set of all chambers is denoted by $\Delta$.

\begin{lemma}
For each chamber $C\in \Delta$, and vertex $x\in X$, there is a unique sector in the class of $C$ which is based at $x$. This sector is denoted $Q(x,C)$.\index{$Q(x,C)$ \hfill sector in the class of $C$ based at $x$|bb}
\end{lemma}

We define, for $x,z\in X$ two vertices, the set
$$\Omega_x(z)=\{C\in \Delta\mid z\in Q(x,C)\}$$
\index{$\Omega_x(z)$ \hfill shadow of $z$ from $x$ in $\Delta$|bb}
When $z$ varies, these sets form a basis for a topology of $\Delta$. This topology does not depend on $x$, and turns $\Delta$ into a compact metrizable space. The group $\Aut(X)$ acts on $\Delta$ by homeomorphisms. A chamber is in the \emph{boundary} of some apartment $A$ if it is the class of some sector in $A$. A basic fact is that any two chambers of $X$ are contained in the boundary of some common apartment.

A \emph{singular ray} of $X$ is a half-line which is contained in a wall of $X$. Two singular rays $\rho_1,\rho_2$ are \emph{equivalent} if there exists an apartment which contains two sub-rays $\rho'_1\subset\rho_1$ and $\rho'_2\subset \rho_2$ and such that $\rho'_1$ and $\rho'_2$ are parallel. The \emph{vertices at infinity} of $X$ are the equivalence classes of singular rays. In fact, there are two types of singular rays, depending on the cyclic permutation of the types of vertices in the ray. We call the two types $+$ and $-$. The set of vertices at infinity of type $+$ is denoted $\Delta_+$, and similarly the set of vertices at infinity of type $-$ is denoted $\Delta_-$. \index{$\Delta_+,\Delta_-$ \hfill set of vertices at infinity of a given type|bb}
For each $u\in \Delta_+\cup\Delta_-$, and every vertex $x\in X$, there is a unique singular ray in the class of $u$ starting from $x$, and we will denote it $[xu)$. Once again, this defines a topology on $\Delta_-$ and $\Delta_+$, two vertices at infinity $u,u'$ being close if the rays $[xu)$ and $[xu')$ share a long initial segment. 

Each sector is bounded by two singular rays. This descends to the corresponding equivalence classes and allow us to define two maps $\pr_+:\Delta\to \Delta_+$ and $\pr_-:\Delta\to \Delta_-$\index{$\pr_+:\Delta\to \Delta_+$, $\pr_-:\Delta\to \Delta_-$ \hfill natural projections|bb}. We also say that a chamber $C$ \emph{contains} the two vertices at infinity $\pr_+(C)$ and $\pr_-(C)$. These two maps are continuous and equivariant with respect to the group of type-rotating automorphisms of $X$. For $u\in \Delta_+$ (resp. $\Delta_-$), the \emph{residue} of $u$, denoted $\Res(u)$, is the fiber $\pr_+^{-1}(u)$ (resp. $\pr_-^{-1}(u)$). \index{$\Res(u)$ \hfill residue of $u$|bb}

It is easily checked that this defines a bipartite graph structure, with set of vertices $\Delta_+\cup \Delta_-$, and with set of edges $\Delta$ (an edge $C$ being attached to a vertex $u$ if $C$ contains $u$). This graph has diameter 3, and is in fact a spherical building of type $A_2$. In particular, for each vertex at infinity $u$, one can define a \emph{projection} $\proj_u:\Delta\to \Res(u)$\index{$\proj_u$ \hfill projection on $\Res(u)$|bb}, which associates to a chamber $C\in \Delta$ the unique chamber $\proj_u(C)\in \Res(u)$ which is at minimal distance from $C$. This projection is continuous.

Two chambers $C,C'$ are \emph{opposite} if they are at maximal distance, that is, at distance 3. We denote $\Delta_\op$\index{$\Delta_\op$\hfill pair of opposite chambers|bb} the set of pairs $(C,C')$ of opposite chambers. Two vertices $u\in \Delta_-$ and $v\in \Delta_+$ are \emph{opposite} if there exists $C\in \Delta$ containing $u$ and $C'\in \Delta$ containing $v$ such that $C$ and $C'$ are opposite. We denote $\Delta_\pmop$\index{$\Delta_\pmop$\hfill set of pairs $(u,v)\in\Delta_-\times \Delta_+$ which are opposite |bb} the set of all pairs  $(u,v)\in \Delta_-\times \Delta_+$ which are opposite. The sets $\Delta_\op$ and $\Delta_\pmop$ are open subsets of $\Delta^2$ and $\Delta_-\times \Delta_+$ respectively.

\subsection{Wall-trees and projectivities}\label{sub:proj}

For $(u,v)\in \Delta_\pmop$, the \emph{interval} between $u$ and $v$ is the union of all apartments in $X$ containing $u$ and $v$ in their boundary. It is denoted $I(u,v)$.\index{$I(u,v)$\hfill interval between $(u,v)\in\Delta_{\pmop}$|bb} A \emph{singular line} from $u$ to $v$ is a wall of such an apartment consisting of the union of two singular rays $[xu)\cup [xv)$ for some $x$.

This interval, up to isometry, does not depend on $u$ and $v$, and has the following structure. If $\ell$, $\ell'$ are two singular lines between $u$ and $v$, then $\ell$, $\ell'$ are parallel (in some apartment), hence there is a well-defined distance $d(\ell,\ell')$ (which does not depend on the apartment chosen).  The set of all such geodesic, equipped with this distance, is in fact a $(q+1)$-regular tree. It is denoted $T_{u,v}$.\index{$T_{u,v}$\hfill wall-tree associated to $(u,v)\in\Delta_\pmop$|bb}

The \emph{model wall-tree} is a model for the simplicial complex $I(u,v)$: as a CAT(0) space,  it is isometric to a product $T\times \RR$ (where $T$ is a $(q+1)$ regular tree), but the simplicial structure does not come from this product decomposition: if $\gamma$ is a geodesic ray in $T$, then the product $\gamma\times \RR$ is isomorphic to the model apartment $\Sigma$. This model wall-tree is denoted $\Upsilon$.\index{$\Upsilon$\hfill model wall-tree |bb}

The automorphism group of this simplicial complex $\Aut(\Upsilon)$  is a subgroup of $\Aut(T)\times \Isom(\RR)$, hence is endowed with a projection $p_T:\Aut(\Upsilon)\to \Aut(T)$ (which is surjective),\index{$p_T:\Aut(\Upsilon)\to \Aut(T)$\hfill natural projection|bb} and $p_\RR:\Aut(\Upsilon)\to \Aut(\RR)$\index{$p_R:\Aut(\Upsilon)\to \Aut(\RR)$\hfill natural projection|bb}. Let $\Aut(\Upsilon)^0$ be the index 2 subgroup such that $p_\RR(\Aut(\Upsilon)^0)$ is formed only by translations. We denote by $S_\Upsilon$ the group $\Ker(p_T)\cap \Aut(\Upsilon)^0$ ; it is a subgroup isomorphic to $\ZZ$, which we call the group of \emph{translations} of $\Upsilon$.\index{$S_\Upsilon$\hfill $\Ker(p_T)\cap \Aut(\Upsilon)^0\simeq \ZZ$ acting on $\Upsilon$|bb}

For $u\in \Delta_-$ (or in $\Delta_+$), one can also define a \emph{panel tree}, by the following construction: if $\rho,\rho'$ are two singular rays pointing towards $u$, then, up to deleting finite segments, they are parallel in some apartment, so again there is a well-defined distance $d(\rho,\rho')$. This defines a metric space, which depends on $u$, and will be denoted $T_u$. There is a natural map $T_{u,v}\to T_u$, associating to a line $\ell$ from $u$ to $v$ the class of a ray $\rho\subset \ell$ pointing towards $u$. This map is in fact an isomorphism, so that $T_u$ is again an $(q+1)$-regular tree.\index{$T_u$ \hfill panel-tree associated to $u$|bb}

\begin{remark}
The terminology wall-tree/panel-tree comes from the book of Weiss \cite{WeissBook}, even though the spaces are defined with a different point of view there. Another terminology, with yet another point of view, can be found in the work of G. Rousseau (see the book \cite{RousseauBook}), where the wall-trees (resp. panel-trees) are special cases of an \emph{inner façade} (resp. a \emph{façade}).
\end{remark}

Let $C$ be a chamber in the residue of $u$. Any sector $Q$ representing $C$ is a union of geodesic rays pointing towards $u$, hence defines a subset of $T_u$, which is in fact itself a geodesic ray of $T_u$. If $Q'$ also represents $C$, then $Q\cap Q'$ contains a subsector, therefore the corresponding geodesic rays in $T_u$ have an infinite intersection, so they define a point in the boundary of $T_u$. This defines a natural map $\phi_u:\Res(u)\to \partial_\infty T_u$.

The following Lemma is well-known, see for example \cite[Lemma 4.2]{RemyTrojan} for a proof.

\begin{lemma}\label{lem:bdTu}
For every $u\in \Delta_-$ or $u\in \Delta_+$, the map $\phi_u:\Res(u)\to \partial_\infty T_u$ is a homeomorphism.
\end{lemma}

If $(u,v)\in \Delta_\pmop$ then from the previous discussion we have  isomorphisms $T_u\to T_{u,v}$ and $T_{u,v}\to T_v$. Their composition is an isomorphism $T_u\to T_v$ called a \emph{perspectivity} and denoted $[u;v]$. If $w$ is another vertex opposite $u$, then the composition $[v;w]\circ [u;v]$ is denote $[u;v;w]$, and we can define similar notations with more vertices.
 
  If we have a sequence of vertices $u_0,u_1,\dots,u_n$ such that $u_i$ is opposite $u_{i+1}$ for every $i$, and $u_n$ is opposite $u_0$, then we can define the \emph{projectivity} 
$[u=u_0;u_1;u_2;\dots; u_n;u]$. The \emph{projectivity group} is the group generated by such projectivities. Up to conjugation, it does not depend on the choice of $u$. Choosing an identification of $T_u$ with the $q$-regular tree $T$, this provides an action of the projectivity group on $T$ (which, again, does not depend on $u$ up to conjugation. Its closure in $\Aut(T)$ (equipped with the compact open topology) is denoted $\Pi$. By construction, $\Pi$ comes with a natural action on every tree $T_u$ and $T_{u,v}$. Using Lemma \ref{lem:bdTu} we see that $\Pi$ also acts on $\Res(u)$.\index{$\Pi$\hfill closure of the projectivity group |bb}

The following lemma is classical, see for example \cite[Proposition 2.4]{CameronBook}

\begin{lemma}\label{lem:3transitive}
The natural action of $\Pi$ on $\partial T$ is 3-transitive.
\end{lemma}

\begin{example}
    In the Bruhat-Tits case when $X$ is the building of $G=\PGL_3(K)$, for $K$ a local field, then $\Delta$ is the set of flags of the projective plane over $K$. Then the associated group of projectivities is $\PGL_2(K)$ (and the action on $T$ is conjugated to the natural action on the Bruhat-Tits tree of $\PGL_2(K)$).

    To see this, note first that a perspectivity $[u;v;w]$ is equal to the restriction (to $\Res(u)$) of some unipotent element (fixing $v$). Hence every projectivity, which is a product of such perspectivities, is the restriction of some element of $G$. So the group of projectivities at say $u$ is contained in the subgroup of $\Aut(T_u)$ induced by the stabilizer of $u$ in $G$. This group is isomorphic to $\PGL_2(K)$.

    On the other hand, the action of $\PGL_2(K)$ on $\partial_\infty T_u$ (which is the projective line over $K$) is sharply 3-transitive.  By Lemma \ref{lem:3transitive} it follows that the projectivity group is $\PGL_2(K)$.
\end{example}

\section{Symmetricity of buildings} \label{sec:SoB}

Our goal in this section is to explain how the setup of prouniform measures developed in \S\ref{sec:measure} can be used in the context of buildings. More precisely, our goal is to define prouniform measures on spaces of the form $\Isom(Y,X)$, for several different $Y$s. Starting from the end, we will have to consider the case when $Y$ is a sector: then we recover a classical construction of "harmonic measures" on the boundary. Another space we will use is $\Isom(\Sigma, X)$, called the \emph{Cartan flow}, where we recover a previous construction of \cite{BCL}. The technically most challenging space is the \emph{singular flow}, introduced also in \cite{BCL}, but which is not really a prouniform measure. However, we will be led to use a kind of intermediate space, the \emph{detecting flow}, which is $\Isom(M,X)$ for $M$ some specific subspace of a wall-tree. However, in order to be able to define the prouniform measure, we must check the appropriate symmetricity, which is what we start by doing.

\subsection{Convex subset of wall trees}

Our first goal
is to prove the following theorem. Let us say that a finite simplicial complex has property $P_\Upsilon$  if it is isometric to a convex subset of a model wall tree $\Upsilon$ in a $q$-regular $\tilde A_2$ building.\index{$P_\Upsilon$\hfill be a convex subset of a wall-tree|bb}

\begin{theorem}\label{thm:Psymetric}
A building of type $\tilde{A_2}$ is $P_\Upsilon$-symmetric.
\end{theorem}

This Theorem will be proved in the rest of the section. Using the results of Section 
 \ref{sec:measure}, we get the following Proposition as a result of Theorem \ref{thm:Psymetric}.

\begin{remark}
    For more general 2-dimensional buildings one can define easily wall-trees and panel-trees in the same way. However, the proof of $P_\Upsilon$-symmetricity that we give here really relies on the assumption that we have $\tilde A_2$ buildings; we do not know if the corresponding statement hold in general.
\end{remark}

\begin{proposition}
Let $\Upsilon$ be a model wall-tree as in \S\ref{sub:proj}. Then for every (finite or infinite) convex subset $Z$ of $\Upsilon$ there exists a measure $\mu_Z$ on the set of simplicial isometries $\Isom(Z,X)$ satisfying
\begin{itemize}
\item For $Z=\{\ast\}$ this is the counting measure on $X \simeq \Isom(\{\ast\},X)$
\item For a simplicial isometry $Z_1 \to Z_2$, the restriction map $\Isom(Z_2,X) \to \Isom(Z_1,X)$ is measure preserving.
\end{itemize}

In particular, $\mu_Z$ is invariant under $Aut(Z) \times Aut(X).$
\end{proposition}

This proposition will be used for various spaces $Z$ in \S \ref{sec:flows} below.

\begin{remark}
    Consider the property $P_X$ of being isometric to a convex subset of an $\tilde A_2$-building.
    Then it is not true that an $\tilde A_2$-building (even a cocompact one) is necessarily $P_X$-symmetric. To construct an example, remember that there exists an exotic building $X$ (with a cocompact lattice) in which there are two different isometry types of balls of radius 2, but only one of radius 1. Hence if $B_1\subset B_2$ is an inclusion of a ball of radius 1 in a ball of radius 2 (of one of these two types) and $\varphi:B_1\to X$ is an embedding, then depending on $\varphi$ there is either 0 or some positive number of extensions of $\varphi$ to an embedding of $B_2$.
\end{remark}

\subsubsection{Crazy diamonds}\label{sec:crazydiamonds}

In order to prove Theorem \ref{thm:Psymetric}, we will give a combinatorial description of convex subsets of wall trees.

\begin{definition}%\jl{Here we identify $T$ with its set of vertices}
A \emph{decorated tree} is a 
 finite connected tree $T$, with two distinguished vertices $x,y$ and two integers $s,t\in \ZZ$, such that
 \begin{itemize}
 \item $s-t=d(x,y) \mod 2$.
 \item For every $z\in T$, $s-t-d(x,z)-d(y,z)\geq 0$.
 \end{itemize}
\end{definition}

\begin{remark}
\begin{itemize}
\item In what follows only the quantity $s-t$ is of interest. In particular we can always assume $s,t\geq 0$, or even $t=0$.
\item Let $l(z)=s-t-d(y,z)-d(x,z)$. Then $l(z)$ is even for every $z\in T$.
\end{itemize}
\end{remark}

\begin{definition}
Let $(T,x,y,s,t)$ be a decorated tree. For $z\in T$, let $c(z)=s-d(x,z)$ and $f(z)=t+d(y,z)$. Note that by assumption we have $c(z)-f(z)\geq 0$. %\jl{$f$ for floor, $c$ for ceiling...}

The \emph{crazy diamond} associated to a decorated tree is the subset of $\Upsilon$ defined as 
$$Y=Y(T,x,y,s,t)=\{ (z,u) | u \in [f(z),c(z)] \}.$$
\index{$Y(T,x,y,s,t)$\hfill crazy diamond associated to a decorated tree|bb}

The space $Y$ is endowed with a metric as follows : each edge of the tree has length $\frac{\sqrt 3}2$,the metric on $\RR$ is the usual metric, and the metric on $X$ is the $\ell^2$-product metric. With this metric, it is isomorphic to a subset of $\Upsilon$, and in fact it is isometric to a convex subcomplex of $\Upsilon$. We equip it with the associated tessellation.
\end{definition}

\begin{example}
Here is an example. The tree $T$ is the line with 5 vertices drawn at the bottom. Here we took $s=6$ and $t=-1$.

\begin{center}

\definecolor{uuuuuu}{rgb}{0.26666666666666666,0.26666666666666666,0.26666666666666666}
\definecolor{xdxdff}{rgb}{0.49019607843137253,0.49019607843137253,1.}
\begin{tikzpicture}[line cap=round,line join=round,>=triangle 45,x=1cm,y=1cm]
\clip(-3.063649411742506,-2.6488525728334564) rectangle (4.903566579885092,3.793603116948162);
\draw (0.,0.)-- (0.,3.);
\draw(0.,0.)-- (2.598076211353316,1.5);
\draw (0.,3.)-- (2.598076211353316,1.5);
\draw (0.,3.)-- (-1.7320508075688772,2.);
\draw (-1.7320508075688772,2.)-- (-1.7320508075688772,1.);
\draw (-1.7320508075688772,1.)-- (0.,0.);
\draw  (0.8660254037844386,0.5)-- (0.8660254037844387,2.5);
\draw (1.7320508075688772,1.)-- (1.7320508075688772,2.);
\draw  (-0.8660254037844387,0.5)-- (-0.8660254037844387,2.5);
\draw  (-1.7320508075688772,2.)-- (0.8660254037844386,0.5);
\draw  (1.7320508075688772,1.)-- (-0.8660254037844387,2.5);
\draw  (0.8660254037844387,2.5)-- (-1.7320508075688772,1.);
\draw  (-0.8660254037844387,-0.5)-- (-0.8660254037844387,2.5);
\draw  (-1.7320508075688772,0.)-- (-0.8660254037844387,-0.5);
\draw (-1.7320508075688772,1.)-- (-1.7320508075688772,0.);
\draw (-0.8660254037844387,-0.5)-- (0.,0.);
\draw (-1.7320508075688772,0.)-- (1.7320508075688774,2.);
\draw  (-1.7320508075688772,-2.0107256100017876)-- (2.598076211353316,-1.9934519475164805);
\begin{scriptsize}
\draw [fill=uuuuuu] (-1.7320508075688772,-2.0107256100017876) circle (2pt);
\draw [fill=uuuuuu] (2.598076211353316,-1.9934519475164805) circle (2pt);
\draw [fill=uuuuuu] (-0.8660254037844386,-2.0072708775047268) circle (2.0pt);
\draw[color=uuuuuu] (-0.7009577728460465,-1.75) node {$y$};
\draw [fill=uuuuuu] (0.,-2.003816145007665) circle (2.0pt);
\draw[color=uuuuuu] (0.12323698490853234,-1.75) node {$x$};
\draw [fill=uuuuuu] (0.8660254037844386,-2.0003614125106037) circle (2.0pt);
\draw [fill=uuuuuu] (1.7320508075688772,-1.996906680013542) circle (2.0pt);
\end{scriptsize}
\end{tikzpicture}

\end{center}

\end{example}

The projection $Y\to T$ given by the first coordinate will be denoted $\pr$.

Let $X$ be a $q$-regular $\tilde A_2$ building, and let $\mathcal T$ be a $(q+1)$-regular tree. Recall that $\Upsilon$ is a model wall-tree of $Y$. It is equipped with a projection $\pi:\Upsilon\to \mathcal T$.

\begin{definition}
 A crazy diamond is called \emph{degenerate} if it is of one of the following forms :
 \begin{itemize}
     \item $Y$ is of the form $Y(\{x\},x,x,s,s)$, that is, $Y$ is a point
     \item $Y$ is of the form $Y(\{x\},x,x,s,t)$ with $s-t\geq 0$. In that case we say that $Y$ is \emph{vertical}. 
     \item $Y$ is of the form $Y(T,x,y,s,t)$ where $T$ is a segment whose extremities are $x$ and $y$, and $s-t=d(x,y)$, in which case we say that $Y$ is \emph{horizontal}.
 \end{itemize}
\end{definition}

\begin{comment}
\begin{lemma}\label{lem:nondegenerate}
 If $Y$ is a non-degenerate crazy diamond, then every edge of $Y$ is contained in at least one alcove of $Y$.
\end{lemma}

\begin{proof}
    By construction we can embed $Y$ as a convex subset of $\Upsilon$. First note that if $Y$ contains an alcove then the Lemma holds: indeed, the projection of this alcove on every edge is an alcove which must be contained in $Y$ by convexity.
    
    Let $e$ be an edge in $Y$. Since $Y$ is non-degenerate, there exists an edge $e'\neq e$. If there is an alcove in the link of $e$ which is the closest to $e'$, then by convexity of $Y$ this alcove is in $Y$, which is what we wanted. Assuming that this is not the case, this means that the convex hull of $e$ and $e'$ is a segment whose extremities are $e$ and $e'$.

    This means in particular that if $x$ is a vertex of $Y$, then all the vertices in the link of $x$ corresponding to edges in $Y$ are at pairwise distance $\pi$. But since $\lk(x)$ is a building of type $A_2$, this implies that there are at most two edges in $Y$ containing $x$. Since this is valid for every $x$, we get that $Y$ is a segment, contained in a wall. The two cases above describe the different directions of walls in $\Upsilon.$
\end{proof}
\end{comment}

\begin{proposition}\label{prop:cvxWT}
Every finite convex subcomplex of $\Upsilon$ is isometric to some crazy diamond.
\end{proposition}

Recall that a function on a tree is convex if and only if it is convex in restriction to every geodesic segment.

\begin{lemma}\label{lem:fonT}
Let $T$ be a finite connected tree, and $f:T\to \ZZ$ be a convex function such that if $x$ is adjacent to $y$ then $f(x)=f(y)\pm 1$.
Then there exists $t\in \ZZ$ and $x\in T$ such that $f(z)=t+d(z,x)$ for every vertex $z\in T$.
\end{lemma}

\begin{proof}
The set of points where $f$ attains its minimum is non-empty (since $T$ is finite) and convex (by convexity of $f$). Since there are no adjacent vertices with the same value, $f$ is minimal at exactly one point $x$. Let $f(x)=t$. If $x'$ is adjacent to $x$ then we must have $f(x')=f(x)+1=t+d(x',x)$. Now if $z\in T$ then considering the geodesic segment $x,x_1,\dots,z$ we see (by convexity) that $f$ must be increasing along this segment, hence that $f(z)=t+d(z,x)$.
\end{proof}

\begin{proof}[Proof of Proposition \ref{prop:cvxWT}]
Let $Y$ be a finite convex subset of $\Upsilon$. Recall $\pi:\Upsilon\to \mathcal T$ is the natural projection, and let $T=\pi(Y)$. Note that $T$ is a finite, connected subtree of $\mathcal T$. Since $Y$ is a convex subcomplex, $Y$ is completely determined by the sets $\pi^{-1}(x)\cap Y$, for $x$ a vertex in  $T$, which are closed segments.  

Let $u\in \partial_\infty\Upsilon$ be an endpoint of a geodesic line $\ell\subset\Upsilon$ such that $\pi(\ell)$ is just a point.
Let $h$ be the horofunction on $X$ based at $u$ (relative to an arbitrary origin $o\in \Upsilon$). For $x\in T$, define $f(x)=\min \{ 2h(p)\mid p\in \pi^{-1}(x)\cap Y\}$ and $c(x)=\max \{ 2h(p)\mid p\in \pi^{-1}(x)\cap Y\}$. Note that by convexity of $Y$, if $x,y\in T$ are adjacent vertices, and if $p\in \pi^{-1}(x)\cap Y$ is a vertex of $Y$, then $\pi^{-1}(y)\cap Y$ contains a vertex $q$ which is adjacent to $p$, hence such that $h(q)=h(p)\pm \frac 12$. Hence, we get that $f(y)=f(x)\pm 1$ and $c(y)=c(x)\pm 1$. Note furthermore that, by construction of $\Upsilon$, if $p,q\in \Upsilon$ are vertices, then $2(h(p)-h(q))$ is of the same parity as $d(\pi(p),\pi(q))$.

If $\ell$ is a geodesic segment in $T$, then $\ell$ is contained in a geodesic bi-infinite line (still denoted $\ell$) in $\mathscr T$. Then $A=\pi^{-1}(\ell)\subset \Upsilon$ is an apartment of $X$. Hence its intersection with $Y$ is a convex subset. Since this intersection is precisely $\{p \mid f(\pi(p))\leq 2h(p)\leq c(\pi(p))\}$, it follows that $f$ is convex and $c$ is concave in restriction to $s$. Since this is true for every segment, $f$ is convex on $T$. Hence by Lemma \ref{lem:fonT}, there exists $y\in T$ and $t\in \ZZ$ such that $f(z)=t+d(z,y)$. Similarly, $-c$ is convex, so there exists $s\in \ZZ$ and $x\in T$ such that $c(z)=s-d(y,z)$. Note also that $f(y)-c(x)=t-s$ is of the same parity as $d(y,x)$ by the above remark.

 By construction it is clear that $Y$ is isometric to the crazy diamond $Y(T,x,y,s,t)$.
\end{proof}

\begin{remark}
The same argument also proves that every convex subset of a crazy diamond is again a crazy diamond.
\end{remark}

\subsubsection{Diamond extensions}

Let $Y=Y(T,x,y,s,t)$ be a crazy diamond. In what follows, for $z\in T$, we note $c(z)=s-d(x,z)$, $f(z)=t+d(y,z)$ and $l(z)=c(z)-f(z)$. When we have to consider another crazy diamond $Y'$, we will denote the same quantities with primes if they refer to $Y'$.

Now we are interested in an isometric  inclusion $Y\subset Y'$ of crazy diamonds.   Note that if $T\subset T'$, and $c(z)\leq c'(z)$ and $f(z)\geq f'(z)$ for every $z\in T$ then we get that $$\{ (z,u)\in T\times\RR | u \in [f(z),c(z)] \}\subset\{ (z,u) \in T'\times \RR| u \in [f'(z),c'(z)] \}.$$
Hence $Y\subset Y'$. 

Every embedding is in fact of this form :

\begin{lemma}\label{lem:extstd}
Let $Y\subset Y'$ be an isometric inclusion of crazy diamonds. Let $Y'=Y(T',x',y',s',t')$. Then there exists a subtree $T\subset T'$, $x,y\in T$ and $s,t\in \ZZ$ such that $Y=Y(T,x,y,s,t)$ and the inclusion $Y\subset Y'$ is as described above.

%Furthermore, $T,x,y,s,t$ are uniquely defined by the inclusion $X\subset X'$. 
\end{lemma}

\begin{proof}
Embed $Y'$ (and therefore $Y$) into $\Upsilon$. Going through the proof of Proposition \ref{prop:cvxWT}, we see that $T'=\pi(Y')$ and $T=\pi(Y)$, so that clearly $T\subset T'$, and by convexity of $Y$, $T$ is also convex and therefore is a subtree of $T'$. Similarly by the construction given in the proof of Proposition \ref{prop:cvxWT} we see that $f'(z)\leq f(z)$ and $c'(z)\geq c(z)$ for every $z\in T$.
\end{proof}

We will decompose this inclusion into a sequence of elementary extensions, defined as follows.

\begin{definition}
Let $Y=Y(T,x,y,s,t)$ be a crazy diamond.
\begin{itemize}
\item An \emph{elementary extension of type $C$} of $Y$ is a crazy diamond $Y'=Y(T,x',y,s+1,t)$, where $x'$ is a vertex adjacent to $x$.
\item An \emph{elementary extension of type $F$} of $Y$ is a crazy diamond $Y'=Y(T,x,y',s,t+1)$ where $y'$ is a vertex adjacent to $y$.
\item An \emph{elementary extension of type $B_z$} of $Y$ is a crazy diamond $Y'=Y(T',x,y,s,t)$ where $T'$ is obtained by attaching an edge at a vertex $z\in T$ such that $l(z)> 0$.
\end{itemize}
In each case, realizing $Y$ as a subset of $T\times \RR$ and $Y'$ as a subset of $T'\times \RR$, we see that there is a natural inclusion $Y\subset Y'$.
\end{definition}

Note that in an extension of type $C$ we have $$c'(z)=\begin{cases}
 c(z) &\textrm{ if } d(x,z)<d(x',z)\\
 c(z)+1 &\textrm{ if not}.
 \end{cases}
$$

In the case when $Y$ is degenerate, some of the extensions above are not possible (since $T$ is too small), but there are also some other possibilities which appear:

\begin{definition}
\begin{itemize}
    \item If $Y=Y(\{x\},x,x,s,t)$ is vertical, or is a point, an \emph{elementary extension of type $C_v$} (resp. \emph{of type $F_v$} is the crazy diamond $Y'=Y(\{x\},x,x,s+2,t)$ (resp. $Y'=Y(\{x\},x,x,s,t+2)$
    \item If $Y=Y(T,x,y,s,t)$ is horizontal, or a point, an \emph{elementary extension of type $B^c$} (resp. of type $B^f$) is of the form $Y'=Y(T',x',y,s+1,t)$ (resp. $Y(T',x,y',s,t+1)$) where $T'$ is obtained by attaching an edge from $x$ to $x'$ (resp. $T'$ is obtained by attaching an edge from $y$ to $y'$).
\end{itemize}
\end{definition}

Note that an elementary extension of type $C_v$ produces a vertical crazy diamond, while an elementary extension of type $B^f$ or $B^c$ produces an horizontal crazy diamond.

\begin{lemma}\label{lem:seqextensions}
Let $Y\subset Y'$ be an isometric inclusion of crazy diamonds, and assume that $Y$ is not degenerate. Then there is a finite sequence $Y=Y_0,Y_1,\dots,Y_n=Y'$ of crazy diamonds such that $Y_i\subset Y_{i+1}$ is an elementary extension. %Furthermore, the number $n$ (and in fact the number of extensions of each type) is uniquely defined.
\end{lemma}

\begin{proof}

As in Lemma \ref{lem:extstd}, we can and shall assume that $T\subset T'$. Note that if $z\in T'$ has valency $>1$ then $l'(z')\geq 1$.
Also note that since $c'(x)\geq c(x)$, we have $s'\geq s+d(x,x')\geq s$.

%Since extensions of type $C$ are the only ones which can (and do) increase $s$, 
%each of them by 1, if there is a sequence of elementary extensions from $Y$ to $Y'$, then it contains exactly $s'-s$  extension of type $C$. 
%Similarly a sequence of elementary extensions from $Y$ to $X'$ will contain $t-t'$ extensions of type $F$.
%The extensions of type $B_z$, for $z\in T'$ are the only ones that can change the tree $T$, and only by adding an edge. Therefore we need to perform such an extension exactly the number of edges in $T'\setminus T$. More precisely, for each $z\in T'$, we need to do $v'(z)-v(z)$ extensions of type $B_z$, where $v'(z)$ (resp. $v(z)$) is the valency of $z$ in $T'$ (resp. in $T$ if $z\in T$, and $1$ otherwise). 

%The only thing to do is to check that there exists indeed a sequence of elementary extensions from $Y$ to $Y'$.

First assume that $Y$ is non-degenerate.
 Our sequence of extensions will proceed as follows. The notations relative to $Y_i$ will be the same as for $Y$ but with index $i$. We use $i$ as a temporary variable and change it after each step.

 By doing two successive extensions of type $C$, we can replace $s$ by $s+2$ without changing $x$. We start by this move, doing $d$ extensions of type $C$, in order to get that $s_i=c_i(x)=c'(x)$. Note that $c'(x)-c(x)$ must be even since $Y$ is a subcomplex of $Y'$.  
  Then we do in a similar way the right number of extensions of type $F$, in order to get that $t_{i}=f_{i}(y)=f'(y)$. It follows that $c_i(z)=c'(z)$ as soon as $d(z,x)<d(z,x')$. By doing an extension of type $C$ again, moving $x$ to the vertex $z_0$ adjacent to $x$ and closer to $x'$, we get that  $c_i(z)=c'(z)$ for $z$ satisfying $d(z,x)<d(z,x')+1$. Repeating the same thing as many times as necessary, we can get that $c_i(z)=c'(z)$ for every $z\in T$. Similarly we can obtain $f_i(z)=f'(z)$ for every $z\in T$. 
  
  If $z\in T$ is of valency $>1$ in $T'$, then $l_i(z)=l'(z)\geq 1$. Therefore we can perform an extension of type $B_z$, in order to add as many edges as necessary. Repeating the above steps, we see that we can get that $T_i=T'$, $c_i=c'$ and $f_i=f'$, which concludes the argument.

Then let us assume that $Y$ is vertical (but not a point). If $Y'$ is non-degenerate, then there is an alcove in $Y'$ which is adjacent to some panel of $Y$; let $Y_1$ be the convex hull of $Y'$ and this alcove. Then $Y\subset Y_1$ is an elementary extension of type $B$, and by the previous argument there exists a finite sequence of elementary extensions from $Y_1$ to $Y'$.
If $Y'$ is also degenerate, then it must be vertical, then it is just a segment containing $Y$, and one can pass from $Y$ to $Y'$ by adding edges to this segment, that is, by a finite sequence of elementary extensions of type $C_v$ or $F_v$.

Now let us assume that $Y$ is horizontal (and not a point). If $Y'$ is non-degenerate, again there is an alcove in $Y'$ which is adjacent to $Y$, and the convex hull $Y_1$ of this alcove and $Y$ is an elementary extension of $Y$ of type either $C$ or $F$.  If $Y'$ is degenerate, then it must be horizontal, and then again one can pass from $Y'$ to $Y$ by a sequence of elementary extensions of type $B^f$ or $B^c$.

Finally, if $Y$ is a point, then there exists a segment in $Y'$ containing $Y$, and  the inclusion of $Y$ in this segment is
an elementary extension of type either $C_v$, $F_v$, $B^f$ or $B^c$ depending on the position of this segment in $Y'$. In any case applying the previous argument we can construct a sequence of elementary extensions from $Y$ to $Y'$.
\end{proof}

\begin{comment}
    \begin{lemma}\label{lem:seqextensions-degenerate}
Let $Y\subset Y'$ be an isometric inclusion of crazy diamonds, and assume that $Y$ is not degenerate. Then there is a finite sequence $Y=Y_0,Y_1,\dots,Y_n=Y'$ of crazy diamonds such that $Y_i\subset Y_{i+1}$ is an elementary extension. Furthermore, let $b$ (resp. $c$, $f$, $b^c$, $b^f$, $c_v$, $f_v$) be the number of extensions of type $B_z$ (resp. $C$, $F$, $b^c$, $b^f$, $c_v$, $c_f$) in that sequence. Then:
\begin{itemize}
    \item If $Y$ is vertical, then $b$, $c+c_v/2$ and $f+f_v/2$ do not depend on the choice of the sequence
    \item If $Y$ is horizonal, 
\end{itemize}
\end{lemma}
\end{comment}

Let $X$ be an $\tilde A_2$ building of finite thickness $q+1$. Assume that $i:Y\to Y'$ is an isometric embedding of crazy diamonds. Let $\rho:Y\to X$ be an isometric embedding.

%
%\begin{proposition}
% Let $X\subset X'$ be an isometric inclusion of crazy diamonds.
%There exists $N>0$ such that, if $\rho:X\to Y$ be an isometric embedding, then there exists exactly $N$ isometric embeddings $\rho':X'\to Y$ such that $\rho'_{|X}=\rho$.
%\end{proposition}

\begin{definition}
An apartment $A\subset X$ is said \emph{compatible with $\rho$} if it intersects $\rho(Y)$ and such that
\begin{itemize}
\item $\pr(\rho(Y)\cap T)$ is a segment between two leaves of $T$, and
\item for every $z\in T$,
 $\rho(\pr^{-1}(z))\cap A$ is either empty or the full segment of length
 $l(z)$.
\end{itemize} 
\end{definition}

\begin{lemma}\label{lem:compatible}
Any two chambers in $\rho(Y)$ are contained in some compatible apartment.
\end{lemma}

\begin{proof}
Let $\ell$ be a geodesic segment between two leaves of $T$. Then $\pr^{-1}(\ell)\subset Y$ is isometric to a subset of $\RR^2$. Hence
 by Theorem \ref{thm:isomappart}, it is contained in an apartment of $X$, which is therefore a compatible apartment.
Now any two edges of $T$ are contained in some segment between two leaves. Hence any two chambers of $X$ are contained in some maximal rhombus of $Y$, and therefore their image by $\rho$ is in a compatible apartment.
\end{proof}

In view of Lemma \ref{lem:seqextensions}, in order to prove Theorem \ref{thm:Psymetric}, it suffices to prove it when $Y\subset Y'$ is an elementary extension. 

\begin{lemma}\label{lem:extFC}
Assume that $Y\subset Y'$ is an elementary extension of type $F$ or $C$. Let $\rho:Y\to X$ be an isometric embedding. If $Y$ is non-degenerate, then there are $q$ possible extensions of $\rho$ to $\rho'\in \Isom(Y', X)$. If $Y$ is degenerate, then it is a segment and there are $q+1$ possible extensions.
\end{lemma}

\begin{proof}
We argue for an extension of type $C$, the other one being similar.

Let $e$ be the edge between $x$ and $x'$, and let $E$ be the set of edges of the subtree $T_{x'}=\{z\in T\mid d(z,x')<d(z,x)\}$. We first remark that $Y'$ is obtained from $Y$ by adding one alcove for each edge in $\{e\}\cup E$. For $e'\in \{e\}\cup E$, let $d_{e'}$ be this alcove, and $c_{e'}$ the alcove of $Y$ which is adjacent to $d_{e'}$ (which exists since $\ell(x)>0$, as $Y$ is non-degenerate), and $\sigma_{e'}$ be the panel between $c_{e'}$ and $d_{e'}$. Note that in $Y'$ the convex hull of $d_e$ and $Y$ is equal to $Y'$. 

Let $C=\rho(c_e)$. There are $q$ alcoves in $Y$ which are adjacent to $\rho(\sigma_e)$ and distinct from $C$. Let $D$ be such a alcove. We claim that there exists a unique $\rho':Y'\to X$ such that $\rho'(d_e)=D$. 

Indeed, the uniqueness follows from the fact that $X'$ is the convex hull of $d_e$ and $Y$. Let us argue for the existence. Let $e'\in E$. By Lemma \ref{lem:compatible} there exists a compatible apartment $A_0$ containing $C$ and $\rho(c_{e'})$. Furthermore $D\cup (\rho(Y)\cap A_0)$ is isometric to a subset of an apartment, hence is contained in an apartment by Theorem \ref{thm:isomappart}. Therefore, there exists a
 compatible apartment $A$ containing $\rho(d_e)$ and $\rho(c_{e'})$. In the apartment $A$, there is a unique alcove $D_{e'}$ which is adjacent to $\rho(\sigma_{e'})$ and distinct from $\rho(c_{e'})$. 

 We define $\rho'(d_{e'})=D_{e'}$. This is well-defined : indeed, if $A'$ is another compatible apartment satisfying the same conditions, then $A'$ contains the convex hull of $\rho(d_e)$ and $\rho(c_{e'})$, hence also $D_{e'}$. Therefore $D_{e'}$ does not depend on the choice of the apartment. Using again Lemma \ref{lem:compatible}, we see that the distance between two chambers can be calculated in a compatible apartment. It follows that $\rho'$ is again isometric. 
 
 Therefore, each of the $q$ possible choices of $D$ gives a map $\rho'$, which concludes the proof in the case when $Y$ is non-degenerate.

 If $Y$ is degenerate, then for an extension of type $C$ to be possible, it must be a horizontal segment in $Y'$. In that case, the same argument applies, but there are in fact $(q+1)$ possible choices of $D$.

\end{proof}

\begin{lemma}\label{lem:extBz}
Assume that $Y\subset Y'$ is an elementary extension of type $B_z$, where $z$ is a vertex of valency $v$ in $T$. Then there are $q-v+1$ possible extensions of $\rho:Y\to X$ to $\rho':Y'\to X$ if $q-v+1\geq 0$, and $0$ if not. 
\end{lemma}

(Note that if $Y$ is vertical, then the lemma still holds, with $v=0$).

\begin{proof}
Let $P$ be the set of panels in $\pr^{-1}(z)\subset Y$. By assumption $l(z)>0$, so that $P$ is not empty. Let $l=l(z)$.
For $\sigma\in P$, there is exactly one alcove in $Y'$ adjacent to $\sigma$ but which is not in $Y$. Let $c_\sigma$ be this alcove. Note that for every $\sigma\in P$, we have that $\Conv(c_\sigma\cup Y)=Y'$.

Let $\rho:Y\to X$ be an isometric embedding. Let $\sigma\in P$. Let $d$ be an alcove adjacent to $\rho(\sigma)$ and not contained in $Y$. Then we claim that there exists exactly one map $\rho':Y'\to X$ such that $\rho'(c_\sigma)= d$. The result follows immediately since there are $q+1$ alcoves in $Y$ adjacent to $\sigma$ and $v$ of them are contained in $\rho(Y)$.

So let us prove the claim. The uniqueness is clear since $Y'=\Conv(c_\sigma\cup Y)$. The existence follows from a similar argument as above. Let $e$ be the new edge in $Y'$, and let $A$ be a compatible apartment for $\rho$ containing $\rho(\sigma)$ (and therefore $\rho(\pr^{-1}(z))$). The wall of $A$ containing $\rho(\sigma)$ divides $A$ into two half-spaces $\alpha$ and $\beta$. Then $\alpha\cup d$ is contained in some apartment $A_1$. In $A_1$ we see that the convex hull of $\rho(Y)\cap A_1$ and $d$ is formed by adding the $2l+1$ alcoves not in $\alpha$ and sharing either an edge or a vertex with some panel $\rho(\sigma')$, for some $\sigma'\in P$.

 Note that this description is independent on the choice of $A$ or $\alpha$ (indeed, it is just given by the convex hull of $d$ and all panels $\rho(\sigma')$, for $\sigma'\in P$).  Since every alcove of $\rho(Y)$ is contained in a compatible apartment containing also $\rho(\sigma)$, it follows that the convex hull of $\rho(Y)$ and $d$ is formed by adding exactly these $2l+1$ alcoves. So we can extend $\rho'$ in the obvious way to the $2l+1$ alcoves that are in $Y'\setminus Y$. Since any alcove in $\rho(Y)$ is contained in a compatible apartment containing $\rho(\pr^{-1}(z))$, we see that the extension $\rho'$ is indeed isometric. 
\end{proof}

\begin{lemma}\label{lem:extdegenerate}
    Let $Y\subset Y'$ be an elementary extension of type $F_v$ or $C_v$ or $B^c$ or $B^f$, and let $\rho\in \Isom(Y,X)$  If $Y$ is not a point, then there are $q^2$ possible extensions of $\rho$ to $\rho'\in \Isom(Y', X)$. If $Y$ is a point, then there are $q^2+q+1$ such extensions.
\end{lemma}

\begin{proof}
    If $Y$ is a point, then $Y\subset Y'$ is just the inclusion of a point in an edge, and this is just counting the degree of vertices in the 1-skeleton of $X$. 

    If $Y$ is a segment, then any of these elementary extensions amount to adding an edge at an extremal vertex. In the link of this vertex, this edge must be opposite the image of $Y$. There are $q^2$ such edges.
\end{proof}

\begin{proof}[Proof of Theorem \ref{thm:Psymetric}]
Let $Y,Y'$ be  crazy diamonds and $i:Y\to Y'$ be an isometry. Let $\rho:Y\to X$ be an isometry. By Lemma \ref{lem:seqextensions}, there exists a sequence of crazy diamonds $Y=Y_0,Y_1,\dots, Y_n=Y$, such that $i_k:Y_k\subset Y_{k+1}$ is an elementary extension. The number of maps $\rho':Y'\to X$ such that $\rho'\circ i=\rho$ is the product over $k$ of the numbers of maps $\rho_{k+1}$ such that $\rho_{k+1}\circ i_k=\rho_k$. By Lemma \ref{lem:extFC}, Lemma \ref{lem:extBz} and \ref{lem:extdegenerate}, this number only depends on the inclusion $Y_i\subset Y_{i+1}$. Therefore the desired number only depends on the inclusion $Y\subset Y'$.
\end{proof}

\subsection{Flows and measures}\label{sec:flows}

In this section, we revisit various constructions of measures and dynamical systems associated to buildings, in the light of the construction of prouniform measures.

\begin{convention}
In this text, the index of the measure will be an indication of the space on which this measure is defined. The exponent will usually denote some origin or restriction.

We also will use the letter $\mu$ only for (restricted or not) prouniform measures.
\end{convention}

\subsubsection{Harmonic measures and the Cartan Flow}

Let us start again with  $X$  an $\tilde A_2$ building, and $\Sigma$  a model $\tilde A_2$ apartment.  Recall that the property $P_\Upsilon$ was defined to be "being a finite convex subset of the wall-tree $\Upsilon$". Note that every convex subset of $\Upsilon$ is ind-$P_\Upsilon$ (as  the convex hull of a finite set in $\Upsilon$ is finite). %\jl{I think we still need an argument for the converse ?}

Recall that we have fixed an origin $0\in \Sigma$, and a sector $\Lambda\subset \Sigma$ based at $0$. Note that $\Sigma$ and $\Lambda$ are ind-$P_\Upsilon$.
Let $o\in X$. The set $\Delta$ of chambers  in the boundary of $X$ can be identified with the set of sectors based at $o$. Therefore we can define the restricted prouniform measure $\mu_\Lambda^o$ on $\Isom(\Lambda,X)^o\simeq \Delta$.  In the following of the text we will denote this measure $\mu_\Delta^o$. \index{$\mu_\Delta^o$ \hfill restricted prouniform measure on $\Delta$|bb}This measure has been studied for a long time, previous definitions can be found in  \cite{Parkinson} or more recently \cite[Proposition 6.1]{RemyTrojan}.  

Recall that we have a natural map $\pr_+:\Delta\to\Delta_+$, associating to a chamber its vertex of type $+$. This map is $\Aut(X)$-equivariant and continuous. Similarly we have $\pr_-:\Delta\to \Delta_-$. We also define the pushforward measures $\mu_+^o =(\pr_+)_* \mu_\Delta^o\in \Prob(\Delta_+)$ and $\mu_-^o =(\pr_-)_*\mu_\Delta^o\in \Prob(\Delta_-)$. \index{$\mu_+^o,\mu_-^o$\hfill pushforward of $\mu_\Delta^o$ by $\pr_+$, $\pr_-$|bb}

\begin{remark}
The measures $\mu_+^o$ and $\mu_-^o$ are not prouniform measures \textit{per se}. However, one could define a slightly different notion a prouniform measures, taking into account the type and considering only type-preserving simplicial isometries. In this sense, identifying $\Delta_+$ with the set of rays of type $+$ starting from $o$, one could check that we would indeed get the measure $\mu_+^o$. This point of view should be very important if one tries to extend the constructions we make here to the setting of buildings of other types, e.g $\tilde{C}_2$ and $\tilde{G}_2$. We decided not to elaborate on this further in this paper, in order to keep things simpler.
\end{remark}

The measure $\mu_\Delta^o $ depends of course on the choice of the point $o$. However, its class does not. More precisely, one can calculate the Radon-Nikodym derivative as follows. Let $x,y\in X$ be two vertices. For $\omega\in\Omega$, the sectors $Q(x,\omega)$ and $Q(y,\omega)$ intersect, by definition, so let $z\in Q(x,\omega)\cap Q(y,\omega)$. Let $\lambda=\sigma(x,z)$ and $\lambda'=\sigma(y,z)$. Then we define the \emph{horofunction} associated to $\omega$ by $h_\omega(x,y)= \lambda-\lambda'$. It is easily seen that this quantity does not depend on the choice of $z$.

Also recall that, for $\lambda\in\Lambda$, $\ell(\lambda)$ denotes the length of $\lambda$.  The following proposition is proven for example in \cite[Theorem 3.17]{Parkinson} (in the general setting of affine buildings).

\begin{proposition}\label{prop:changebasepoint}
For every vertices $x,y\in X$,
\begin{enumerate}
\item the measures $\mu_\Delta^x$ and $\mu_\Delta^y$ are absolutely continuous with respect to  each other
\item 
for every $\omega\in \Delta$, $$\frac{d\mu_\Delta^x}{d\mu_\Delta^y}(\omega)  = q^{2\ell(h_\omega(x,y))}$$
\end{enumerate}
\end{proposition}

The set $\Delta\times \Delta$ contains a dense open subset $\Delta_\op$ formed by the set of pairs of chambers which are opposite. This set has full $\mu^o_\Delta \times \mu^o_\Delta$-measure (see \cite[Proposition 6.8]{BCL}). In \cite[\S6.2]{BCL} we constructed a $\Aut(X)$-invariant measure on $\Delta\times \Delta$, denoted $m$ there (and that we will denote $m_\op$ here). This measure was given by the following formula:
$$dm(C,C')= q^{-2\ell(\beta_o(C,C'))}d\mu_\Delta^o(C) d\mu_\Delta^o(C')$$
(where $\beta_o(C,C')$ is some kind of Gromov product on $\Delta_\op$, which satisfies in particular that $\beta_o(C,C')=0$  if and only if $o$ is in the apartment containing $C$ and $C'$).
Using \cite[Lemma 6.9]{BCL} one proves that this formula does not depend on the choice of the origin $o$, and that it gives indeed an $\Aut(X)$-invariant Radon measure. Similarly in \cite[Definition 6.14]{BCL} we defined a measure $m_\pm$ on the set $\Delta_\pmop\subset \Delta_-\times\Delta_+$ of opposite vertices, which is in the class of $\mu_-^o \times \mu_+^ o$. The natural map $\pr_-\times \pr_+:\Delta_\op\to \Delta_\pmop$ is measure-preserving.
\index{$m_\op,m_\pm$\hfill invariant measures on $\Delta_\op$ and $\Delta_\pmop$|bb}

We define the \emph{Cartan flow} as $\scrF = \Isom(\Sigma, X)$, \index{$(\scrF,\mu_{\scrF})$ \hfill Cartan flow with its prouniforma measure|bb}equipped with the prouniform measure, denoted $\mu_\scrF$\index{$\mu_\scrF$\hfill prouniform measure on $\scrF$|bb}. This is again a measure which has been considered before, although maybe it is not so well-known (see \cite{Pansu} for a previous construction). The space $\scrF$ is equipped with commuting actions of $\Aut(\Sigma)$ (which is virtually $\mathbb Z^2$) and $\Aut(X)$, which are measure-preserving.

Note however that in \cite{BCL}  we considered a slightly different construction of the measure. Indeed, identifying $\scrF$ with $\Delta_\op\times \Sigma$, we equipped $\scrF$ with the product measure $\zeta_\scrF:=m_\op\times \lambda$, where $\lambda$ is the counting measure on $\Sigma$.

In order to apply the results of \cite{BCL} we are led to checking the following:

\begin{lemma}
The measure $\zeta_\scrF$ on $\scrF$ defined (and denoted $\zeta'$) in \cite[Theorem 6.21]{BCL} is equal to the prouniform measure $\mu_\scrF$.
\end{lemma}

\begin{proof}
Write $\scrF$ as a disjoint union of the sets $\scrF^x$ for $x\in X$, where $\scrF^x$ is the set of embeddings $\phi:\Sigma\to X$ such that $\phi(0)=x$. All the sets $\scrF^x$ are compact subsets of $\scrF$, and by construction of $\mu_\scrF$ we have $\mu_\scrF(\scrF^x)=1$ for every $x$. Using the invariance of $\zeta_\scrF$ by translations we also see that these sets have all the same $\zeta_\scrF$ measure, which is therefore positive.

Therefore, it remains to see that $\mu_\scrF=\zeta_\scrF$ in restriction to $\scrF^x$. To understand these measures, first note that there is a homeomorphism $\Psi:\scrF^x\to\Delta^{x}_\op\subset \Delta^2$ from $\scrF^x$ to the set $\Delta^x_\op$ formed by pairs of  chambers at infinity $(C,C')$ with $C$ opposite $C'$ and $x$ in an apartment containing $C$ and $C'$. For $(C,C')\in \Delta^x_\op$ we have $\beta_x(C,C')=0$, so by construction, the pushforward of $\zeta_{\scrF}$ by $\Psi$ is the restriction of the measure $\mu^x_\Delta\times \mu^x_\Delta$ to $\scrF^x$.

Recall that $\Lambda\subset \Sigma$ is our choice of a sector based at $0$. For $y$ in the interior of  $\Lambda$, denote $\Sigma^y=\Conv(y,-y)$. Note that $\Sigma$ is the union of all the finite convex subsets $\Sigma^y$. For an embedding $\phi^y:\Sigma^y\to X$ such that $\phi^y(0)=x$, denote $\scrF^{\phi^y}=\Isom(\Sigma,X)^{\phi^y}$ the subset of $\scrF$ formed by all $\phi:\Sigma\to X$ such that $\phi_{|\Sigma^y}=\phi^y$. By construction, $\mu_\scrF(\scrF^{\phi^y})$ depends only on $y$ and not on the choice of $\phi^y$. Since $\Sigma$ is the union of all $\Sigma^y$, this characterizes $\mu_\scrF^x$, so that it suffices to prove that $(\zeta_\scrF)_{|\scrF^x}$ satisfies the same property.

Now note that  defining $\phi^y$ amounts to choosing $y_1:=\phi^y(y)$ and  $y_2:=\phi^y(-y)$. Hence $\Psi(\scrF^{\phi^y})$ is the set of pairs of chambers $(C,C')$ in $\Delta^{x}_\op$ such that $C\in \Omega_x(y_1)$ and $C'\in \Omega_x(y_2)$.  Conversely, if $C\in \Omega_x(y_1)$ and $C'\in \Omega_x(y_2)$, then using Theorem \ref{thm:isomappart} one can check that there exists an apartment containing $x,C$ and $C'$, hence $(C,C')\in \Delta_\op^x$. Therefore we have $\Psi(\scrF^{\phi^y})=\Omega_x(y_1)\times \Omega_x(y_2)$, hence $\Psi_*\zeta_\scrF(\scrF^{\phi^y})=\mu^x_\Delta(\Omega_x(y_1))\times\mu_\Delta^x(\Omega_x(y_2))$, which only depends on $\sigma(x,y)$. This concludes the proof.
\end{proof}

\begin{example}
Let $G=\PGL_3(K)$, where $K$ is a local field, and let $X$ be the associated Bruhat-Tits building. Then $G$ acts strongly transitively on $X$, so in particular acts transitively on the set $\scrF$. Hence the Cartan flow is just a $G$-orbit, so it is identified to $M\backslash G$ where $M$ is the stabilizer of an element of $\scrF$, that is, the pointwise stabilizer of some apartment. It is equipped with an action of $G$ (on the right), and an action of the normalizer of $M$ (on the left). As it is clear that the action of $G$ stabilizes the prouniform measure, it follows that (up to normalization) the prouniform measure is the Haar measure on $M\backslash G$.

More concretely, if $F$ is the apartment assoicated to the canonical basis of $K^3$, the pointwise stabilizer of $F$ is the group of diagonal matrices with diagonal entries of norm $\leq 1$. Then the group $A$ of diagonal matrices acts by translations on $F$, and the natural action of $A/M\simeq \ZZ^2$ on $M\backslash G$ is the action by precomposing by translations on $\scrF=\Isom(\Sigma,X)$.
\end{example}

%In \cite[Definition 6.14]{BCL} we defined similarly a measure $m_\pm$ on the set $\Delta_\pmop$ and proved that the map $(\scrF,\mu_\scrF)\to (\Delta_\pmop,m_\pmop)$ is measure-class preserving. Its restriction to the set of $(u,v)\in \Delta_\pmop$ such that $o\in I(u,v)$ will be denoted $m_\pm^o$.
%
%On the other hand, we can define a prouniform measure on the set of $\mathscr L$ of embeddings of a line  in $X$ (as a simplicial embedding, therefore as a singular line). The set $\mathscr L$ clearly maps onto $\Delta_\pmop$, associating to a line its two endpoints. If we choose a singular line in $\Sigma$, then we get a restriction map from $\scrF$ to $\mathscr L$ which is measure-preserving. We denote $\mu^o_\pmop\in \Prob(\Delta_\pmop)$ the pushforward of $\mu_\scrF^o$ by these two maps. % and $\mu_\pmop=\sum_o \mu^o_\pmop$. 
% 
%\begin{corollary}\label{cor:mupmop}
%The measure $\mu_\pmop^o$ is equivalent to the measure $m_\pm^o$. \jl{In fact they are the same. I don't think we use it.}
%\end{corollary}

\subsubsection{Martingales}\label{sec:martingales}

Our goal in this section is to reinterpret the Martingale Convergence Theorem (stated in Theorem \ref{thm:martingales}) and draw consequences in our particular setting.

Fix $o\in X$. For $v\in\Delta_+$, let $\mu_v^o\in \Prob(\Delta)$ \index{$\mu_v^o$\hfill prouniform measure on $\Res(v)$|bb} be the restricted prouniform measure on the subset of $\Isom(\Lambda,X)$ which sends the boundary wall of type $+$ of $\Lambda$ to the ray $[ov)$. Clearly $\mu_v^o$ is supported on $\Res(v)$. Furthermore by Lemma \ref{lem:desintegration_prouniform} the measure $\mu_v^o$ for $v\in\Delta_+$ are also obtained as a disintegration of $\mu_\Delta^o$ along a fiber of the map $\Delta\to\Delta_+$.

Let $\lambda \in \Lambda$. 
Consider the equivalence relation $\sim_\lambda$ on $\Ch(\Delta)$, defined as follows : $C\sim_\lambda C'$ if and only if $Q(o,C)\cap Q(o,C')$ contains a point in $V_\lambda(o)$. Let $\calB_\lambda$ be the $\sigma$-algebra consisting of Borel sets which are $\sim_\lambda$-invariant. Note that if $\lambda'\geq \lambda$ then $\sim_{\lambda'}$ is a refinement of $\sim_\lambda$, and therefore $\calB_{\lambda'} \supset  \calB_\lambda$. For $C\in \Delta$, let $\lambda(C)$ be the point in $V_\lambda(o)\cap Q(o,C)$, and write  $\Omega_\lambda(C)=\Omega_o(\lambda(C))$. If $\lambda$ is on the wall of type $+$ of $\Lambda$, then $\lambda(C)$ only depends on $v=\pr_+(C)$, and we can also write $\Omega_\lambda(C)=\Omega_\lambda(v)$. By definition we have, for $f\in L^\infty(\Delta)$, 

$$\EE(f\mid \calB_\lambda)(C)=\frac{1}{\mu_\Delta^o(\Omega_\lambda(C))}\int_{\Omega_\lambda(C)} f(x)d\mu_\Delta^o(x)$$

Let $\calB_n$ be the subalgebra $\calB_{\lambda_n}$ where $\lambda_n$ is the point on the wall of type $+$ of $\Lambda$ at distance $n$ from $0$.\index{$\lambda_n$ \hfill points on the wall of type $+$ of $\Lambda$|bb} It is clear that $\calB_{n+1}\supset \calB_n$. Let $\calB_\infty$ be the algebra generated by all the $\calB_n$. By construction it is also the algebra obtained as a pullback from the factor $\Delta\to\Delta_+$. Therefore, we get for every $C\in\Ch(\Delta)$, denoting $v=\pr_+(C)$, that for $f\in L^\infty(\Delta)$ 
$$\EE(f\mid \calB_{\infty})(C)= \int_{\Res(v)} f(\xi) d\mu_v^o(C)  $$

For $v\in \Delta_+$, denote $\lambda_n(v)$ the $n$-th vertex on the ray $[ov)$ and note that $\Omega_{\lambda_{n+1}}(v)\subset \Omega_{\lambda_n}(v)$. Let $E_n(v)=\Omega_{\lambda_n}(v)\setminus \Omega_{\lambda_{n+1}}(v)$.

More generally, for $m\geq 0$, let $W_m$ be the half-wall of $\Lambda$ parallel to the wall of type $+$ and starting at distance $m$ from $0$. Let $\lambda_{n,m}$ be the vertex of $W_m$ at distance $n$ from the origin of $W_m$.
%Fixing $o\in X$, for $\xi\in \Delta$, denote $\lambda_{m,n}(\xi)$ the point of $Q(o,\xi)$ such that $\sigma(o,\lambda_{m,n}(\xi))= \lambda_{m,n}$. 

Let $\calB_{m,\infty}$ be the $\sigma$-algebra generated by all the $\{\calB_{\lambda_{m,n}}\}_{n\in\NN}$.
We denote $\Omega_{m,\infty}(\xi)=\bigcap_{n\in\NN}\Omega_{\lambda_{m,n}}(\xi)$ and note that $\Omega_{m,\infty}(\xi)$ can also be seen as the set of all $\xi'$ such that $Q(o,\xi')$ contains the "strip" $\Conv([o,v),\{\lambda_{m,n}(\xi)\}_{n\in \NN})=\Conv([o,v),\lambda_{m,0}(\xi)) $ (where $v=\pr_+(\xi)$, see Figure \ref{fig:strip}). In particular $\Omega_{m,\infty}(\xi) \subset \Res(v)$ , and $\mu_v^o (\Omega_{m,\infty}(\xi))>0$.\index{$\Omega_{m,\infty}(\xi)$\hfill intersection of shadows |bb}

\begin{figure}[h]

\includegraphics[width=12cm]{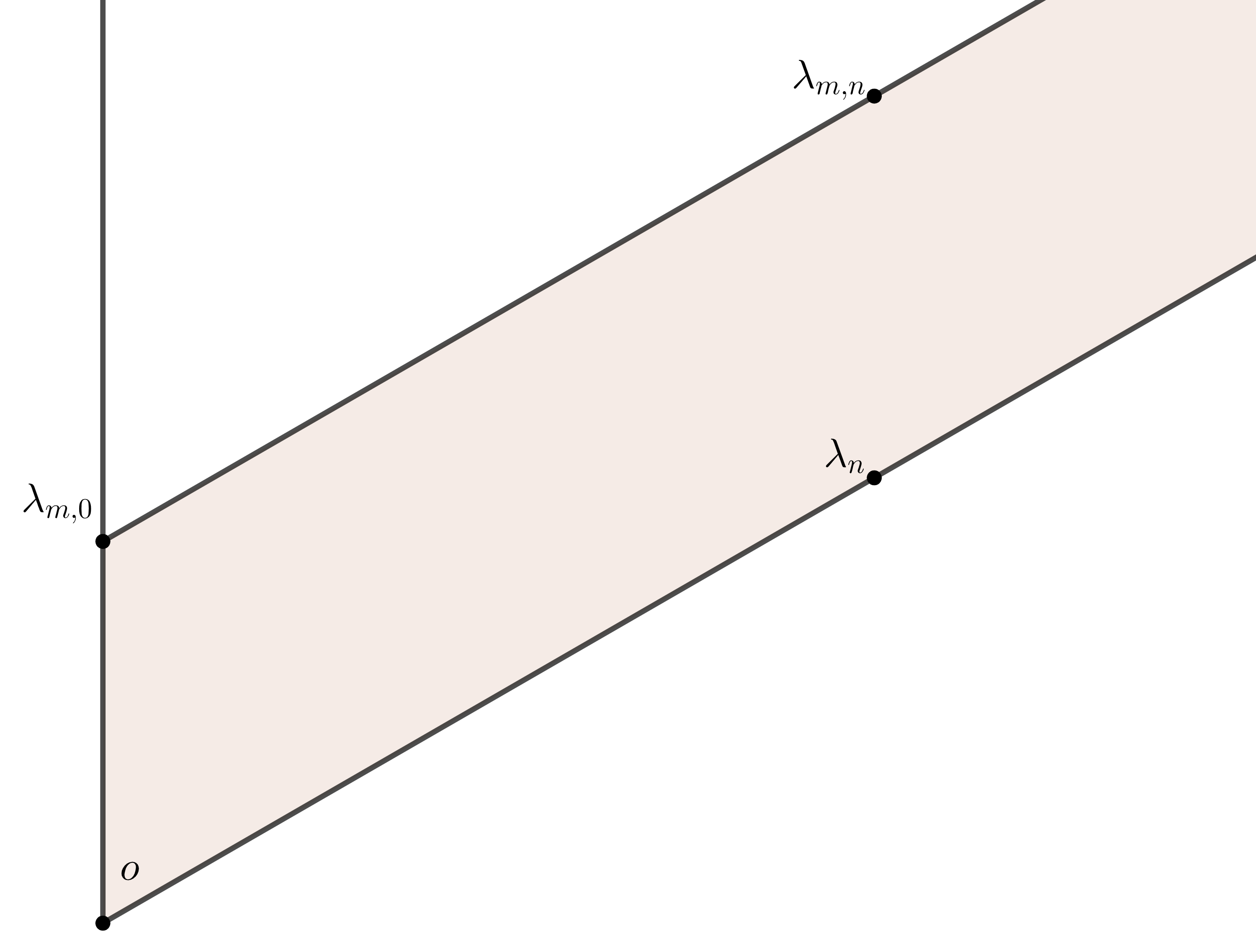}

\caption{The convex hull of $[0,v)$ and $\lambda_{m,n}$}
\label{fig:strip}
\end{figure}

In that setting, Theorem \ref{thm:martingales} reads as follows:

\begin{lemma}\label{lem:martingale}
    Let $f\in L^\infty(\Delta)$. For almost every $\xi\in\Delta$ we have
    $$\lim_{n\to +\infty}\frac{1}{\mu_\Delta^o(\Omega_{\lambda_{m,n}}(\xi))} \int_{\Omega_{\lambda_{m,n}}(\xi)} f(\eta)d\mu_\Delta^o(\eta)=  \frac{1}{\mu_v^o(\Omega_{m,\infty)}(\xi)}\int_{\Omega_{m,\infty}(\xi)} f(\eta) d\mu_v ^o(\eta)$$
\end{lemma}

For $\xi\in \Delta$, let $E_{m,n}(\xi)=\Omega_{\lambda_{m,n}}(\xi))\setminus \Omega_{\lambda_{m,n+1}}(\xi))$. \index{$E_n(v)$,$E_{m,n}(v)$\hfill difference of shadows |bb}
Note that 
\[ \mu_\Delta^o(\Omega_{\lambda_{m,n+1}}(\xi))=\frac{1}{q^2}\mu_\Delta^o(\Omega_{\lambda_{m,n}}(\xi))). \]
Indeed, if $\eta$ is such that $\lambda_{m,n}(\xi)\in Q(o,\eta)$ then there are $q^2$ possibilities for the choice of $\lambda_{m,n+1}(\eta)$.

\begin{lemma}\label{lem:MartingaleEnm}

    Let $f\in L^\infty(\Delta)$. For every $m\in\NN$ and $\mu_\Delta^o$-almost every $\xi$ we have
   
    $$\lim_{n\to +\infty} \frac{1}{\mu_\Delta^o(E_{m,n}(\xi))}\int_{E_{m,n}(\xi)} f(\eta)d\mu_\Delta^o(\eta) 
    = \frac{1}{\mu_u^o(\Omega_{m,\infty}(\xi))}\int_{\Omega_{m,\infty}(\xi)} f(\eta) d\mu_v^o(\eta)$$

    In particular, 
    for $\mu_+^o$-almost every $v$ we have

$$\lim_{n\to +\infty} \frac{1}{\mu_\Delta^o(E_n(v))}\int_{E_n(v)} f(C)d\mu_\Delta^o(C)=\int_{\Res(v)} f(C) d\mu^o_{v}(C)$$
\end{lemma}

\begin{proof}
The second statement follows from the first by taking $m=0$ (using also the fact that $\pr:\Delta\to\Delta_+$ is measure-preserving).

So let us prove the first statement.
Fix $\xi$ such that Lemma \ref{lem:martingale} holds.
Write $\Omega_{m,n}=\Omega_{\lambda_{m,n}}(\xi)$ and $E_{m,n}=E_{m,n}(\xi)$. 
   By the above estimate we have $\mu(E_{m,n})=(1-\frac{1}{q^2})\mu(\Omega_{m,n})$.

    It follows that
     \begin{align*}
         \frac{1}{\mu_\Delta^o(E_{m,n})}\int_{E_{m,n}} f(C)d\mu_\Delta^o(C) &= \frac{q^2}{q^2-1}\cdot  \frac{1}{\mu_\Delta^o (\Omega_{m,n})} 
         \int_{E_n} f(C)d\mu_\Delta^o(C)\\
         &=\frac{q^2}{q^2-1}\cdot  \frac{1}{\mu_\Delta^o (\Omega_{m,n})} \left(
         \int_{\Omega_{m,n}} f(C)d\mu_\Delta^o(C)-\int_{\Omega_{m,n+1}} f(C)d\mu_\Delta^o(C)\right)\\
         &= 
             \frac{q^2}{q^2-1}\cdot \frac{1}{\mu_\Delta^o (\Omega_{m,n})}\int_{\Omega_{m,n}} f(C)d\mu_\Delta^o(C) 
             \\  & \qquad \qquad- \frac{1}{q^2-1}\cdot \frac{1}{\mu_{\Delta}^o(\Omega_{m,n+1})}\int_{\Omega_{m,n+1}} f(C)d\mu_\Delta^o(C)
     \end{align*}
     which converges by assumption to $$(\frac{q^2}{q^2-1}-\frac{1}{q^2-1})\frac 1{\mu_v^o(\Omega_{m,\infty}(\xi))}\int_{\Omega_{m,\infty}(\xi)} f(\eta) d\mu_u ^o(\eta)=\frac 1{\mu_u^o(\Omega_{m,\infty}(\xi))}\int_{\Omega_{m,\infty}(\xi)} f(\eta) d\mu_v ^o(\eta).$$
\end{proof}

Finally, let us record the following remark.

\begin{lemma}\label{lem:Enmintersection}
    Let $\xi\in \Delta$ and $v=\pr_+(\xi)$. We have $\Omega_{m,n}(\xi)\cap E_n(v)= E_{m,n}(\xi)$.
\end{lemma}

\begin{proof}
    If $x\in \Omega_{m,n+1}(\xi)$ then $\Conv(o,\lambda_{m,n+1}(\xi))\subset Q(o,x)$ and therefore $\lambda_{n+1}(v)\in Q(o,x)$, so that $x\not \in E_n$. Hence we have the inclusion $\Omega_{m,n}(\xi)\cap E_n(v)\subset E_{m,n}(\xi)$.

    Conversely, if $x\in E_{m,n}(\xi)$ then $x\in \Omega_{m,n}(\xi)\supset \Omega_n(u)$ by definition, and as $\lambda_{m,n+1}(\xi)\in \Conv(\lambda_{n+1}(v),\lambda_{m,n}(\xi))$ we have $x\not\in \Omega_{n+1}(v)$. 
\end{proof}

\subsubsection{Singular flow}

Recall our definition of the model wall-tree $\Upsilon$, which is isometric to $T\times\RR$ (where $T$ is a $(q+1)$-regular tree).
We also define the space of \emph{marked wall trees} as $\tilde \scrW = \Isom(\Upsilon,X)$. Note that $\Aut(\Upsilon)$ acts on $\tilde\scrW$ (by precomposition), and $\Aut(\Upsilon)$ contains the subgroup $S_\Upsilon$ isomorphic to $\ZZ$ (acting by translation on the factor $\RR$ of the decomposition $T\times \RR$). Furthermore, we have a canonical projection $p_T:\Aut(\Upsilon)\to \Aut(T)$, which is surjective. An element of $\tilde\scrW$ has two \emph{endpoints} in $\Delta_{\pmop}$, which we define as the image of the two endpoints of a line in the $\ZZ$ factor of $\Upsilon$.

 Unfortunately, it turns out that the prouniform measure on this space is not adapted to what we can do ; in fact we need to consider a smaller subset. 
The \emph{singular flow} $\scrW$ is a subset of $\tilde \scrW$, which is defined as one (arbitrary) equivalence class of $\tilde \scrW$, for an equivalence relation defined in \cite[Definition 5.17]{BCL}. Informally, one can think of $\scrW$ as the closure of some $\Gamma\times p_T^{-1}(\Pi)$-orbit. 
(Recall that $\Pi$ is the closure of the projectivity group, as indicated in the list of notations.) 

More precisely, identifying $\Upsilon$ with $T\times \RR$, we say that $y,y'\in \Isom(\Upsilon,X)$ are $\sim_+$-equivalent if there is $T>0$ such that for $t>T$, for every $x\in T$ we have $y(x,t)=y'(x,t)$. Similarly we can define the $\sim_-$ equivalence. We define the final equivalence relation $\simeq$ on $\tilde \scrW$ as the equivalence relation generated by the two equivalence relations $\sim_-$, $\sim_+$, and the $S_\Upsilon\times \Gamma$-orbit relation.

This construction is easily seen to be related to the definition of perspectivities. More precisely, if $y\in \tilde W$, then $y$ has two endpoints, let us call them $(u,v)\in\Delta_{\pmop}$. If $u'\in \Delta_-$ is opposite $v$ then there is a unique $y'$ such that $y\sim_+ y'$ and $y'$ has endpoints $u'$ and $v$, which can be seen as the perspectivity $[u,v,u']$ applied to $y$. 
Hence  if $y,y'\in \tilde\scrW$ are equivalent by a sequence of consecutive $\sim_+$ and $\sim_-$ relations, that is, $y\sim_+ y_1\sim_- y_2\dots \sim_+y'$ for some $y_1,y_2,\dots$, then they differ by some sequence of perspectivities. In particular, if $y$ and $y'$ furthermore have the same endpoints, then they differ by an element of the projectivity group.

Now we fix an arbitrary $y_0\in\tilde \scrW$, and declare the \emph{singular flow} to be the $\simeq$-equivalence class of $y_0$. (The rest of the argument does not depend on the choice of $y_0$, as another choice would give a dynamical system conjugated to that one).

\medskip 

In particular, by construction, the space $\scrW$ is
\begin{itemize}
\item Invariant by the action of $S_\Upsilon$,
\item Invariant by the action of $p_T^{-1}(\Pi)$, 
\item Invariant by $\Gamma$.
\end{itemize}

\medskip

Note that there is a natural map $\pi:\scrW\to \Delta_\pmop$, associating to an embedding the image of the two endpoints of the $\RR$ factor. By construction, for every $(u,v)\in \Delta_{\pmop}$ the fiber $\pi^{-1}(u,v)$ is in bijection with a closed group $\tilde \Pi<\Aut(T)$ which contains $\Pi$.

 Lemma 5.20 of  \cite{BCL} provides the following basic topological  properties of $\scrW$:

\begin{proposition}\label{prop:singular_fibers}
    The space $\scrW$ is closed in $\tilde{\scrW}$, and is $\tilde\Pi \times \Gamma$-invariant. The group $\tilde\Pi$ is a closed subgroup of $\Aut(\Upsilon)$. 

The action of $\Gamma$ on $\mathscr W$ is proper and cocompact and the $\tilde \Pi$ action is proper and free.
\end{proposition}

\medskip 

Using the Proposition, and the map $\pi:\scrW\to \Delta_\pmop$, we can identify $\scrW$ to $\Delta_\pmop\times\tilde \Pi$. Under the identification, we define a measure $\zeta_{\scrW}$ as the product measure of $m_\pm$ by the Haar measure on $\tilde \Pi$ (see \cite[Theorem 6.21]{BCL} for more details). %In the classical case, 

In the end, the important features of $\zeta_{\scrW}$ that we will use are that:

\begin{itemize}
\item The map $\pi$ is measure-class preserving 
\item The disintegration of $\zeta$ along the map $\pi$ is the Haar measure on $\tilde \Pi$.
\end{itemize}\index{$(\scrW,\zeta_{\scrW}$) \hfill singular flow |bb}

 More generally, the measure $\zeta_\scrW$ is also such that all the natural maps given in the commutative diagram \cite[5.4 (3)]{BCL} are measure-class preserving.
 
\begin{remark}
Let us emphasize that the measure $\zeta_\scrW$ is \emph{not} in general the prouniform measure on $\Isom(\Upsilon,X)$ ! The reason we need a different measure is because the measure $\zeta_\scrW$ will be ergodic under the action of $\Gamma\times S_\Upsilon$, while the measure $\mu_\scrW$ may not. This will be crucial in our proof of Proposition \ref{prop:Poincare2}.
\end{remark}

\begin{example}
In the case when $X$ is the Bruhat-Tits building of $G=\PGL_3(K)$, for $K$ a local field, this seemingly weird equivalence relation is needed for the following reason: our goal is to retrieve only one $G$-orbit in the space $\Isom(\Upsilon,X)$. 
Thus in the classical case the singular flow is in fact contained in
$M\backslash G$, where $M$ is the pointwise stabilizer of a wall-tree. 

To see this, note that the $\sim_+$ and $\sim_-$ equivalence are generated by unipotent elements (in the sense that if $\phi,\phi'\in\Isom(\Upsilon,X)$ then there exists an unipotent $u\in G$ such that $u\phi=\phi'$), and the action by $S_\Upsilon$ is also given by a conjugate of some diagonal matrices in $G$. Hence the singular flow is contained in a $G$-orbit. 

Conversely, we can identify $\Delta_\pmop$ to $G/Q_1\times G/Q_2$, where $Q_1$ and $Q_2$ are two opposite maximal parabolic subgroups. This intersection is measurably isomorphic to $G/L$ where $L=Q_1\cap Q_2$ is a Levi subgroup. The map $\scrW\to G/Q_1\times G/Q_2$ is measure-class preserving (when $G/Q_1\times G/Q_2=\Delta_\pmop$ is equipped with the Haar measure class, which is equivalent to $m_\pm$). Furthermore, the fibers are identified with the orbit of a group which contains $p_T^{-1}(\Pi)$. Since this group is identified the Levi subgroup $L$, we see that the fibers are also  $L$-invariant. Hence the measure $\zeta_\scrW$ on $\scrW$ is equivalent to the Haar measure on $G/M$ (and similarly $\scrW$ is $G$-equivariantly homeomorphic to $G/M$).

In formulas, fixing a basis $(e_1,e_2,e_3)$ of $K^3$, and taking the wall-tree associated to the two points at infinity defined by the plane whose basis is $(e_1,e_2)$ and the line defined by $e_3$, we see that we have
$$M=\left\{
\begin{pmatrix}
    u& 0&0\\
    0&u&0\\
    0&0&v
\end{pmatrix}\mid u,v \textrm{ of norm $\leq 1$}\right\}$$ 

\begin{comment}
    Proof :  $M$ stabilizes the basis (e_1,e_2,e_3) so is formed by diagonal matrices. In fact it must stabilize any basis (e'_1,e'_2,e_3) when vect(e_1,e_2)= vect(e'_1,e'_2), so it is formed of matrices with the first two diagonal elements equal. Since it fixes the "origin" its element have norm \leq 1. Now the stabilizer of another point is the conjugate of the origin by a diagonal matrix (in the adequate basis) ; since these commute with such elements the announced elements fix all points of the apartment, and therefore of the wall-tree.
    
\end{comment}
 
 The space $M\backslash G$ is
equipped with an action of $G$ (hence $\Gamma$) on the right, and on the left by the group (which normalizes $M$) of matrices of the form
$$\begin{pmatrix}
    A & 0\\0 & \lambda
\end{pmatrix}$$
for some $A\in\GL_2(K)$ and $\lambda\in K\setminus\{0\}$. This is consistent with the identification of the projectivity group with $\PGL_2(K)$ (see \cite{Knarr}).

\end{example}

\section{Approximation of the singular flow}

In this section, we use the singular flow $\scrW$ defined in Section \ref{sec:flows} to construct sequences of singular hyperbolic elements, with a prescribed transverse action. We then prove that this sequence of elements have a nice dynamical action on the boundary $\Delta$.

\subsection{The argument for the tree}

Before diving into the details of the proof, we will start by describing the full argument in an easier case, namely, the case of a tree. Many elements of the proof are similar to the case of the building, so we will not give full details,
as our main purpose here is to illustrate the ideas. Later, when we will deal with the building case, we will be more cautious with the details.

Let $T$ be a locally finite regular tree and $\partial T$ its boundary. Let $\Gamma$ be a group acting properly and cocompactly on $T$. We endow $\partial T$ with the restricted prouniform measure $\mu_{\partial T}^o$ on $\Isom([0,+\infty),T)$ (where $o\in T$ is an arbitrary origin).

The main goal in this section is the proof of the following.

\begin{proposition}\label{prop:tree}
    For every $f\in L^\infty(\partial T,\mu_{\partial T}^o)$, and almost every $\xi\in \partial T$, there exists a sequence $(\gamma_n)\in \Gamma^\NN$ such that $\gamma_n\cdot f$ converges to (the constant function) $f(\xi)$ in weak-* topology.
\end{proposition}

\subsubsection{A recurrence argument}

Consider the geodesic flow $\mathscr G$ on the tree $T$: it is the set of parametrized geodesics on $T$, equipped with the prouniform measure $\mu_{\mathscr G}$ and the action of the shift $S_{\mathscr G}$. It is classical, and not hard to see, that the action of $\Gamma$ on $\mathscr G/S_{\mathscr G}$ is ergodic, or equivalently that the action of $S_{\mathscr G}$ on $\Gamma\backslash \mathscr G$ is ergodic. 

The first step of our proof is to find a contracting sequence in every direction in $\Gamma$. Note that since $\mu_{\partial T}^o$ is non-atomic, $(\mu_{\partial T}^o)^2$-almost surely two points $(\eta,\xi)\in\partial T\times \partial T$ are distinct, hence linked by a geodesic (denoted $(\eta,\xi)$).

Say that a sequence $(\gamma_n)$ of automorphisms of $T$ is $(\eta,\xi)$-contracting if for every finite segment $I$ contained in the geodesic $(\eta,\xi)$, for $n$ large enough $\gamma_n$ acts on $I$ by translating it towards $\xi$ by a length at least $n$.

\begin{lemma}\label{lem:treesec}
    For $(\mu_{\partial T}^o)^2$-almost every $(\eta,\xi)\in\partial T\times \partial T$ there exists 
    a $(\eta,\xi$)-contracting sequence $(\gamma_n)_{n\in\NN}$ of elements of $\Gamma$.
\end{lemma}

The analogous statement of the building is proved in Proposition \ref{prop:Poincare} below. The argument is roughly the same in both cases, except that Proposition \ref{prop:Poincare} uses the singular flow on the building instead of the geodesic flow on the tree.

\begin{proof}
    Let $I$ be a finite segment. Then the shift $S_{\mathscr G}$ on $\Gamma\backslash \mathscr G$ preserves (by cocompactness of $\Gamma$) a probability measure in the class of the pushforward of $\mu_{\mathscr G}$. Hence it is recurrent. If $I$ is a finite segment, then the set of geodesics containing $I$ has positive measure, thus is $S_{\mathscr G}$-recurrent, which means that there is a sequence $\gamma_n$ such that $\gamma_n S_{\mathscr G}^{-N}$ fixes $I$ (for some $N>n$).
    
    This almost gives the sequence $\gamma_n$ that we want, except that we do not have any control over the endpoints $\xi$ and $\eta$. To get the Lemma for almost every $\xi$ and $\eta$ one needs to use the ergodicity of $S_{\mathscr G}$. 
\end{proof}

Note that for such a sequence $(\gamma_n)$ we have that $\gamma_n\xi'$ converges to $\xi$, for every point $\xi'\neq \eta$. In particular if $f$ is continuous on $\partial T$ then $\gamma_n f$ converges to $f(\xi)$ pointwise. Of course it is far from sufficient to get Proposition \ref{prop:tree} as we only have a measurable function which is only almost everywhere defined.

\subsubsection{Martingales and Lebesgue differentiation}

The second step of our argument is to apply some sort of disintegration. This is the analogous of what is done in \S\ref{sec:martingales} above.

 For $o\in T$ fixed, and $x\in T$ denote by $\Omega(x)\subset\partial T$ the shadow of $x$ seen from $o$, that is, the set of endpoints $\xi\in\partial T$ for which $x$ is on the geodesic ray $[o\xi)$. For
 $n\in \NN$, $\xi\in \partial T$,  denote by $\Omega_n(\xi)$ the shadow of the $n$th vertex of $[o\xi)$.

 The sets $\Omega_n(\xi)$ can be seen to be balls around $\xi$, for a natural metric on $\partial T$. Hence the following is a direct application of the Lebesgue Differentiation Theorem. In that case, one can also easily apply the Martingale Convergence Theorem to get the same result.
 
 \begin{lemma}\label{lem:treemart}
For every $f\in L^\infty(\partial T,\mu_{\partial T}^o)$, for $\mu_{\partial T}^o$-almost every point $\xi$ we have
$$\lim_n \frac{1}{\mu_{\partial T}^o(\Omega_n(\xi))}\int_{\Omega_n(\xi)} f d\mu_{\partial T}^o = f(\xi)$$
 \end{lemma}   

In fact, using a classical trick (compare for example with \cite[p.98]{Folland}) one can prove the same convergence when averaging over a subset of $\Omega_n(\xi)$ which has a measure equal to a fixed (independent of $n$) proportion of $\Omega_n(\xi)$ (for the measure $\mu_{\partial T}^o$). Furthermore one can easily check that  the "annulus" $E_n(\xi):=\Omega_n(\xi)\setminus \Omega_{n+1}(\xi)$ has positive measure, and in fact takes a fixed proportion of the ball $\Omega_n(\xi)$. Hence the Lemma also holds when replacing the average of $f$ over $\Omega_n(\xi)$ with its average over $E_n(\xi)$.

 An important difficulty (and also a crucial point of the proof) in the building case is that the analogous limit that we get will not be a constant anymore, but a function on a smaller set.

\subsubsection{The detecting flow}

The final step in the proof of Proposition \ref{prop:tree}
is to relate the sequence $(\gamma_n)$ to the averages over $E_n(\xi)$. While this can also be done directly in the case of the tree, we found it easier in the building case to use yet another space, which we call the \emph{detecting flow}.

We define a space $M$, which we think of as a {\em model geodesic endowed with antennas}, or a \emph{comb}, by setting
\[ M=\left(\RR\times \{0\}\right) \cup \left(\ZZ\times [0,\infty)\right) \subset \RR^2, \]
which we view as the union of a geodesic ($\RR\times \{0\}$) with antennas at each vertex (the lines $\{n\}\times [0,\infty)$) (see Figure \ref{fig:detectingFlowTree} below).
We define the {\em detecting geodesic flow} on $T$ to be
\[ \JR=\mathrm{Isom}(M,T). \]

\begin{figure}[h]
\includegraphics[width=15cm]{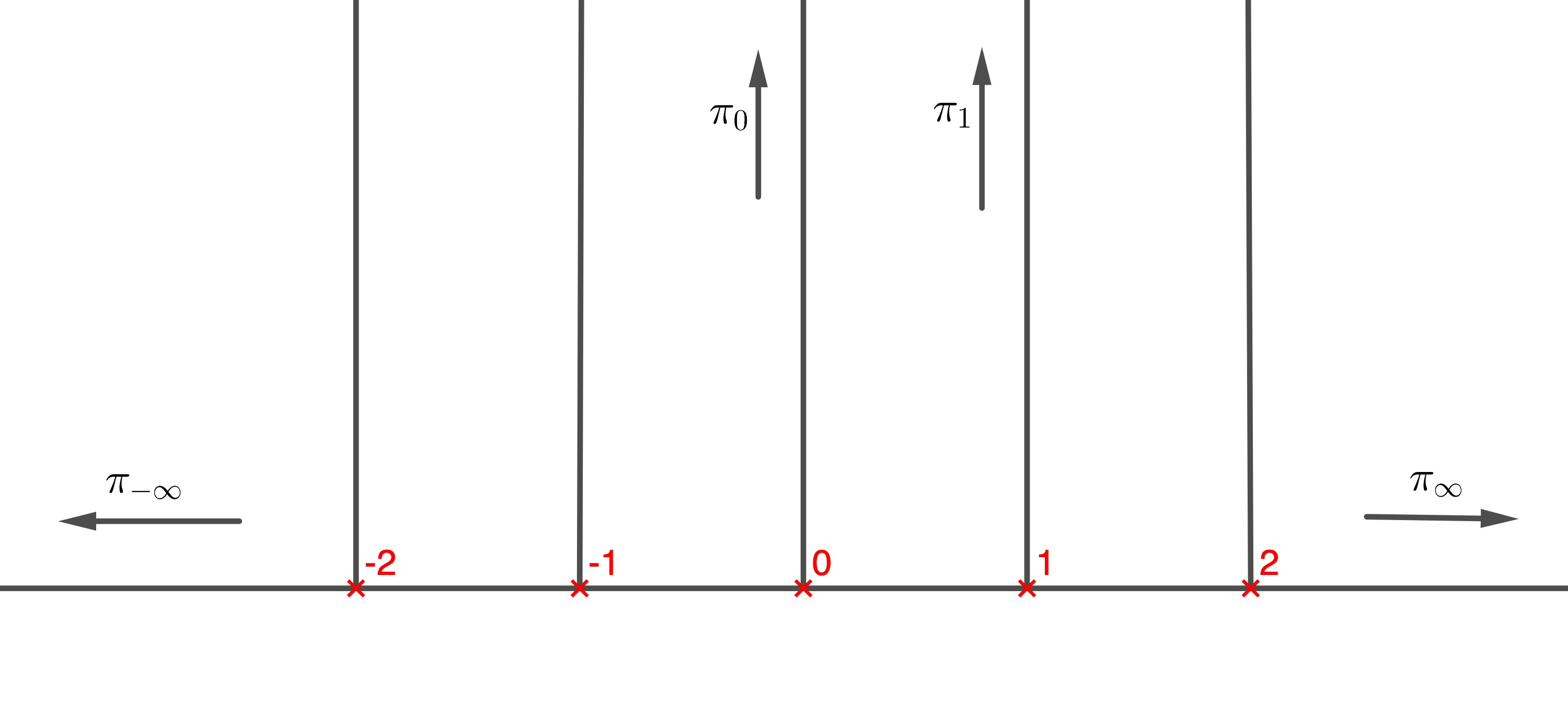}
\caption{The model space for the detecting flow on a tree}
\label{fig:detectingFlowTree}
\end{figure}

We endow $\JR$ with the prouniform measure $\mu_\JR$.
For every $n\in \ZZ\cup \{\pm\infty\}$ we define a map $\pi_n:\JR\to \partial T$ as the limit of the image of the ray $\{n\}\times [0,\infty)$ (these are the "antenna" maps).
On the other hand, we can also define $\pi_\infty,\pi_{-\infty}:\JR\to \partial T$ as the limit of the geodesic rays $[0,+\infty)\times \{0\}$ and $(-\infty,0]\times \{0\}$. 

We define the "shift" on $\RR^2$ by $S(x,y)=(x+1,y)$ and note it leaves $M$ invariant, hence it induces an automorphism of $\JR$, which commutes with the action of $\Gamma$. Furthermore we have easily that $\pi_n\circ S=\pi_{n+1}$. 

The maps $\pi_n$ (for $n\in \ZZ\cup\{\pm\infty\}$) define pullbacks $\pi_n^*:L^\infty(\partial T)\to L^\infty(\JR)$. (Note that $\pi_n:\JR\to \partial T$ is measure-class preserving). However, since our favorite contracting sequence $(\gamma_n)$ depends on a choice of points $\xi$ and $\eta$, we need to first proceed to a disintegration and fix the two endpoints.  More precisely, for a pair of distinct points $\xi,\eta\in\partial T$ and $o$ on the geodesic between $\xi$ and $\eta$, let $\mu_{\xi,\eta}^o$ be the restricted prouniform measure on the set $\JR_{\xi,\eta}^o\subset \JR$ consisting of all $j$ which send $0$ to $o$ and the horizontal  line in $M$ to the geodesic line $(\xi,\eta)$.

Note that $\pi_n:\JR_{\xi,\eta}^o\to \partial T$ is not well-defined anymore. However, we have:

\begin{lemma}
    For $n>0$, the measure $(\pi_n)_*\mu_{\xi,\eta}^o$ is the normalized restriction of $\mu_{\partial T}^o$ to $E_n(\xi)$. 
\end{lemma}

Now, in order to prove Proposition \ref{prop:tree}, we note that if $(\gamma_n)$ is a sequence as in Lemma \ref{lem:treesec}, then (for $n$ large enough) $\gamma_n$ acts as a translation of length say $N$ on a large segment between $\xi$ and $\eta$. Of course it is not correct that $\gamma_n \pi_0 (j)=\pi_N(j)$, as we do not know what happens on the "branches". Nevertheless the general idea is that this will be true on average on $j$.

It turns out that this is sufficient to finish the proof of Proposition \ref{prop:tree}. 
Using Lemma \ref{lem:treemart} we deduce that $\pi_n^* f$ converges to $f(\xi)$ in the weak-* topology of  $L^\infty(\JR_{\xi,\eta}^o)$. Indeed by density (Lemma \ref{lem:densityL1_gen}) it is sufficient to prove the convergence of $\langle (\pi_N)^* f,\phi \rangle$ to $\langle f(\xi),\varphi\rangle$ when $\varphi$ is the characteristic function of a given embedding of a finite subset  $F\subset M$. But then for $N$ large enough we get that $(\pi_N)_*\varphi$ is constant equal to 1, so that we are precisely reduced to the limit calculated in Lemma \ref{lem:treemart}.

To compare $\langle (\pi_n)^* f,\phi \rangle$ with $\langle (\pi_n)^* f,\varphi \rangle$ the idea is to construct from the embedding of $F$ a new embedding, of an infinite convex subset of $M$, for which the branch at the point $N$ is given by the image of the branch at $0$ of $F$ by $\gamma_n$. Then if $\varphi_0$ is the characteristic function of this embedding we check that $$\langle (\pi_N)^* f,\varphi_0 \rangle=\langle \gamma_n \pi_0^*f,\varphi\rangle$$

Note that having this for one fixed $o$ is not sufficient, as it only give weak-* convergence on $L^\infty(\Delta,(\pi_0)_*\mu_{\xi,\eta}^o)$. However since this is valid for every $o$ we get the convergence with respect to the right measure class.

\subsection{Singular flow and contracting sequences}

The first step in our proof is to construct sequences of elements of $\Gamma$ with a prescribed dynamic on the boundary. More precisely, our goal is to prove that there is a sequence of elements contracting almost all of $\Delta$ onto a fixed residue (Proposition \ref{prop:Poincare}). In fact, it will be later important to us to also prove that it is possible to do so while also acting as a given projectivity (Proposition \ref{prop:Poincare2}). 

\begin{proposition}\label{prop:Poincare}
For almost every $(u,v)\in \Delta_\pmop$, there exists a sequence $(\gamma_n)\in \Gamma^\NN$ such that for  every $C\in \Delta$ which is opposite $u$, the sequence $\gamma_n C$ converges to the chamber $\proj_v(\proj_u(C))$.
\end{proposition}

The main tool in the proof of this Proposition is the flow $\scrW$ constructed in \cite{BCL}, whose main properties we recalled in Section \ref{sec:flows}. In particular, $\scrW$ is equipped with a measure $\zeta_\scrW$.

\begin{lemma}
For every $(u,v)\in \Delta_\pmop$ and any finite subset $F\subset I(u,v)$, the set $E_F:=\{\phi\in \scrW \mid F\subset \Im(\phi)\}$ has positive $\zeta_\scrW$-measure.
\end{lemma}

\begin{proof}
Recall the map $\pi:\scrW\to \Delta_\pmop$, associating to an embedding in $\scrW$ its two endpoints. By construction of $\zeta_\scrW$ this map is measure-preserving. Also note that $F\subset \Im(\phi)$ is equivalent to $F\subset I(\pi(\phi))$, as $\Im(\phi)$ is in fact equal to $I(\pi(\phi))$ by construction of $\Upsilon$. 
Therefore we are led to proving that the set $E'_F=\{(u',v')\in \Delta_\pmop\mid F\subset I(u',v')\}$ has a positive $(\mu_-^o\times \mu_+^o)$-measure.

Now, using Theorem \ref{thm:Psymetric}, we can define the prouniform measure $\mu_{\tilde\scrW}$ on the set $\tilde\scrW =\Isom(\Upsilon,X)$. We can still define a map $\tilde\scrW\to \Delta_\pmop$ in the same way, and still denote it $\pi$. By Corollary \ref{cor:measurepreserving} we get that $\pi$ is measure-class preserving. On the other hand, by construction, it is clear that the set $\tilde E_F=\{\phi\in \tilde\scrW \mid F\subset \Im(\phi)\}$ has positive $\mu_{\tilde\scrW}$-measure. Hence $\pi(\tilde E_F)$ has positive $(\mu_-^o\times \mu_+^o)$-measure. But since $\pi(\tilde E_F)=E'_F$ this concludes the proof.

%Note that if the Lemma holds for $F$, and $F'\subset F$, then the Lemma holds for $F'$. Hence it suffices to prove the Lemma for a set $F$ of the form $T(u,v)\times F_0$, where $F_0$ is a finite interval of $\RR$. Note that if $F_0$ is indeed contained in the image of $\phi$, then by definition of $\Upsilon$ we get that if $\ell$ is a segment parallel to $F_0$ (in some apartment), and the convex hull of $\ell$ and $F_0$ is a rectangle, then $\ell$ is also in the image of $\phi$. Hence we obtain that the full set $F$ is in the image of $\phi$.

%It follows that it suffices to prove the Lemma for $F$ a finite geodesic segment between $u$ and $v$. 
%Now we know that the image of $\zeta_W$ by the natural map $\scrW\to \Delta_\pmop$ is $m_\pm$. But the set $\{(u,v)\mid F \text{ is a geodesic segment between } u \text{ and } v\}$ is of positive $m_\pm$-measure, as $m_\pm$ is in the same class as the product of the prouniform measures on $\Delta_+$ and $\Delta_-$. The result follows.
\end{proof}

\begin{lemma}\label{lem:construction_contracting}
For $m_\pm$-almost every $(u,v)\in \Delta_\pmop$, for every finite $F\subset I(u,v)$ and $N\in \NN$, there exists $n>N$ and $\gamma\in \Gamma$ such that $\gamma$ acts on $F$ as a translation of length $n$ in the direction of $u$.
\end{lemma}

\begin{proof}
Let $E_F$ be the set of $\phi\in\scrW$ such that $F\subset \Im(\phi)$.
Let $\phi_0\in E_F$, and $B=\phi_0^{-1}(F)$. Let $E_{0}=\{\phi\mid \phi_{|B}=\phi_{0|B}\}$. Since there are countably many possible embeddings of $B$ into $X$, up to changing $\phi_0$, we may and shall assume that $\zeta_{\scrW}(E_0)>0$. Let $E$ be the image of $E_{0}$ in $\scrW/\Gamma$. Note that $\scrW/\Gamma$ is equipped with a measure $\zeta_{\scrW/\Gamma}$ such that the quotient map $\scrW\to\scrW/\Gamma$ is measure-class preserving.

Recall that $\scrW$ is equipped with an action of a group $S_\Upsilon$ translating the $\RR$ part of $\Upsilon$, and that this action commutes with $\Gamma$, hence descends to an action on $\scrW/\Gamma$. Let $s\in S_\Upsilon$ be the generator, acting by translating the lines in the direction of the vertex of type $+$. 
For $N>0$, let $F_N=\{\varphi\in\scrW \mid \exists n\geq N\; \exists \gamma\in \Gamma \; \gamma\circ \varphi\circ s^n \in E_0\}$. Note that $F_N$ is $\Gamma$-invariant by construction ; we still denote by $F_N$ its image in $\scrW/\Gamma$. We claim that $F_N$ has full measure. Indeed, assume by contradiction that its complement $F_N^c$ has positive measure.
As the action of $S_\Upsilon$ on $\scrW/\Gamma$ is ergodic \cite[Theorem 7.3]{BCL}, there exists a $k>N$ such that $s^kF_N^c\cap E\neq \varnothing$. Therefore there exists $\varphi\in E$ and $\varphi'\in F_N^c$ with $\varphi' s^k = \varphi$. In other words, $\varphi' s^k \in E_0$: this is a contradiction, and therefore $F_N$ has full measure. Therefore the set $F=\bigcap_{N\in\NN} F_N$ has full measure in $\scrW$. The set of all $(u,v)\in\Delta_\pmop$ satisfying the condition of the Lemma is the image of $F$ by the map $\pi$.
Since this map is measure-class preserving, we get the result.

%As the action of $S$ on $\scrW/\Gamma$ is ergodic \cite[Theorem 7.3]{BCL} for almost every $\phi\in E$,  the sequence $\frac{1}{n}\sum_{k=0}^{n-1}\mathbb{1}_E(s^{-n}\phi)$ converges $\zeta_{W/\Gamma}(E)>0$. In particular for every $N$ there exists $n>N$ such that $s^{-n}\phi\in E$, that is, there exists $\gamma\in \Gamma$ such that $\gamma^{-1}s^{-n}\phi\in E_0$. Now let $x\in F$, and $y=\phi^{-1}(x)$, so that $y\in B$. Thus we have  $s^{-n} \phi(x) =\gamma \phi(x)$, i.e., $s^{-n}y = \gamma y$. This means that every point of $F$ is translated by $\gamma$ by some  length $n$ in the direction of $u$.
\end{proof}

\begin{proof}[Proof of Proposition \ref{prop:Poincare}]
Fix $(u,v)\in \Delta_\pmop$ such that Lemma \ref{lem:construction_contracting} holds for $(u,v)$. Let $(F_n)$ be an increasing sequence of finite sets exhausting $I(u,v)$. Fix $\gamma_n\in \Gamma$ such that $\gamma_n^{-1}$ translates $F_n$ by a length at least $n$ in the direction of $u$.

For $\mu^o_\Delta$-almost every $C$, the chamber $C$ is opposite to $u$. Fix such a $C$. Let $A_0$ be an apartment containing $C$ and $u$, and $A_1$ be an apartment containing $\proj_u C$ and $v$. Note that $A_0$ and $A_1$ both contain a sector pointing towards $\proj_u (C)$. Let $o\in A_0\cap A_1$ and $C'=\proj_v(\proj_u C)$. Let $z\in Q(o,C')$, note that $z\in A_1$. To prove the desired convergence, it is enough to prove that for every $n$ large enough we have $z\in Q(o,\gamma_n C)$, that is, $\gamma_n^{-1} z\in Q(\gamma_n^{-1} o,C)$.

Note that $A_1\subset I(u,v)$, so that $o,z\in I(u,v)$. Therefore for $n$ large enough we have $o,z\in F_n$, and $\gamma_n^{-1}$ will translate $o,z$ in $A_1$ towards $u$. For $n$ large enough it follows that $\gamma_n^{-1} z\in Q(o,\proj_u C)$, since $\proj_u C'=\proj_u C$ by definition. Let $\rho:A_0\to A_1$ be the unique isomorphism fixing $A_0\cap A_1$ pointwise. Since $z\in Q(o,C')$, and  $\gamma_n^{-1}$ translates $o$ and $z$ in $A_1$, we have $\gamma_n^{-1} z \in Q(\gamma^{-1}_n o, C')$. Therefore $\rho(\gamma_n^{-1} z)\in Q(\rho(\gamma_n^{-1}o),\rho(C'))$. But since $\gamma_n^{-1}o,\gamma_n^{-1}z\in Q(o,\proj_u C)$ we have $\rho(\gamma_n^{-1} z)=\gamma_n^{-1} z$ and $\rho(\gamma_n^{-1} o)=\gamma_n^{-1} o$. Also, since $C$ is the chamber of $A_0$ opposite $u$ and at closest distance from $\proj_u C$, and $C'$ is the chamber of $A_1$ satisfying the same condition, we get that $\rho(C)=C'$. Therefore, we obtain $\gamma_n^{-1} z\in Q(\gamma_n^{-1}o,C')$, which is the desired result.
\end{proof}

In fact, it is possible to obtain a much richer set of sequences, following any projectivity sequences.  Recall that $\Aut(\Upsilon)$ is a subgroup of the product $\Aut(T)\times \ZZ$. For every $(u,v)\in \Delta_\pmop$, choosing and embedding  in $\scrW$ from $\Upsilon$ to the interval $I(u,v)$ gives us  an action of $\Aut(\Upsilon)$ on $I(u,v)$, and in particular of $p_1^{-1}(\Pi)$ (where $\Pi$ is the projectivity group and $p_1$ the first projection). This action also defines actions of $\Pi$ on $\Res(u)$ and $\Res(v)$. 

\begin{proposition}\label{prop:Poincare2}
For every projectivity $p\in\Pi$, and almost every $(u,v)\in \Delta_\pm^\op$,  there exists a sequence $(\gamma_n)\in \Gamma^\NN$ such that for  every $C\in \Delta$ which is opposite $u$, the sequence $\gamma_n C$ converges to the chamber $p\circ \proj_v(\proj_u(C))$.
\end{proposition}

 The main new ingredient in the proof of Proposition~\ref{prop:Poincare2} is the following lemma,  analogous to Lemma \ref{lem:construction_contracting}.

\begin{lemma}\label{lem:construction2}
Fix a projectivity $p\in \Pi$. 
 For $m_\pm$ every $(u,v)\in \Delta_\pmop$, for every finite $F\subset I(u,v)$ and $N\in \NN$, there exists $n>N$ and $\gamma\in \Gamma$ such that the restriction of $\gamma$  on $F$ induces a translation of length $n$ in the direction of $u$, and the projectivity $p$ on the $T_{u,v}$ factor.
\end{lemma}

\begin{proof}
Let $E_F$ be the set of $\phi\in\scrF$ such that $F\subset \Im(\phi)$.
As in Lemma \ref{lem:construction_contracting}, let $\phi_0\in E_F$, and $B=\phi_0^{-1}(F)$. Let $E_{0}=\{\phi\mid \phi_{|B}=\phi_{0|B}\}$. Up to changing $\phi_0$ we may and shall assume $\zeta_\scrW(B)>0$. 

Since $\tilde \Pi$ is closed in $\Aut(T)$, the set $\tilde \Pi_F^p$ of all $p'\in \tilde \Pi$ which have the same action as $p$ on $F$ is open, and therefore its $\tilde \Pi$-Haar measure is positive. Let $E_{p}=\{\phi\mid \phi_{|B} = \phi_{0|B}\circ p\}$. By construction, for almost every $(u,v)\in\Delta_\pmop$, the disintegration of $\zeta_\scrW$ along $\pi:\scrW\to \Delta_\pmop$ identifies to the Haar measure on $\tilde \Pi$. For $(u',v')\in \Delta_\pmop$, denote this measure by $\zeta^{u',v'}_\scrW$.

Now the set of pairs $(u',v')\in\Delta_\pmop$ such that $F\subset I(u',v')$ has positive $m_\pm$-measure. By the previous discussion, for  every such $(u',v')$, we have $\zeta^{u',v'}_\scrW(E_p)>0$. Integrating this over all pairs $(u',v')$ gives that $\zeta_{\scrW}(E_p)>0$.

The rest of the argument is exactly similar to Lemma \ref{lem:construction_contracting}, replacing $E_0$ by $E_p$.
\end{proof}

\begin{proof}[Proof of Proposition \ref{prop:Poincare2}]
The argument is almost the same as in the proof of Proposition \ref{prop:Poincare}. 
We start by applying Lemma \ref{lem:construction2} and get a sequence $(\gamma_n)$ such that $\gamma_n^{-1}$ translates $F_n$ in the direction of $u$ by a length $N>n$ and acts as $p^{-1}$ on the transverse tree.

Recall that $\Aut(\Upsilon)$ is a subgroup of the product $\Aut(T)\times \ZZ$. 
We know that $\gamma_n^{-1}$ acts on $F_n$ as the pair $(p^{-1},-N)$ for some $n>N$. As in the proof of Proposition \ref{prop:Poincare}, let $C$ be a chamber opposite $u$, let $A_0$ be an apartment containing $C$ and $u$. Let $C'=  (\proj_v(\proj_u(C))$ as above, and $C''=pC'$ (where $p\in \Pi$ acts on $\Res(v)$) and let $A_1$ be an apartment containing $\proj_u(C)$ and $C''$. Note that $A_1$ also contains $C'$ since it contains $v$ and $\proj_u(C)$. Let $o\in A_0\cap A_1$ and let $z\in Q(o,C'')$. %Note that the action of $P$ on $\Res(u)$ is also such that $C'=\proj_v(p\proj_u(C))$.

 For $n$ large enough we have $o,z\in F_n$ so $\gamma_n^{-1}$ will act on $o,z$ as the pair $(p^{-1},-N)$. By similar arguments as before we obtain that $\gamma_n^{-1}z\in Q(\gamma_n^{-1} o,C')$, from which it follows (using again the isomorphism $\rho:A_0\to A_1$) that $\gamma_n^{-1}z\in Q(\gamma^{-1}_n o,C),$ and therefore that $z\in Q(o,\gamma_nC)$, which is the conclusion we wanted.

\end{proof}

\subsection{The detecting flow on the building}\label{sec:detecting}

Our next goal is to improve Proposition \ref{prop:Poincare2} into a convergence in weak-* topology for $L^\infty$ functions (Theorem \ref{thm:convTuvproj}). It turns out that this is not a direct argument: in order to do so, we are led to introduce a new flow (which we call \emph{the detecting flow}), and apply the same kind of recurrence arguments as in the proof of Proposition \ref{prop:Poincare2} to this new flow.

We start with a few notations. Start with $\Sigma$ the model apartment of type $\tld A_2$. In $\Sigma$, we choose a wall $\ell$ that we draw horizontally. We choose an origin of $\ell$ and index each vertex of $\ell$ by an integer accordingly (from left to right). For each $n\in \ZZ$, we denote by $\ell_n$ the line which intersects $\ell$ at the vertex $n$ and such that the oriented angle $(\ell,\ell_n)$ is equal to $\pi/3$. For each $n\in \ZZ$, we consider a half-plane $\alpha_n$ (with an $\tld A_2$ tessellation) that we glue on $\Sigma$ along $\ell_n$. 
We let $Y$ be the result of these gluings. 
Alternatively, we view $Y$ as the following subset of $\RR^3$:
\[ \{(x,y,z)\mid z=0 \} \cup \bigcup_{n\in \ZZ} \{(x,y,z)\mid \sqrt{3}(x-n)+y=0,~z\geq 0\}\]\index{$Y$ \hfill model space for the detecting flow|bb}
It is represented on Figure \ref{fig:detectingFlow} below.

\begin{figure}[h]

\includegraphics[width=15cm]{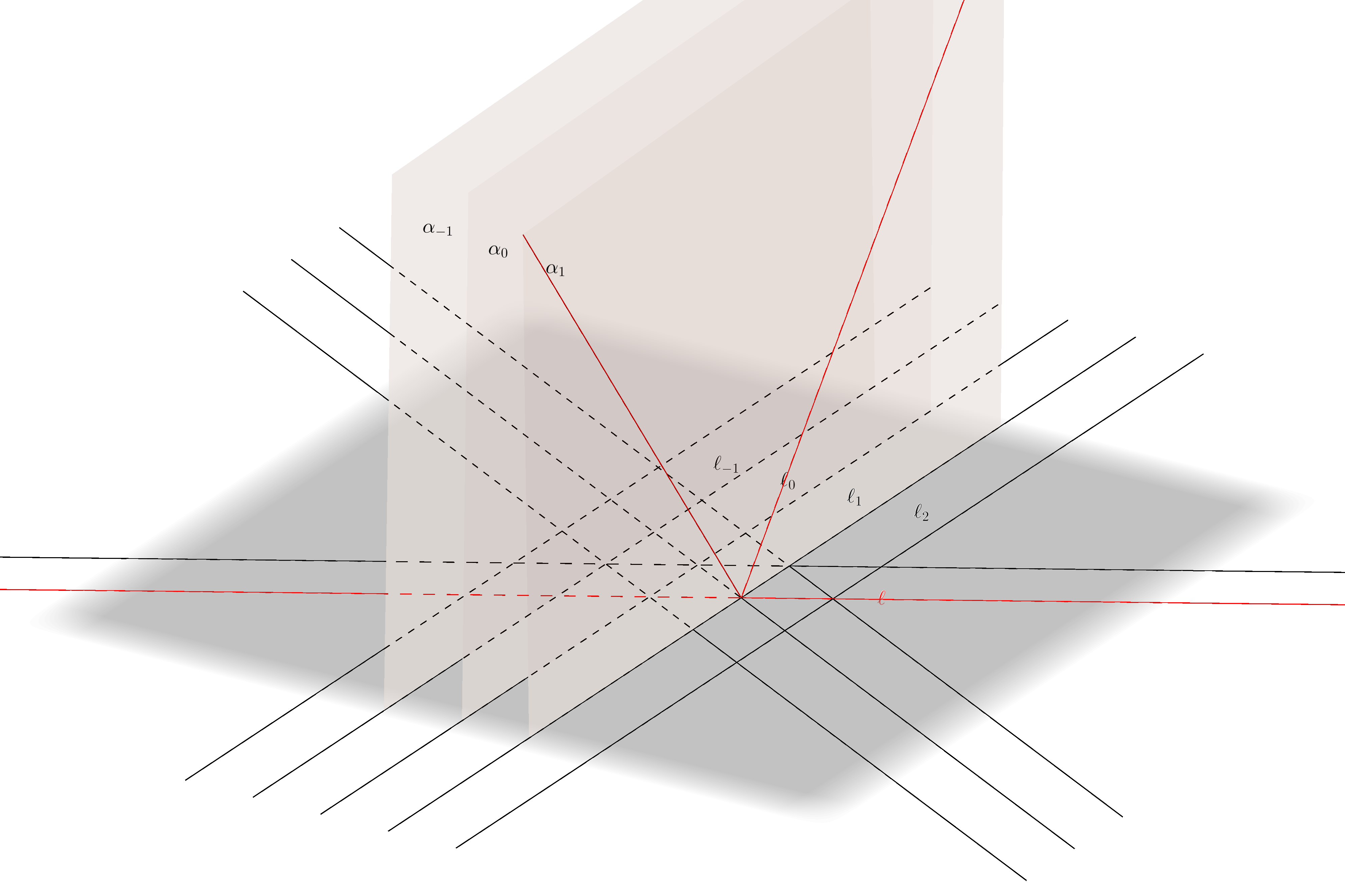}

\caption{The model space for the detecting flow}
\label{fig:detectingFlow}
\end{figure}

\begin{lemma}
The space $Y$ is a convex subset of a wall-tree in an $\tld A_2$ building.
\end{lemma}

\begin{proof}
Fix two opposite vertices $(u',v')\in\Delta_\pmop$ and as before let $T_{u',v'}$ be the tree formed by the set of geodesics between $u'$ and $v'$. In $T_{u',v'}$ consider a "comb", that is, a graph isomorphic to the space $M$ as in Figure \ref{fig:detectingFlowTree} (or in other word, a geodesic line $g$ with attached geodesic rays starting from each vertex of $g$). The union of all geodesics of $X$ corresponding to points of $M$ is then isomorphic to $Y$.
\end{proof}

In particular, $Y$ is ind-$P_\Upsilon$, where $P_\Upsilon$ is the property of being isometric to a convex subset of a model wall tree $\Upsilon$, as defined in the beginning of \S\ref{sec:SoB}. 
Furthermore, the space $X$ is $P_\Upsilon$-symmetric, by Theorem \ref{thm:Psymetric}. Hence we can define the Radon prouniform measure  on $\Isom(Y,X)$.

\begin{definition}
The \emph{detecting flow} is the space $\IR=\Isom(Y,X)$, endowed with the prouniform measure, denoted $\mu_\IR$.
\end{definition}\index{$(\IR,\mu_\IR)$\hfill detecting flow with its prouniform measure|bb}

The space $\IR$ is endowed with an action of $\Aut(X)$ (by post-composition), and a 'shift' $S_Y:\IR\to \IR$ which acts on $Y$ by translating $\Sigma$ and sending $\alpha_n$ on $\alpha_{n+1}$.\index{$S_Y:\IR\to\IR$\hfill Shift map |bb}

Furthermore, let $C_+$ be the chamber at infinity of $\Sigma$, bounded by $\ell$ and one of the lines $\ell_i$, and $C_-$ be its opposite in $\Sigma$. Similarly let $C_n$ be the chamber of the half-plane $\alpha_n$ which is not bounded by $\ell_n$. We define $\pi_+:\IR\to \Delta$ (resp. $\pi_-$, $\pi_n$) by $\pi_+(f)=f(C_+)$ (resp. $\pi_-(f)=f(C_-)$, $\pi_n(f)=f(C_n)$).

Let $Y_k$ be the ball of radius $k$ in $Y$, and $\IR_k=\Isom(Y_k,X)$. Since $Y_k$ is finite, the space $\IR_k$ is equipped with the uniform measure $\mu_{Y_k}$. Let $r_k:\IR\to \IR_k$ be the restriction. We view $L^1(\IR_k,\mu_{Y_k})$ as a subset of $L^1(\IR,\mu_Y)$ by identifying it with its pullback by $(r_k)^*$. 
\index{$Y_k$\hfill ball of radius $k$ in $Y$|bb}
\index{$(\IR_k,\mu_{Y_k})$\hfill $\Isom(Y_k,X)$ with its prouniform measure|bb}

For $g\in\IR$, the line $g(\ell)\subset g(\Sigma)$ joins two points $(u,v)\in \Delta_{\pmop}$. This defines a map $p:\IR\to \Delta_{\pmop}$. The pushforward of $\mu_\IR^o$ by this map is a prouniform measure that we denote $\mu_{\pmop}^o$.

Using Lemma \ref{lem:desintegration_prouniform}, we get that for every $(u,v)\in \Delta_\pmop$, and $o\in I(u,v)$, we have a probability measure $\mu_{u,v}^o$ on $p^{-1}(u,v)$ which is the restricted prouniform measure on the set of embeddings with $\ell$ sent to the line passing through $o$ and joining $u$ to $v$, and these measures satisfy the disintegration formula
$$\mu_\IR^o = \int_{\Delta_\pmop} \mu_{u,v}^o d\mu_\pmop^o (u,v)$$

In the sequel, we will denote $p^{-1}(u,v)$ by $\IR^{u,v}$.
\index{$(\IR^{u,v},\mu_{u,v}^o$)\hfill subset of $\IR$ with fixed endpoints $(u,v)\in\Delta_\pmop$~~~~~~ \\ \hfill with restricted prouniform measure|bb}

Write $o,x_1,x_2,\dots$ (resp. $o,y_1,y_2,\dots$) for the consecutive vertices of the geodesic ray $[ou)$ (resp. $[ov)$).

\begin{definition}
A chamber at infinity $C$ is said to be \emph{$(u,v)$-generic} if it is opposed to both $u$ and $v$.
\end{definition}

\begin{lemma}\label{lem:sectorpi0}
Let $Q=Q(o,C)$ be a sector of $X$ with basepoint $o$, and assume that the alcove of $Q$ containing $o$ is opposite (in $\lk(o)$) to both $x_1$ and $y_1$. Then $C\in \pi_0(\IR_{u,v}^o)$.
\end{lemma}

A chamber $C$ such that $Q(o,C)$ satisfies the condition of Lemma \ref{lem:sectorpi0} will be said to be \emph{$(u,v)$-generic at $o$}.

\begin{remark}
    Since $\IR_{u,v}^o$ is the set of embeddings containing the geodesic line from $u$ to $v$ passing through $o$, this set does not change when choosing a different origin. In particular for $n>0$ we have $C\in\pi_0 (\IR_{u,v}^{x_n})$ if and only if $C$  is $(u,v)$-generic at $x_n$. Applying the shift $S_\Upsilon$, this implies that $C\in \pi_n(\IR_{u,v}^o)$ if and only if $C$ is $(u,v)$-generic at $x_n$.
\end{remark}

\begin{proof}
From Corollary \ref{cor:root+alcove} we know that if $\alpha$ is a half-apartment, and $c$ is an alcove intersecting $\partial\alpha$, then $c\cup \alpha$ is contained in an apartment. Using this fact, one can prove by induction that there exists an apartment $A_n$ containing $Q\cup [o,x_n]$. Therefore, by Corollary \ref{cor:sequenceappt} there exists an apartment $A$ containing $Q\cup [o,u)$.  Let $\alpha$ be the half-apartment of $A$ containing $Q$, bounded by a wall passing through $o$ but not adjacent to $Q$. Let $\alpha'$ be its complement in $A$, so that $\alpha'$ contains $[ou)$. 

Using again Corollary \ref{cor:root+alcove}, we see by induction that there exists an apartment $A'_n$ containing $\alpha\cup [o,y_n]$, and therefore there exists an apartment $A'$ containing $\alpha\cup [o,v)$. Let $\beta$ be the complement of $\alpha$ in $\alpha'$, so that $[o,v)\subset \beta$. Note that $\beta$ does not intersect $\alpha'$, since $x_1$ and $y_1$ are opposite in $\lk(o)$.
Since $\beta$ is a half-apartment bounded by $\partial \alpha'=\partial \alpha$, with empty intersection with $ \alpha'$, it follows that $\beta\cup \alpha'$ is also an apartment. 

Therefore $\alpha\cup \alpha'\cup \beta$ is isometric to a convex subspace $Y'$ of $Y$. Fixing such an isometry $f':Y'\to \alpha\cup \alpha'\cup \beta$, one can extend $f'$ to an isometric embedding $f':Y\to X$, which satisfies $Q=\pi_0(f)$.

\end{proof}

Recall the definition of $E_n(v)$ from \S\ref{sec:martingales} for $n>0$. We extend the definition to $n=0$ by defining $E_0(v)$ as the set of chambers which are $(u,v)$-generic at $0$.

\begin{proposition}\label{prop:pinac}
Let $(u,v)\in \Delta_\pm^\op$, and $o\in I(u,v)$.
The measure $(\pi_0)_*\mu_{u,v}^o$ is absolutely continuous with respect to $\mu_\Delta^o$. 
Furthermore 
for any $n\geq 0$ the measure $(\pi_n)_*\mu_{u,v}^o$ is equal to the normalization of the restriction of $\mu_\Delta^o$ to $E_n(v)$.
\end{proposition}

\begin{proof}[Proof of Proposition \ref{prop:pinac}]

Let $E=E_0(v)$ be the clopen subset of $\Delta$ of chambers which are $(u,v)$-generic at $o$. We first prove that $\mu^o$ and $(\pi_0)_*\mu_{u,v}^o$ are equivalent in restriction to $E$. In particular this proves the first part of the Proposition.
%Since $(\pi_0)_*\mu_{u,v}^o$ is zero outside of $E$, this will prove the first part of the Proposition.

Recall that $\Lambda\subset\Sigma$ is a model sector (based at $0$), so that $\mu^o_\Delta$ is the prouniform measure on the set of embeddings $\Isom(\Lambda,X)$ which send $0$ to $o$. Let $c$ be the alcove of $\Lambda$ containgin $0$. Let $\Lambda'\subset \Lambda$ be a finite convex subcomplex containing $c$, and let $\alpha\in \Isom(\Lambda',X)$ sending $0$ to $o$. By definition, $\mu_\Delta^o(\Isom(\Lambda,X)^\alpha)$ is independent of $\alpha$, equal to a constant say $K$.  Now, view $\Lambda$ as the sector of $Y$ based at $0$ and defining $\pi_0$. Say that $\alpha$ is \emph{admissible} if $\alpha(c)$ is opposite to both $x_1$ and $y_1$.
By Lemma \ref{lem:sectorpi0}, we have $\pi_0(\Isom(Y,X)^\alpha\cap \IR_{u,v}^o)$ which is either empty or equal to $\pi_0(\Isom(Y,X)^\alpha)$, depending on whether $\alpha(c)$ is admissible or not. It follows from Corollary \ref{cor:measurepreserving}
  that $(\pi_0)_*\mu_{u,v}^o (\Isom(Y,X)^\alpha)$ is either equal to $0$ (if $\alpha$ is not admissible) or to some constant independent of $\alpha$ (if it is). Thus, we have 
  $(\pi_0)_*\mu_{u,v}^o (\Isom(Y,X)^\alpha)= K'$ where $K'$ is the proportion of admissible $\alpha$, amongst all the possible ones which send $0$ to $o$.

Since $K$ does not depend on $\Lambda'$, we deduce that $\frac{d(\pi_0)_*(\mu_{u,v}^o)}{d\mu^o_\Delta}(C) = 0$ if $C$ is not $(u,v)$-generic at $o$, and $\frac {K'}K$ if it is. This proves the claim.
%that on the clopen set of chambers which are $(u,v)$-generic at $o$, $(\pi_0)_*\mu_{u,v}^o$ and $\mu^o_\Delta$ are equivalent.

Now let us prove the second statement. For $n=0$ this is what we just did. For $n>0$, we claim first, $E_n(v)$ is exactly the set of chambers which are $(u,v)$-generic at $x_n$. Indeed, if $C\in E_n(v)$, then $Q(o,C)$ contains the geodesic ray $[o x_n]$, and it follows that $C$ is opposite $v$ at $x_n$. The ray $[x_n \pr_+(C)]$ starts with a segment of the same type as $[x_n x_{n+1}]$, but which is not equal to this segment. Hence these two segments are opposite in $\lk(x_n)$, which proves that $C$ is $(u,v)$-generic at $x_n$. Conversely, if $C$ is $(u,v)$-generic at $x_n$, then the union of the rays $[o x_n]\cup[x_n \pr_+(C)) $ has an angle $\pi$ at $x_n$, so is still a geodesic ray. Hence $C\in \Omega_o(\lambda_n(v))$. Furthermore it is clear that $x_{n+1}\not\in Q(o,C)$, hence we do have $C\in E_n(v)$.

Therefore $(\pi_n)_*\mu$ is supported on $E_n(v)$. Using the same argument as for $\pi_0$, we see that $\frac{d(\pi_n)_*\mu_{u,v}^o}{d\mu_\Delta^o}$ is in fact constant (and $>0$) on $E_n(v)$. It follow that $(\pi_n)_*\mu_{u,v}^o$ is equal to a constant times the restriction of $\mu_{u,v}^o$ on $E_n(v)$ ; since it is a probability measure we get the result.

\end{proof}

Finally, we are interested in what happens when we change the basepoint $o$.

\begin{lemma}\label{lem:uvgeneric}
For any $(u,v)$-generic chamber $C$, there exists $o$ between $u$ and $v$ such that $C$ is $(u,v)$-generic at $o$.
\end{lemma}

\begin{proof}
Let $\ell$ be a geodesic line between $u$ and $v$. Let us number the vertices of $\ell$ by $(\ell_i)_{i\in \ZZ}$ in such a way that $\lim_{i\to +\infty} \ell_i =v$. 
 There exists an apartment containing $u$ and $C$ and every $\ell_i$, for $i$ small enough. However, if for every $i\in \ZZ$ there existed an apartment containing $(u\ell_i]$ and $C$, then there would be an apartment containing $\ell$ and $C$. This is not possible as $C$ is opposite both $u$ and $v$.

It follows that there exists a maximal $i_0$ such that $u,\ell_{i_0}$ and $C$ are in a common apartment. We define $o=\ell_{i_0}$ and prove that $C$ is $(u,v)$-generic at $o$. By construction $o$ is in some apartment containing $u$ and $C$, and $u$ is opposite $C$, so the projection of $C$ and $u$ on $\lk(o)$ are opposite.

On the other hand, the geodesic ray $[o,v)$ is the line $\ell_{i_0},\ell_{i_0+1},\dots$. Let $y=\ell_{i_0+1}$ Let $c$ be the alcove which is the projection of $C$ on $\lk(o)$ and $x$ be the vertex of $c$ of the same type as $\ell_{i_0-1}$. We want to prove that $x$ is opposite $y$ in $\lk(o)$.

Assume not. Since $\lk(o)$ is of type $A_2$, and $x$ and $y$ have distinct type and are not opposite, they are adjacent.
Start with an apartment $A_0$ containing $u$, $o$ and $C$ (which exists by assumption). By construction, $x$ is also a vertex in $A_0$. Let $\hat h$ be the wall of $A_0$ containing $[ux)$, and $h$ be the half-space of $A_0$ containing $C$. Since $y$ is adjacent to $y$, we see that $y$ is a vertex which is contained in an alcove which has an edge in $\hat h$. Therefore by Corollary \ref{cor:root+alcove} there exists an apartment $A$ containing $h$ and $y$. In particular, $A$ contains $u,y $ and $C$. This contradicts the maximality of $i_0$.

\end{proof}

\begin{proposition}\label{prop:sumpi0}
The sum over all $o\in I(u,v)$ of the measure classes of $(\pi_0)_*\mu_{u,v}^o$ is equivalent to $\mu_\Delta^o$.
\end{proposition}

\begin{proof}
Let $\nu$ be a measure in the measure class defined by the sume of $(\pi_0)_*\mu_{u,v}^o$ 
 By Proposition \ref{prop:pinac} the measure $\nu$ is absolutely continuous with respect to $\mu^o_\Delta$. In fact, the same proposition shows that a subset $A\subset \Delta$ is of positive $\nu$-measure if and only if its intersection with $E_0(v)$ is of positive $\mu_\Delta^o$-measure for some $o$. But $E_0(v)$ is the set of chambers which are $(u,v)$-generic at $o$.
 
  But by Lemma \ref{lem:uvgeneric}, every $(u,v)$-generic chamber is $(u,v)$-generic at $o$ for some $o\in I(u,v)$. Assume that $A$ has $\mu^o_\Delta$-positive measure. Since the set of $(u,v)$-generic chambers is of full measure, by countability there exists an $o\in I(u,v)$ such that the set of $(u,v)$-generic chambers at $o$ has a positive measure intersection with $A$. This proves that $A$ has positive $\nu$-measure. Hence the measure class of $\nu$ is the same as the measure class of $\mu_\Delta^o$.
  
  %this set is in fact the set of all $(u,v)$-generic chambers, which is of full measure. Hence the two measure classes are equal.
\end{proof}

\begin{remark}
The measure $ \sum_{o\in I(u,v)}(\pi_0)_*\mu_{u,v}^o$ is not locally finite. However, the measure class is well-defined.
\end{remark}

\subsection{Contracting properties} \label{subsec:CP}

As explained above, the goal of the detecting flow is to improve on Proposition \ref{prop:Poincare2} to a convergence for bounded measurable functions. This is what we finally do in this section: our goal is to prove Theorem \ref{thm:convTuvproj} below.

%%%%%%%

\begin{lemma}\label{cor:weak*conv-uv}
For every $f\in L^\infty(\Delta)$, for $\mu_\pmop$-almost every $(u,v)\in \Delta_\pmop$, for every $o$ between $u$ and $v$ and for every $\varphi\in L^1(\IR_{u,v}^o)$ we have
$$\int \varphi(y) f(\pi_n y) d\mu_{u,v}^o(y)\longrightarrow \int \varphi(y) \pi_\infty^* (f) d\mu_{u,v}^o(y)$$ 
\end{lemma}

\begin{proof}
\begin{comment}
    Assume $\varphi=1$\jl{for now}. By Proposition \ref{prop:pi0ac}, we have $$\int f(\pi_n y) d\mu_{u,v}^o(y)=\frac{1}{\mu_\Delta^o(E_n(v))}\int_{E_n(v)} f(\xi)d\mu_\Delta^o(\xi)$$  

    On the other hand, by construction of prouniform measures, $(\pi_\infty)_*\mu_{u,v}^o$ is the restricted prouniform measure on $\Delta=\Isom^o(\Lambda,X)$ restricted to the chambers with a fixed ray $[ov)$ as a wall. In other words, it is $\mu_v^o$. Therefore we get  $ \int f(\pi_\infty(y)) d\mu_{u,v}^o(y)=\int_{\Res(v)} f(\xi) d\mu_v^o(\xi)$. The result follows from Corollary \ref{cor:martingaleEn}.
\end{comment}

    Since the norm of the function $y \mapsto f(\pi_n y) -
\pi_\infty^* (f)$ is bounded by $2\|f\|$, it is enough to prove the convergence with respect to a total subset of functions $\varphi$. Hence by Lemma \ref{lem:densityL1_gen}, it is sufficient to prove the Lemma for $\varphi$ in the pullback of $L^1(\IR_k)$, for every $k>0$. Fix such a $k$. By linearity it is also sufficient to prove the Lemma when $\varphi$ is the characteristic function of a fixed embedding $\iota_0\in \IR_k\cap \IR_{u,v}^o$. Fix $\iota\in \IR_{u,v}^o$ such that $\iota_{|\Upsilon_k}=\iota_0$ and let $\xi=\pi_\infty(\iota)$.

     Fix $n>k$.  Note that there exists an automorphism $\tau$ of order 2 of $Y$ which exchanges the two half-apartments bounded by $\ell_n$ not containing $0$ ; in particular $\pi_n\circ \tau=\pi_\infty$. Furthermore $\tau$ fixes pointwise the ball of radius $k$ in $Y$.

     Let $\eta\in E_n(v)$. Then $(\pi_n)_*\varphi(\eta)=1$ if and only if there exists an extension $\iota'\in \IR_{u,v}^o$ of $\iota_0$ such that $\pi_n(\iota')=\eta$. We claim that $(\pi_n)_*\varphi(\eta)=1$ if and only if $\lambda_{k,n}(\xi)\in Q(o,\eta)$. Indeed, if there is an extension $\iota'$ as above, then let $\xi'=\pi_\infty(\iota')$ (and note that $\pr_+(\xi')=v$). Since $\iota'$ is equal to $\iota$ in restriction to $Y_k$ (which contains $\lambda_{k,0}$) we have $\lambda_{k,0}(\xi')=\lambda_{k,0}(\xi)$, and the geodesic ray from $\lambda_{k,0}(\xi)$ to $u$ is contained in $Q(o,\xi')$. Applying $\tau$ we see that the initial segment of length $n$ of this ray is contained in $Q(o,\pi_n(\iota))$, so that indeed $\lambda_{k,n}(\xi)\in Q(o,\eta)$. Conversely, assume that $\lambda_{k,n}(\xi)\in Q(o,\eta)$. Then we have $Q(\lambda_{k,n}(\xi),\eta)\subset Q(o,\eta)$, and this sector is also contained in a half-apartment $\alpha$ which is bounded by the wall $\iota(\ell_n)$. Replacing the half-apartment of $\iota$ bounded by $\ell_n$ by $\alpha$, we still get an isometric embedding of $Y$ which is equal to $\alpha$ on $Y_k$.

     Using the claim, together with Proposition \ref{prop:pinac} and Lemma \ref{lem:Enmintersection} we get
     $$\int \varphi(y) f(\pi_n y) d\mu_{u,v}^o(y)=\frac{1}{\mu(E_n(v))}\int_{E_{k,n}(\xi)} f(\eta) d\mu_\Delta^o(\eta)$$

     Note also that for $k>0$ we have $\mu_\Delta^o(E_{k,n}(\xi))=\frac{1}{(q+1)q^{k-1}}\mu_\Delta^o(E_n(v))$. Hence by Lemma \ref{lem:MartingaleEnm} we get that 
     $$\lim_n\int \varphi(y) f(\pi_n y) d\mu_{u,v}^o(y)=\frac{1}{(q+1)q^{k-1}}\frac{1}{\mu_u^o(\Omega_{k,\infty}(\xi))} \int_{\Omega_{k,\infty}(\xi)}f(\eta)d\mu_u^o(\eta)$$

     On the other hand, we have $\mu_u^o(\Omega_k,\infty)(\xi))=\frac{1}{(q+1)q^{k-1}}$. Also, note that $(\pi_n)_*\varphi(\eta)=1$ if and only if $\eta\in \Omega_{k,\infty}(\xi)$. This implies that
     $$\int \varphi(y) \pi_\infty^* (f) d\mu_{u,v}^o(y) = \int_{\Omega_{k,\infty}(\xi)} f(\eta) d\mu_u^o(\eta). $$
     The result follows.
\end{proof}
%%%%%%%%

(Note that the "almost sure" set of pairs $(u,v)$ for which Lemma \ref{cor:weak*conv-uv} hold
might depend on $f$.)

\begin{lemma}\label{lem:detectingfunction}
    Fix $(u,v)\in \Delta_\pm^\op$, $k>0$, $p$ a projectivity, $o\in I(u,v)$ and let $(\gamma_n)$ be a sequence of elements of $\Gamma$ provided by Lemma \ref{lem:construction2} (with $F$ the ball of radius $k$ around $o$), namely such that $\gamma_n$ acts as a composition of $p$ and a translation of length $N(n)>0$ in the direction of $u$ on the ball of radius $k$ around $o$. Let $o'$ be the translate of $\gamma_1(o)$ by $N(1)$ in the direction of $v$.

    Assume that $N(n)<N(n+1)$ for every $n\in\mathbf{N}$. 
    Fix $\iota\in \IR_{u,v}^{o}$. 
    Let $$E=\{\iota'\in \IR_{u,v}\mid \iota'(x)=\iota(x) \;\forall x\in B(0,k)\}
$$ 
and let $\varphi\in L^1(\IR_{u,v}^o)$ be the characteristic function of $E$.  Then there exists a $\varphi_0\in L^1(\IR_{u,v}^{o'})$ such that for every $n>0$ we have
    $(\pi_0)_* \gamma_n \varphi = (\pi_{N(n)})_*\varphi_0$.

    Furthermore $(\pi_\infty)_*\varphi_0=(\pi_\infty)_* \varphi \circ p^{-1}$.

\end{lemma}
\begin{proof}
    Let $\Sigma_k$ be the intersection of  $\iota(B(0,k))$ and the interval $I(u,v)$. Then $\Sigma_k$ is a ball of radius $k$ in an apartment. By construction all the images by $\gamma_n$ of $\Sigma_k$ are translates of one another in the direction of $u$, so there exists an apartment $\Sigma_0$ which contains them all. In particular $o'\in \Sigma_0$.   Let $D$ be the convex hull of $\gamma_n \Sigma_k$ and $u$ and $v$. By construction $D$ does not depend on $n$ and is the strip of width $k$ around the line from $u$ to $v$ containint $o'$ in $\Sigma_0$. 

    Now our aim is to define $\varphi_0$. In order to do so, we first construct a $\iota_0\in \IR_{u,v}^{o'}$ whose image will contain $\Sigma_0$ and all the $\gamma_n B(o,k)$. Recall that our model space $Y$ for the detecting flow is the union of the base apartment $\Sigma$ and a sequence of half-spaces $\alpha_i$. Note that for every $i$ $\Sigma\cup\alpha_i$ is convex, so that by Theorem \ref{thm:Psymetric} it is possible to define $\iota$ on each $\alpha_i$'s independently (as long as they agree on $\Sigma$). 

    We define first $\iota_0$ by taking $\iota_0(\Sigma)=\Sigma_0$, $\iota_0(0)=o'$, oriented in such a way that $\pi_\infty(\iota_0)\in \Res(u)$. Changing our choice of $\Sigma_0$ if needed, we can assume further  that $\pi_\infty(\iota_0) = p(\pi_\infty(\iota))$, where $p$ is the projectivity from the Lemma, acting on $\Res(u)$. Let $F$ be the intersection of $\iota(\alpha_0)$ with $B(o,k)$. By definition of $\gamma_n$ we have $\gamma_n (F \cup \Sigma_k)$ which intersects $\Sigma_0$ on $\Sigma_k$ ; hence $D\cup \gamma_n F$, and even $\Sigma\cup \gamma_n F$ is a convex subset of a model detecting flow and we can choose $\iota_0$ such that $\iota_0(\alpha_{N(n)})$ contains $\gamma_n F$ (still keeping $\iota_0(\Sigma)=\Sigma_0$). Since we assumed that the sequence $(N(n))_{n>0}$ is strictly increasing, we can do that independently for every $n$. 

    We extend $\iota_0$ arbitrarily. The set that we fixed to be in the image $\iota_0$ is the union of $D$ with all the $\gamma_nF$. We denote its preimage by $H$ and define $\varphi_0$ as the characteristic function of 
    
    $$ E_0=\{\iota'\in \IR_{u,v} \mid \iota'(x)=\iota_0(x) \;\forall x\in H\}$$
    
    Note that $\varphi_0\in L^1(\IR_{u,v})$, as $E_0$ has a positive and finite measure.

    Fix $n>0$ and denote $N=N(n)$. We first check that $(\pi_0)_*\gamma_n\varphi=(\pi_{N})_*\varphi_0$. Indeed, let $C\in \Delta$. We see that $(\pi_0)_*\gamma_n\varphi(C)=\gamma_n(\pi_0)_*\varphi(C)$ is equal to $1$ if and only if $C\in \gamma_n \pi_0(E)$, that is when the sector $Q(\gamma_n o,C)$ contains $\gamma_n F$. On the other hand, $(\pi_{N})_*\varphi_0(C)=1$ if and only if $C\in \pi_N(E_0$), that is when the sector $Q(x_N,C)$ contains $\iota_0(\alpha_N\cap B(x_N,k))$, which is equal to $\gamma_n F$ by construction of $\iota_0$. Hence we get the result.

    Finally let us check that $(\pi_\infty)_*\varphi_0=(\pi_\infty)\varphi\circ p^{-1}$. For the sake of notations, let $\omega=\pi_\infty(\iota)$ and $\omega_0=\pi_\infty(\iota_0)$. Let $\pi_u:X\to T_u$ be the natural projection. We identify the action of the projectivity group on $\Res(u)$, $T_{u,v}$ and $T_u$ in the natural way. Let $C\in\Res(u)$, we have that 
    $(\pi_{\infty})_*\varphi_0(C)=1$ if and only if $Q(o',C)$ contains $D_+:=Q(o',\omega_0)\cap D$, which is equivalent to $Q(o',p^{-1}(C))$ containing $Q(o,p^{-1}(\omega_0))\cap p^{-1}(D_+)$, or to $\pi_u(Q(o,p^{-1}(C))$ containing $p(\pi_u(D_+))$ (note that $p(\pi_u(o))=(\pi_u(o')$). But by definition of $\gamma_n$ we have that $p^{-1}(\pi_u(D_+)) = \pi_u(Q(o,\omega)\cap B(o,k))$. Hence we get that $(\pi_{\infty})_*\varphi_0(C)=1$ if and only if  $\pi_u(Q(o,p^{-1}(C)))$ contains $\pi_u(Q(o,\omega)\cap B(o,k))$, which means exactly that $p(C)\in \pi_\infty(E)$, that is, $(\pi_\infty)_*\varphi(C)=1$.

\end{proof}

This finally allows us to prove the main theorem of this section.

\begin{theorem}\label{thm:convTuvproj}
For every $f\in L^\infty(\Delta)$, for almost every $(u,v)\in \Delta_\pm^\op$, for every projectivity $p$ (acting on $\Res(v)$), 
there is a sequence $(\gamma_n)_{n\in\NN}$ of elements of $\Gamma$ such that the sequence $(\gamma_n.f)_{n\in \NN}$ converges (in the weak-* topology on $L^\infty(\Delta)$) to the  map $x\mapsto  f(p\circ \proj_v(\proj_u(x)))$.
\end{theorem}

\begin{proof}
Let us fix $f\in L^\infty(\Delta)$. 
Fix $(u,v)$ in the intersection of the full measure sets provided by Corollary \ref{cor:weak*conv-uv} and Lemma \ref{lem:construction2}. Let $(\gamma_n)$ be the sequence provided by Lemma \ref{lem:construction2}. 

Note that if $g\in \IR_{u,v}$ then $\pi_\infty( g )= \proj_v \circ \proj_u \circ \pi_0 (g)$: this follows from the construction of the detecting flow. In particular by Proposition~\ref{prop:Poincare2} we have $(\gamma_n \pi_0(g))_{n\in \NN}$ which converges to $p\circ \pi_\infty(g)$.

We claim that the sequence $(\gamma_n^* \pi_0^* f)$ converges to $(p\circ \pi_\infty)^* f$ in the weak-* topology of $L^\infty (\IR_{u,v}^o)$, for every $o$ in $I(u,v)$. Granting the claim, this proves that $\gamma_n f$ weak-* converges to $x\mapsto  f(p\circ \proj_v(\proj_u(x)))$ in the weak-* topology of $L^\infty(\Delta,(\pi_0)_* \mu_{u,v})$. Since this is valid for every $o$, using  Proposition \ref{prop:sumpi0}
 the same convergence holds in $L^\infty(\Delta,\mu_{\Delta})$, which is the conclusion of the Theorem.

So it remains to prove the claim. Fix $o\in I(u,v)$. By compactness of balls in $L^\infty(\IR_{u,v}^o)$ (for the weak-* topology), it is enough to prove that every convergent subsequence of $(\gamma_n^* \pi_0^* f)$ converges to $(p\circ \pi_\infty) ^* f$. Therefore, we may and shall assume that the sequence $(\gamma_n^* \pi_0^* f)$ converges. In order to identify the limit, it is sufficient to identify the limit against a dense set of functions in $L^1(\IR_{u,v}^o)$. This dense set will be the union of  $ L^1(\IR_{u,v}^{o,k})$  for $k\geq 0$ (see Lemma \ref{lem:densityL1_gen}).  In fact, by linearity, it is enough to identify the limit for the characteristic function $\varphi$ of a fixed embedding $\iota\in \IR_{u,v}^{o,k}$.

So fix such a $k$, and $g\in \IR_{u,v}^{o,k}$. For $n$ large enough, we know that $\gamma_n$ acts as the composition of a translation of the ball of radius $k$ in $I(u,v)$ by some length $N$  with $p$, where $N=N(n)$ tends to $+\infty$ as $n$ increases. Now let $\varphi_0\in L^1(\IR_{u,v}^{o'})$ be the function provided by Lemma \ref{lem:detectingfunction}. Then we have

\begin{align*}
\langle \gamma_n^* \pi_0^* f,\varphi\rangle  & = \langle f,(\gamma_n  \pi_0)_*\varphi\rangle \\
&=\langle f,(\pi_N)_*\varphi_0\rangle\\
&=\langle \pi_N^*f,\varphi_0\rangle
\end{align*}

which converges to $\langle \pi_\infty^*f,\varphi_0\rangle$ by Lemma \ref{cor:weak*conv-uv}. By Lemma \ref{lem:detectingfunction}, this limit is also equal to $\langle (p\circ \pi_\infty^*)f,\varphi\rangle$. This concludes the proof.

\end{proof}

Let us emphasize the particular case when $p$ is the identity:

\begin{corollary}\label{cor:convTuv}
For every $f\in L^\infty(\Delta)$, for almost every $(u,v)\in \Delta_\pm^\op$, there is a sequence $(\gamma_n)_{n\in\NN}$ of elements of $\Gamma$ such that the sequence $(\gamma_n.f)_{n\in \NN}$ converges (in the weak-* topology on $L^\infty(\Delta)$) to the  map $x\mapsto  f(\proj_v(\proj_u(x)))$.
\end{corollary}

\subsection{Dynamics on the boundary}

 Recall that an action of a group $G$ on a topological space $Z$ is \emph{topologically transitive} if for every nonempty open set $U,V\subset Z$, there exists $g\in G$ such that $gU\cap V\neq\varnothing$. If $Z$ is Polish, it is equivalent to the existence of a dense orbit in $Z$. The action is \emph{minimal} if every orbit is dense.
 
 We say that the action is \emph{topologically 2-minimal} if the orbit of every pair of distinct points is dense in $Z\times Z$.

The main interest of topological 2-minimality is the following.

\begin{lemma}
Assume that $G$ is a group acting topologically 2-minimally on a space $Z$. Let $p:Z\to Y$ be a continuous, $G$-equivariant surjective map. Then either $Y$ is reduced to a point or $p$ is bijective.
\end{lemma}

\begin{proof}
Assume that $Y$ is not reduced to a point and that $p$ is not bijective. Then there exists $x_1\neq x'_1\in Z$ such that $p(x_1)=p(x'_1)=y_1$ and $x_2\in Z$ with $p(x_2)=y_2\neq y_1$.

By 2-transitivity, there exists a sequence $(g_n)$ of elements of $G$ such that $g_n(x_1,x'_1)$ converges to $(x_1,x_2)$. By the continuity of $p$, we see that $g_n (y_1,y_1)$ converges to $(y_1,y_2)$. This is a contradiction.
\end{proof}

\begin{lemma}\label{lem:projectivityis2minimal}
The action of the projectivity group on the residue of any panel of $\Delta$ is $2$-minimal.
\end{lemma}

\begin{proof}
Let $P$ be the projectivity group, acting on the set $Z=\Res(u)$ for some panel $u$.
We have seen in Lemma \ref{lem:3transitive} that the action of $P$ on $Z$ is $3$-transitive. In particular $P$ acts transitively on the set of pairs of distinct points. Also note that (since $X$ is thick) there is no isolated point in $Z$. Hence, every $(x,x)\in Z\times Z$ can be approached by a sequence $(x_n,x)$ with $x_n\neq x$ for every $n$.
\end{proof}

%\begin{lemma}\label{lem:TuvProj}
%For almost every pair $(u,v)\in \Delta^\pm_\op$, and every projectivity $\varphi$ (acting on $\Res(u)$), there exists a sequence $(\gamma_n)$ of elements of $\Gamma$ such that $\gamma_n C$ converges to $\varphi \circ T_u^v(C)$. 
%\end{lemma}
%
%\begin{proof}
%Apply the same arguments as in Theorem \ref{thm:convTuv}, but find a sequence of elements which act like a given projectivity on a large ball. \{Do it in the right section}
%\end{proof}

\begin{theorem}\label{thm:2minimal}
The action of $\Gamma$ on $\Delta_+$ is topologically 2-minimal.
\end{theorem}

\begin{proof}

Let $P_0\subset P$ be a countable dense subgroup of the projectivity group $P$. Let $E$ be the subset of full measure of $\Delta_\pmop$ provided by Proposition \ref{prop:Poincare2} which works for every $\varphi\in P_0$.

Let $x_1\neq x_2\in\Delta_+$.  We prove that the closure of $\Gamma\cdot (x_1,x_2)$ contains a full measure subset of $\Delta_+^2$, hence is dense since $\mu_+^o$ is fully supported. Let $u_0$ be the vertex of type $-$ adjacent to both $x_1$ and $x_2$. Since the action of $\Gamma$ on $\Delta_-$ is minimal, the closure of $\Gamma u_0$ is equal to $\Delta_-$. In particular, replacing $(x_1,x_2)$ by an element in the closure of the orbit of $(x_1,x_2)$, we may and shall assume that the set of $v\in \Delta_+$ such that $(u_0,v)\in E$ is of full measure. In the sequel, for every $(u,v)\in \Delta_\pmop$, we define $T_u^v=\proj_v\circ \proj_u$, and $T_v^u = \proj_u\circ \proj_v$.\index{$T_u^v$\hfill $\proj_v\circ\proj_u$|bb}

 Let $y_1,y_2\in\Delta_+$. Let $u$ be the (or any if $y_1=y_2$) vertex of type $-$ which is adjacent both to $y_1$ and $y_2$. Almost surely, the set of $v\in \Delta_+$ such that $(u,v)\in E$ is of full measure. Let $v$ be such that $(u,v)\in E$ and $(u_0,v)\in E$. Let $c_1,c_2$ be the chambers of $\Res(u_0)$ adjacent to $x_1$ and $x_2$ respectively, and $d_1,d_2$ be the chambers of $\Res(u)$ adjacent to $y_1,y_2$ respectively. 
 Then by construction of $E$ there exists a sequence $(\gamma_n)$ such that $\gamma_n(c_1,c_2)$ converges to $(c'_1,c'_2)=(T_{u_0}^v c_1,T_{u_0}^v c_2)\in \Res(v)^2$. Similarly, for any $\varphi \in P_0$, there exists yet another sequence $(\gamma'_n)$ such that $\gamma_n'(c'_1,c'_2)$ converges to  $(d'_1,d'_2)=(\varphi\circ T_v^u (c'_1),\varphi\circ T_v^u (c'_2))$. Hence the closure of $\Gamma (c_1,c_2)$ contains the $P$-orbit of $(d'_1,d'_2)$. By Lemma \ref{lem:projectivityis2minimal}, it follows that this closure contains all of $\Res(u)^2$, and therefore the pair $(d_1,d_2)$. We have proven that $\overline{ \Gamma\cdot (x_1,x_2)}$ contains $(y_1,y_2)$.

\end{proof}

\begin{remark}
    The corresponding statement is not true for example if $X$ is a $\tilde C_2$ building. Indeed if $u,v$ are two vertices in the boundary of a $X$ at Tits distance $\pi/2$ then they are of the same type, say $+$, but the orbit of $(u,v)$ cannot be dense in $\Delta_+\times \Delta_+$ as it will be impossible to approach a pair of opposite vertices.    
\end{remark}

Of course the same argument holds for $\Delta_-$ instead of $\Delta_+$. Therefore, we get :

\begin{corollary}\label{cor:FactorTop}
There is no non-trivial $\Gamma$-equivariant \emph{topological} factor of $\Delta_+$ and of $\Delta_-$.
\end{corollary}

\section{Subalgebras of $L^\infty(\Delta)$}

\subsection{Four algebras}

The map $\pr_+:\Delta\to \Delta_+$ which associates to a chamber its vertex of type $+$ is measure-class preserving, hence gives an inclusion of algebras $\pr_+^*:L^\infty(\Delta_+)\to L^\infty(\Delta)$. In the sequel, we identify $L^\infty(\Delta_+)$ with $\pr_+^*(L^\infty(\Delta_+))$. In other words, we can view a function on $\Delta_+$ as a function on $\Delta$ which is constant on every residue of type $+$.
 Similarly, we identify $L^\infty(\Delta_-)$to a subalgebra of $L^\infty(\Delta)$ via the map $\pr_-^*$.

The following two lemmas only use the formalism of prouniform measures, and are valid in any 2-dimensional affine building.

\begin{lemma}\label{lem:join}
The algebras $L^\infty(\Delta_+)$ and $L^\infty(\Delta_-)$ generate a weak-* dense subalgebra of $L^\infty(\Delta)$.
\end{lemma}

\begin{proof}
It suffices to show that for every $z\in X$, the characteristic function of $\Omega_o(z)$ is contained in the algebra $A$ generated by $L^\infty(\Delta_+)$ and $L^\infty(\Delta_-)$, as these functions generate a dense subalgebra of $L^\infty(\Delta)$.

So let us fix such a $z$ (as well as an origin $o\in X$), and let $f_{o,z}$ be the characteristic function of $\Omega_o(z)$. The convex hull $\Conv(o,z)$ is a (possibly degenerate) parallelogram whose other vertices will be denoted $z_+$ and $z_-$, where $[oz_+]$ (resp. $[oz_-]$) is the initial segment of a singular ray of type $+$ (resp. of type $-$).

Let $\Omega_o^+(z)$ be the projection of $\Omega_o(z)$ by the natural map $\pr_+:\Delta\to\Delta_+$. Then we claim that $\pr_+^{-1}(\Omega_o^+(z))=\Omega_o(z^+)$. Indeed, if $C\in \Omega_o(z)$ then $\Conv(o,z)\subset Q(o,C)$ so that $z_+\in Q(o,C)$, and therefore $C\in \Omega_o(z_+)$. This proves that $\pr_+^{-1}(\pr_+(\Omega_o(z))\subset \Omega_o(z_+)$. Conversely, assume that $C\in \Omega_o(z_+)$. Let $u=\pi_+(C)$, so that $[oz_+]$ is an initial segment of $[ou)$. Using repeatedly Corollary \ref{cor:root+alcove} one can prove that there exists an apartment $A$ containing $[ou)$ and $z$, so that there is a chamber $C'$ in $A$ containing $[ou)$ and $z$. Hence $C'\in \Omega_o^+(z)$, and since $\pr_+(C)=\pr_+(C')=u$ we get that $C\in \pr_+^{-1}(\Omega_o^+(z))$.

It follows that $f_{o,z^+}\in A$. Similar arguments apply to prove that $f_{o,z^-}\in A$. More generally since $o$ is arbitrary we get that if $x,y$ are on a singular ray then $f_{x,y}\in A$. 

Now we claim that $\Omega_o(z)=\Omega_o(z_+)\cap \Omega_{z_+}(z)$. Indeed, if $C\in \Omega_o(z)$ then $Q(o,C)$ contains $\Conv(o,z)$ and therefore $z_+$. Conversely if $C\in \Omega_o(z_+)\cap \Omega_{z_+}(z)$ then $z_+\in Q(o,C)$ and therefore $Q(z_+,C)\subset Q(o,C)$, and since $z\in Q(z_+,C)$ by assumption we conclude that $z\in Q(o,C)$.

From this claim it follows that $f_{o,z}=f_{o,z^+}\cdot f_{z^+,z}$. Since $f_{o,z^+}\in A$ and $f_{z^+,z}\in A$ we have $f_{o,z}\in A$. This concludes the proof.

\end{proof}

\begin{lemma}\label{lem:meet}
The intersection $L^\infty(\Delta_+)\cap L^\infty(\Delta_-)$ is reduced to the constant functions.
\end{lemma}

\begin{proof}
Choose model positive chamber $C_+\subset \Sigma$, and let $C_-=-C_+$. Recall first that the map $\scrF\to \Delta_{\op}$ associating to $\varphi\in \scrF$ the pair $(\varphi(C_+),\varphi(C_-))$ sends the  measure-class of $\mu_\scrF$ to the one of $\mu_\Delta^o\times\mu_\Delta^o$  \cite[Theorem 6.21]{BCL}. Let $C'$ be the chamber of $\Sigma$ which is $+$-adjacent to $C_+$, $C''$ the chamber of $\Sigma$ which is $+$-adjacent to $C_-$ ; note that $C'$ and $C''$ are $-$-adjacent and denote by $u_0$ their common vertex. Similarly, the map $\varphi\to \varphi(C')$ (resp. $\varphi\to \varphi(C'')$) sends the measure class of $\mu_\scrF$ to the one of $\mu_\Delta$.

Now let $f\in L^\infty(\Delta_+)\cap L^\infty(\Delta_-)$. Up to changing $f$ on a zero measure subset, we can assume that $f=\pr_+^* g$ for some $g\in L^\infty(\Delta_+)$.  In other words $f$ has the same value on every pair of $+$-adjacent chambers. 
By the previous remark, it suffices to prove that for almost every $\varphi\in \scrF$ we have $f(\varphi(C_+)) = f(\varphi(C_-))$. 

Note first that by assumption we have, for every $\varphi\in\scrF$, $f(\varphi(C_+))=f(\varphi(C')) $ and $f(\varphi(C_-))= f(\varphi(C''))$, so it suffices to prove that $f(\varphi(C'))=f(\varphi(C''))$ for almost every $\varphi\in \scrF$. Now since $f\in L^\infty(\Delta_-)$ we get that $f$ coincides almost everywhere with $\pr_-^*h$ for some $h\in L^\infty(\Delta_-)$, on a set of full measure which we denote $E$. Since $E$ is of full measure, the set $E'$ of $\varphi\in \scrF$ such that $\varphi(C')\in E$ and $\varphi(C'')\in E$ is of full measure in $\scrF$. For every $\varphi\in E'$ we get that $f(\varphi(C'))=h(u_0) = f(\varphi(C''))$, which proves that $f(\varphi(C_+))=f(\varphi(C_-))$.

We have proven that for $(\mu_\Delta\times\mu_\Delta)$-almost every pair of chambers $(C_1,C_2)$ we have $f(C_1)=f(C_2)$. This proves that $f$ is essentially constant.

%Let  Let $(C,C')\in \Delta^2$, we have to prove that almost surely $f(C)=f(C')$. Almost surely $C$ and $C'$ are opposite and are contained in a unique apartment $A$ of $\Delta$. Let $C_1$ be the chamber of $A$ which is $+$-adjacent to $C$, $C_2$ be the chamber of $A$ which is $-$-adjacent to $C_1$ ; it is $+$-adjacent to $C'$.

%Since $f$ is almost constant on residues of type $+$, it follows that $f(C_1)$ is almost surely equal to $f(C)$\jl{Implicitely uses a disintegration- more precisely that $(\proj_u)_*\mu_o$ is the disintegration of $\mu_o$ with respect to $\Delta\to\Delta_+$}. Similarly we have almost surely $f(C_2)=f(C_1)$ and $f(C')=f(C_2)$. Therefore $f(C')=f(C)$, which proves that $f$ is essentially constant.
\end{proof}

Our goal in the rest of the section is  to prove Theorem~\ref{thm:factor}, which we recall here, in a different form (see \cite[Theorem 3.3.4]{AnantharamanPopa} for the equivalence between the two statements).

\begin{theorem}\label{thm:subalgebras}
Any $\Gamma$-invariant, weakly-* closed subalgebra of $L^\infty(\Delta)$ is either 
\begin{itemize}
\item reduced to the constants,
\item $L^\infty(\Delta_-)$,
\item $L^\infty(\Delta_+)$,
\item or $L^\infty(\Delta)$.
\end{itemize}
\end{theorem}

\subsection{Radial functions}

Let $u$ be a vertex of $\Delta$ of type $-$. Recall that we denote $\proj_u:\Delta\to\Res(u)$ the projection on the residue of $u$. It sends the measure $\mu_\Delta^o$ onto some measure $(\proj_u)_*\mu_\Delta^o$, which we will denote $\mu_u ^o$. 

\begin{lemma}\label{lem:projmp}
The measure $\mu_u^o$ is in fact the restricted prouniform measure on the set  embeddings $\Lambda\to X$ which send the wall of $\Lambda$ of type $-$ to the geodesic ray $[ou)$.
\end{lemma}

\begin{proof}
Fix $\lambda\in \Lambda$.
For $z\in V_\lambda(o)$, recall that $\Omega_o(z)$ is the set of $C\in\Delta$ such that $z\in Q(x,C)$. By definition of the prouniform measure, it suffices to prove that $\mu_u^o(\Omega_o(z))=\mu_\Delta^o(\proj_u^{-1}(\Omega_o(z))$ depends only on $\lambda$ and not on $z$ (as long as $\Omega_o(z)\cap \Res(u)\neq \varnothing$).

For $x\in X$, let $\Delta^{x,u}$ be the set of $C\in \Delta$ such that $C,o$ and $u$ are in an apartment. Note that for any $C\in \Delta$ there is an apartment containing $C$ and $u$, and therefore there exists an $x\in [ou)$ such that $C\in \Delta^{x,u}$. Therefore we can write $\Delta$ as an increasing union $\bigcup_{x\in [ou)} \Delta^{x,u}$. Hence it suffices to prove that  $\mu_\Delta^o(\proj_u^{-1}(\Omega_o(z))\cap \Delta^{x,u})$ depends only on $\lambda$, as long as $\Omega_o(z)\cap \Res(u)\neq\varnothing$.

 First we claim that $\mu_\Delta^o(\proj_u^{-1}(\Omega_o(z))\cap \Delta^{o,u})$ does not depend on $z$, assuming $\Omega_o(z)\cap \Res(u)\neq \varnothing$. To prove the claim, let us denote as before $\scrF^o$ the subset of $\scrF$ consisting of maps sending $0$ to $o$. Then the natural map $\scrF^o$ to $\Delta$, associating to an apartment the chamber at infinity which is the image of $\Lambda$, sends the measure $\mu_\scrF^o$ to $\mu_\Delta^o$, by Corollary \ref{cor:measurepreserving}.

Furthermore since $\mu_\Delta^o$-almost every $C$ is opposite $u$, the $W$-distance between $C$ and $\proj_u(C)$ is almost surely constant, say equal to $w_1$. Hence we have $C\in \Delta^{o,u} \cap \proj_u^{-1}(\Omega_o(z))$ if and only if there exists an embedding $\phi:\Sigma\to X$ such that $\phi(\ell_-)=[o,u)$ (where $\ell_-$ is the ray of type $-$ starting from $0\in\Sigma$ and opposite  $\Lambda$), $\phi(w_1\lambda)=z$ and $\phi(\Lambda)=Q(o,C)$.
Hence $$\mu_\Delta^o(\proj_u^{-1}(\Omega_o(z))\cap \Delta^{o,u})
 = \mu_\scrF^o(\{\phi:\Sigma \to X \mid  \phi(\ell_-)=[ou) \textrm{ and } z\in \phi(\Lambda) \})$$ which only depends on the isometry class of $\Conv([ou)\cup\{z\})$ by construction of the prouniform measure $\mu_\scrF$. As we assumed $\Omega_o(z)\cap \Res(u)\neq \varnothing$ we get that this convex hull depends only on $\lambda$, as it is isometric to the convex hull in $\Sigma$ of $\lambda$ and $\ell_-$.

For $\lambda,\nu\in \Sigma$ (identified to the vector space $\RR^2$), write $\nu>\lambda$ if $\nu-\lambda\in \Lambda$, and $\nu\gg \lambda$ if $\nu>w.\lambda$, for all $w\in W_0$ (where $W_0$ is the stabilizer of $0$ in $\Aut(\Sigma)$).

 Let $\nu>\lambda$. Then $\Omega_o(z)$ is a union of $\Omega_o(z')$ over all the $z'\in V_\nu(o)$ such that $z\in \Conv(o,z')$. The number of such $z'$ depends only on $\lambda$ and $\nu$. Therefore it suffices to prove that $\mu_\Delta^o(\proj_u^{-1}(\Omega_o(z'))\cap \Delta^{x,u})$ depends only on $\nu$, and we can choose $\nu$ to be arbitrarily large.

Fix $x\in [o,u)$.
Now take $\nu\gg \sigma(o,x)$, and $z$ such that $\Omega_o(z)\cap \Res(u)\neq \varnothing$. By \cite[Theorem 3.6]{Parkinson} we get that $\Omega_x(z) = \Omega_o(z)$. In particular, if $\omega\in\Omega_x(z)$, then $h_\omega(x,o)=\sigma(x,z)-\sigma(o,z)=-\sigma(o,x)$, which does not depend on $z$. Hence by Proposition \ref{prop:changebasepoint}, the Radon-Nikodym derivative $\frac{d\mu_\Delta^x}{d\mu_\Delta^o}$ is constant on $\Omega_o(z)$ and is independent of $z$.
  It follows that $\mu_\Delta^o(\proj_u^{-1}(\Omega_o(z))\cap \Delta^{x,u})$ is equal, up to a constant, to $\mu_\Delta^x(\proj_u^{-1}(\Omega_x(z))\cap \Delta^{x,u})$
which is independent of $\lambda$ by our previous argument. 

\end{proof}

\begin{definition}
For every $u\in \Delta_+$, we denote $R_u = \proj_u^*(L^\infty(\Res(u),\mu_u^o))\subset L^\infty(\Delta,\mu_\Delta^o)$. Functions in $R_u$ are called 
\emph{$u$-radial}. In other words, a function 
 $f$ is $u$-radial if $f(x)=f(x')$ for every $x,x'$ such that $\proj_u(x)=\proj_u(x')$.

%We denote by $R_u\subset L^\infty(\Delta)$ the set of $u$-radial functions.
\end{definition}

If $f$ is radial, then we see that $f$ is constant on every residue of type $-$ except $u$. Since $\mu(\Res(u))=0$, it follows that $f$ is also a function in $L^\infty(\Delta_-)$. In other words, we have $R_u\subset L^\infty(\Delta_-)$. Note that by construction, the projection $\proj_u$ induces an isomorphism from $R_u$ to $L^\infty(\Res(u),\mu_u^o)$.

\begin{proposition}\label{rugenerate}
Let $E$ be a $\Gamma$-invariant subset of $\Delta_-$. Then the weak-* closure of the algebra generated by all the algebras $R_u$, where $u\in E$, is $L^\infty(\Delta_-)$.
\end{proposition}

\begin{proof}
Let $A$ be the algebra generated by all the algebras $R_u$, where $u\in E$. Since $\gamma R_u=R_{\gamma u}$, the algebra $A$ is $\Gamma$-invariant. 

We claim first that $A$ contains a non-constant function $f_0$ which is continuous on $\Delta_-$. Granting the claim, we see by Corollary \ref{cor:FactorTop} that $A$ contains $C(\Delta_-)$, which itself is dense in $L^\infty(\Delta_-)$ for the weak-* topology.

Now let us prove the claim. Consider a continuous function $f$ on $\Res(u)$, and define $f':\Delta\to\CC$ by $f'(x)=f(\proj_u(x))$. Obviously, $f'$ is a function in $R_u$. Furthermore, we see that $f'$ is continuous at every point of $\Delta_-$ except $u$.

Since $E$ is $\Gamma$-invariant, and the action of $\Gamma$ is non-elementary, there are three pairwise distinct points $u_1, u_2, u_3\in E$. Let $f_1$ be a continuous function on $\Res(u_1)$ which is zero on a neighborhood of $\proj_{u_1}(u_2)$, but non zero on a neighborhood of $\proj_{u_1}(u_3)$, and consider $f'_1\in R_{u_1}$ as above. Then $f'_1$ is a function which is zero on a neighborhood of $u_1$, and which is continuous everywhere except at $u_2$. Now let $f_2$ be a continuous function on $\Res(u_2)$ which is zero on a neighborhood of $\proj_{u_2}(u_1)$ and non-zero on a neighborhood of $\proj_{u_2}(u_3)$, and consider its lift $f'_2$. 

Let $f=f'_1 f'_2$. Then $f$ is continuous everywhere except maybe on a neighborhood of $u_1$ and $u_2$. But it is also zero on a neighborhood of $u_1$ and $u_2$. So in fact $f$ is continuous everywhere. Furthermore $f(u_3)\neq 0$, so that $f$ is non constant, proving the claim.

\end{proof}

Let $A$ be a $\Gamma$-invariant closed subalgebra of $L^\infty(\Delta)$. The following Proposition is the heart of the proof of Theorem \ref{thm:subalgebras}.

\begin{proposition}\label{prop:radial}
Assume that $A$ contains a function which is not in $L^\infty(\Delta_+)$. Then $A$ contains $L^\infty(\Delta_-)$.
\end{proposition}

\begin{proof}[Proof of Theorem \ref{thm:subalgebras} from Proposition \ref{prop:radial}]

Assume that $A$ contains a non-constant function. Then by Lemma \ref{lem:meet}, there is some function in $A$ which is either not in $L^\infty(\Delta_+)$ or not in $L^\infty(\Delta_-)$. Assume the former (the other case being similar). Then Proposition \ref{prop:radial} implies that $A$ contains $L^\infty(\Delta_-)$. Now either $A=L^\infty(\Delta_-)$, and we are done, or $A$ contains another function which is not in $L^\infty(\Delta_-)$. In that case a similar argument proves that $A$ contains also $L^\infty(\Delta_+)$. Hence by Lemma \ref{lem:join} it follows that $A$ is equal to $L^\infty(\Delta)$.
\end{proof}

In order to understand subalgebras of $R_u$, recall the definition of the group $\Pi$ from Section \ref{sub:proj}. 
Note that the action of the group $\Pi$ on $\Res(u)$ defines an action on the algebra $R_u$, via the isomorphism $R_u\simeq L^\infty(\Res(u),\mu_u^o)$. As we have seen in Lemma~\ref{lem:3transitive}, $\Pi$ is 3-transitive on $\Res(u)$.

\begin{lemma}\label{lem:SubalgebraRu}
Any $\Pi$-invariant, weakly-* closed subalgebra of $R_u$ is either $R_u$ or reduced to the constant functions.
\end{lemma}

\begin{proof}
Since $\Pi$ is transitive on $\Res(u)$, we can identify $\Res(u)$ to $\Pi/L$, where $L$ is the stabilizer of a point. Since $\Pi$ acts $3-$transitively on $\Res(u)$, $L$ is a maximal subgroup of $\Pi$.

Now any measurable $\Pi$-factor of $\Pi/L$ is of the form $\Pi/L'$ where $L<L'<\Pi$. It follows that either $L'=L$ or $L'=\Pi$.
\end{proof}

We also define the map $T_u^v:\Delta\to \Res(v)$ by $T_u^v=\proj_v\circ \proj_u$. Recall from Corollary \ref{cor:convTuv} that for every $f\in L^\infty(\Delta)$, for almost every $(u,v)$  there exists a sequence $\gamma_n$ such that $\gamma_n f$ converges to $f\circ T_u^v$ (in weak-* topology), which is $v$-radial.

\begin{proof}[Proof of Proposition \ref{prop:radial}]
Recall that $A$ is an algebra which contains a non-constant function $f$  in $L^\infty(\Delta_+)$. Since $A$ is closed and $\Gamma$-invariant, it contains almost all the functions $f\circ T_u^v$ by Corollary \ref{cor:convTuv}. Now, since $f$ is non-constant, there is a positive measure set of $v\in \Delta_+$ such that $f_{|\Res(v)}$ is not $\mu_v^o$-essentially constant on $Res(v)$ (as $\mu_v^o$ is the disintegration of $\mu_\Delta^o$ by Lemma \ref{lem:desintegration_prouniform}). Fix $u,v$ such that $v$ is in this set and such that the function $g=f\circ T_u^v$ is in $A$. Since the map $\proj_u:\Delta\to \Res(u)$ is measure-preserving by Lemma~\ref{lem:projmp}, and the map $\proj_v:(\Res(u),\mu_u^o)\to (\Res(v),\mu_v^o)$ is also measure-preserving (as it is induced by an automorphism of the wall-tree $T_{u,v}$, and using Corollary \ref{cor:measurepreserving}), we see that $g$ is also non-constant. Furthemore it is clear that if $\proj_u(C)=\proj_u(C')$ then $g(C)=g(C')$, thus $g\in R_u$.
 Furthermore, using Theorem \ref{thm:convTuvproj}, we see that $A\cap R_u$ contains the algebra generated  by the $\Pi$-orbit of $g$. By Lemma \ref{lem:SubalgebraRu} it follows that $A\cap R_u=A$. Since $A$ is is $\Gamma$-invariant, it also contains $R_{\gamma u}$ for every $\gamma\in \Gamma$. We conclude with Proposition \ref{rugenerate}.

%Pick a countable, dense subset $P_0$ of $P$. Then by Theorem \ref{thm:convTuvproj}, for almost every $u$, we have that  every $f\in R_u\cap A$ implies that $p f\in A$. Hence $A\cap R_u$ is stable under a dense subgroup of $P$, hence since it is closed by the full group $P$. 
\end{proof}

\subsection{Conclusion}

\begin{lemma}\label{lem:fixedpointae}
Let $g\in \Aut(X)$ which fixes almost every point of $\Delta_-$ (resp. $\Delta_+$). Then $g$ is the identity.
\end{lemma}

\begin{proof}
Let $x_1,x_2\in \Delta_-$ be fixed by $g$. Let $\ell_1\in \Delta_+$ be the vertex of $\Delta$ adjacent to both, it is also fixed by $g$. The set of all $x\in \Delta_-$ opposite $\ell_1$ is of full measure in $\Delta_-$, so there is an $x_3$ opposite $\ell_1$ which is fixed by $g$. In particular $g$ fixes pointwise the appartment of $\Delta$ containing $x_1,x_2$ and $x_3$. Therefore $g$ fixes almost every element of $\Delta$, and in fact the boundary of almost every apartment.

It is clear that $g$ cannot be an hyperbolic isometry of $X$. Therefore, if $A$ is an apartment whose boundary is fixed by $g$, then $g$ cannot act as a nontrivial translation of $A$. Therefore it fixes $A$ pointwise. 

We have proven that almost every apartment of $X$ is fixed pointwise by $g$. Let $x\in X$. The set of apartments containing $x$ is of positive measure, therefore it contains some apartment which is fixed pointwise by $g$. Hence $x$ is fixed by $g$. Therefore $g$ acts trivially on $X$.
\end{proof}

Finally we are able to prove the Normal Subgroup Theorem (Theorem~\ref{mainthm}):

\begin{theorem}
Let $\Gamma$ be a uniform lattice of an $\tilde A_2$ building, and $N$ be a normal subgroup of $\Gamma$. Then either $N=\{1\}$ or $N$ is of finite index in $\Gamma$.
\end{theorem}

\begin{proof}
By \cite{CartwrightMlotkowski}, the group $\Gamma$, hence the quotient $\Gamma/N$, has property (T). Therefore, it suffices to prove that if $N$ is non-trivial then $\Gamma/N$ is amenable, or equivalently that if $\Gamma/N$ is not amenable then $N=\{1\}$.

So assume that $\Gamma/N$ is non-amenable. It follows that there exists a compact convex $\Gamma/N$-space $K$ with no fixed point. We can see this space as a $\Gamma$-space on which $N$ acts trivially.

Now by \cite{BoundaryAmenability} the action of $\Gamma$ on $\Delta$ is amenable. This implies that there is a $\Gamma$-equivariant map $\varphi:\Delta\to K$. We equip $K$ with the pushforward measure $\phi_+\mu_\Delta^o$. Let $A=\varphi^*(L^\infty(K))$: this is a $\Gamma$-invariant subalgebra of $L^\infty(\Delta)$, on which $N$ acts trivially. Since $\Gamma$ has no fixed point on $K$, the algebra $A$ is not reduced to the constant functions, hence by Theorem \ref{thm:subalgebras} it is either $L^\infty(\Delta_-)$, $L^\infty(\Delta_+)$ or $L^\infty(\Delta)$. But since $N$ acts trivially on $K$, it also acts trivially on $A$. Therefore for every $g\in N$, $g$ fixes almost every point of either $\Delta_-$, $\Delta_+$ (or even almost every point of $\Delta$). It follows that $N=\{1\}$ by Lemma \ref{lem:fixedpointae}.
\end{proof}

\indexprologue{The following is a list of notations which are used multiple times, starting on section \S\ref{sec:buildings}.
Notations which only appear locally are not listed here.}
\printindex
\bibliographystyle{alpha}
\bibliography{biblio}

% \bib, bibdiv, biblist are defined by the amsrefs package.
\begin{bibdiv}
\begin{biblist}

\bib{AbramenkoBrown}{book}{
      author={Abramenko, Peter},
      author={Brown, Kenneth~S.},
       title={Buildings},
      series={Graduate Texts in Mathematics},
   publisher={Springer, New York},
        date={2008},
      volume={248},
        note={Theory and applications},
}

\bib{AnantharamanPopa}{unpublished}{
      author={Anantharaman, Claire},
      author={Popa, Sorin},
       title={An introduction to {$II_1$} factors},
         url={https://www.math.ucla.edu/~popa/Books/IIunV15.pdf},
}

\bib{BCL}{article}{
      author={Bader, Uri},
      author={Caprace, Pierre-Emmanuel},
      author={L\'{e}cureux, Jean},
       title={On the linearity of lattices in affine buildings and ergodicity
  of the singular {C}artan flow},
        date={2019},
        ISSN={0894-0347},
     journal={J. Amer. Math. Soc.},
      volume={32},
      number={2},
       pages={491\ndash 562},
         url={https://doi-org.revues.math.u-psud.fr/10.1090/jams/914},
      review={\MR{3904159}},
}

\bib{BaderShalom}{article}{
      author={Bader, Uri},
      author={Shalom, Yehuda},
       title={Factor and normal subgroup theorems for lattices in products of
  groups},
        date={2006},
        ISSN={0020-9910,1432-1297},
     journal={Invent. Math.},
      volume={163},
      number={2},
       pages={415\ndash 454},
         url={https://doi.org/10.1007/s00222-005-0469-5},
      review={\MR{2207022}},
}

\bib{BorelTits}{article}{
      author={Borel, Armand},
      author={Tits, Jacques},
       title={Homomorphismes ``abstraits'' de groupes alg\'ebriques simples},
        date={1973},
     journal={Ann. of Math. (2)},
      volume={97},
       pages={499\ndash 571},
}

\bib{BurgerMozes_lattices}{article}{
      author={Burger, Marc},
      author={Mozes, Shahar},
       title={Lattices in product of trees},
        date={2000},
     journal={Inst. Hautes \'Etudes Sci. Publ. Math.},
      number={92},
       pages={151\ndash 194 (2001)},
}

\bib{Busch}{misc}{
      author={Busch, Sira},
      author={Schillewaert, Jeroen},
      author={Maldeghem, Hendrik~Van},
       title={Groups of projectivities and levi subgroups in spherical
  buildings of simply laced type},
        date={2024},
         url={https://arxiv.org/abs/2407.09226},
}

\bib{CameronBook}{book}{
      author={Cameron, Peter~J.},
       title={Projective and polar spaces},
      series={QMW Maths Notes},
   publisher={Queen Mary and Westfield College, School of Mathematical
  Sciences, London},
        date={1992},
      volume={13},
}

\bib{CapraceRemy}{article}{
      author={Caprace, Pierre-Emmanuel},
      author={R\'emy, Bertrand},
       title={Simplicity and superrigidity of twin building lattices},
        date={2009},
        ISSN={0020-9910,1432-1297},
     journal={Invent. Math.},
      volume={176},
      number={1},
       pages={169\ndash 221},
         url={https://doi.org/10.1007/s00222-008-0162-6},
      review={\MR{2485882}},
}

\bib{CartMlotkT}{article}{
      author={Cartwright, D.~I.},
      author={M\l~otkowski, W.},
      author={Steger, T.},
       title={Property ({T}) and {$\tilde{A_2}$} groups},
        date={1994},
        ISSN={0373-0956},
     journal={Ann. Inst. Fourier (Grenoble)},
      volume={44},
      number={1},
       pages={213\ndash 248},
         url={http://www.numdam.org/item?id=AIF_1994__44_1_213_0},
      review={\MR{1262886}},
}

\bib{CartwrightMlotkowski}{article}{
      author={Cartwright, Donald~I.},
      author={M{\l}otkowski, Wojciech},
       title={Harmonic analysis for groups acting on triangle buildings},
        date={1994},
     journal={J. Austral. Math. Soc. Ser. A},
      volume={56},
      number={3},
       pages={345\ndash 383},
}

\bib{Folland}{book}{
      author={Folland, Gerald~B.},
       title={Real analysis},
     edition={Second},
      series={Pure and Applied Mathematics (New York)},
   publisher={John Wiley \& Sons, Inc., New York},
        date={1999},
        ISBN={0-471-31716-0},
        note={Modern techniques and their applications, A Wiley-Interscience
  Publication},
      review={\MR{1681462}},
}

\bib{Knarr}{article}{
      author={Knarr, Norbert},
       title={Projectivities of generalized polygons},
    language={English},
        date={1988},
        ISSN={0381-7032},
     journal={Ars Comb.},
      volume={25B},
       pages={265\ndash 275},
}

\bib{KramerWeiss}{article}{
      author={Kramer, Linus},
      author={Weiss, Richard~M.},
       title={Coarse equivalences of {Euclidean} buildings. ({With} an appendix
  by {Jeroen} {Schillewaert} and {Koen} {Struyve}.)},
    language={English},
        date={2014},
        ISSN={0001-8708},
     journal={Adv. Math.},
      volume={253},
       pages={1\ndash 49},
}

\bib{BoundaryAmenability}{article}{
      author={L{\'e}cureux, Jean},
       title={Amenability of actions on the boundary of a building},
        date={2010},
     journal={Int. Math. Res. Not. IMRN},
      number={17},
       pages={3265\ndash 3302},
}

\bib{StrongT}{article}{
      author={{L{\'e}cureux}, Jean},
      author={{de la Salle}, Mikael},
      author={{Witzel}, Stefan},
       title={{Strong Property (T), weak amenability and $\ell^p$-cohomology in
  $\tilde{A}_2$-buildings}},
        date={2024},
     journal={Ann. Sci. \'Ec. Norm. Sup\'er. (4)},
      volume={57},
      number={5},
       pages={1371\ndash 1444},
}

\bib{Margulis}{book}{
      author={Margulis, G.~A.},
       title={Discrete subgroups of semisimple {L}ie groups},
      series={Ergebnisse der Mathematik und ihrer Grenzgebiete (3) [Results in
  Mathematics and Related Areas (3)]},
   publisher={Springer-Verlag, Berlin},
        date={1991},
      volume={17},
}

\bib{MiteWitzel}{misc}{
      author={Mite, Thomas~Titz},
      author={Witzel, Stefan},
       title={A {C2}-tilde-lattice that is not residually finite},
        date={2023},
}

\bib{oppenheim}{misc}{
      author={Oppenheim, Izhar},
       title={Property ({T}) for groups acting on affine buildings},
        date={2024},
         url={https://arxiv.org/abs/2410.05716},
}

\bib{Pansu}{article}{
      author={Pansu, Pierre},
       title={Formules de {M}atsushima, de {G}arland et propri\'et\'e ({T})
  pour des groupes agissant sur des espaces sym\'etriques ou des immeubles},
        date={1998},
     journal={Bull. Soc. Math. France},
      volume={126},
      number={1},
       pages={107\ndash 139},
}

\bib{Parkinson}{article}{
      author={Parkinson, James},
       title={Spherical harmonic analysis on affine buildings},
        date={2006},
     journal={Math. Z.},
      volume={253},
      number={3},
       pages={571\ndash 606},
}

\bib{Peterson}{article}{
      author={Peterson, Jesse},
       title={Character rigidity for lattices in higher-rank groups},
        date={Preprint},
        note={Avalaible at
  \url{https://math.vanderbilt.edu/peters10/rigidity.pdf}},
}

\bib{Radu}{article}{
      author={Radu, Nicolas},
       title={A homogeneous {$\tilde A_2$}-building with a non-discrete
  automorphism group is {B}ruhat-{T}its},
        date={2019},
        ISSN={0046-5755},
     journal={Geom. Dedicata},
      volume={199},
       pages={1\ndash 26},
  url={https://doi-org.ezproxy.universite-paris-saclay.fr/10.1007/s10711-018-0338-1},
      review={\MR{3928791}},
}

\bib{Ronan_triangle}{article}{
      author={Ronan, M.~A.},
       title={Triangle geometries},
        date={1984},
     journal={J. Combin. Theory Ser. A},
      volume={37},
      number={3},
       pages={294\ndash 319},
}

\bib{RousseauBook}{book}{
      author={Rousseau, Guy},
       title={Euclidean buildings. {Geometry} and group actions},
    language={English},
      series={EMS Tracts Math.},
   publisher={Berlin: European Mathematical Society (EMS)},
        date={2023},
      volume={35},
        ISBN={978-3-98547-039-6; 978-3-98547-539-1},
}

\bib{RemyTrojan}{misc}{
      author={Rémy, Bertrand},
      author={Trojan, Bartosz},
       title={Martin compactifications of affine buildings},
   publisher={arXiv},
        date={2021},
         url={https://arxiv.org/abs/2105.14807},
}

\bib{StuckZimmer}{article}{
      author={Stuck, Garrett},
      author={Zimmer, Robert~J.},
       title={Stabilizers for ergodic actions of higher rank semisimple
  groups},
        date={1994},
        ISSN={0003-486X},
     journal={Ann. of Math. (2)},
      volume={139},
      number={3},
       pages={723\ndash 747},
  url={https://doi-org.ezproxy.universite-paris-saclay.fr/10.2307/2118577},
      review={\MR{1283875}},
}

\bib{Tits_cours84}{article}{
      author={Tits, Jacques},
       title={Th\'eorie des groupes},
        date={1983/84},
     journal={Ann. Coll\`ege France},
      volume={84},
       pages={85\ndash 96},
}

\bib{Tits_Andrews}{inproceedings}{
      author={Tits, Jacques},
       title={Buildings and group amalgamations},
        date={1986},
   booktitle={Proceedings of groups---{S}t.\ {A}ndrews 1985},
      series={London Math. Soc. Lecture Note Ser.},
      volume={121},
   publisher={Cambridge Univ. Press, Cambridge},
       pages={110\ndash 127},
}

\bib{Venka}{article}{
      author={Venkataramana, T.~N.},
       title={On superrigidity and arithmeticity of lattices in semisimple
  groups over local fields of arbitrary characteristic},
        date={1988},
        ISSN={0020-9910},
     journal={Invent. Math.},
      volume={92},
      number={2},
       pages={255\ndash 306},
         url={https://doi.org/10.1007/BF01404454},
      review={\MR{936083}},
}

\bib{WeissBook}{book}{
      author={Weiss, Richard~M.},
       title={The structure of affine buildings},
      series={Annals of Mathematics Studies},
   publisher={Princeton University Press, Princeton, NJ},
        date={2009},
      volume={168},
}

\bib{Witzel}{article}{
      author={Witzel, Stefan},
       title={On panel-regular {$\tilde{A}_2$} lattices},
        date={2017},
        ISSN={0046-5755},
     journal={Geom. Dedicata},
      volume={191},
       pages={85\ndash 135},
         url={https://doi-org.revues.math.u-psud.fr/10.1007/s10711-017-0247-8},
      review={\MR{3719076}},
}

\bib{Zuk}{article}{
      author={\.{Z}uk, Andrzej},
       title={La propri\'{e}t\'{e} ({T}) de {K}azhdan pour les groupes agissant
  sur les poly\`edres},
        date={1996},
        ISSN={0764-4442},
     journal={C. R. Acad. Sci. Paris S\'{e}r. I Math.},
      volume={323},
      number={5},
       pages={453\ndash 458},
      review={\MR{1408975}},
}

\end{biblist}
\end{bibdiv}

\end{document}